\documentclass[oneside,11pt]{article}

\usepackage{amsmath}
\usepackage{amssymb}
\usepackage{graphicx}
\usepackage{eucal}
\usepackage{mathrsfs}
\usepackage{theorem}
\usepackage{pifont}
\usepackage{color}
\usepackage{framed}
\usepackage{enumitem}
\usepackage{hyperref}
\usepackage{dsfont}
\usepackage{float}
\usepackage[normalem]{ulem}
\usepackage[margin=2.5cm]{geometry}
\usepackage[nottoc]{tocbibind}

\definecolor{shadecolor}{rgb}{0.8,0.8,0.8}

\usepackage{xcolor} 
\definecolor{ocre}{RGB}{243,102,25} 
\definecolor{darkocre}{RGB}{121,51,12} 
\definecolor{lightocre}{RGB}{255,150,37} 
\definecolor{verylightocre}{RGB}{255,204,50} 
\definecolor{soccerfield}{RGB}{107,142,35} 
\definecolor{lightgray}{RGB}{200,200,200} 
\definecolor{warmblue}{RGB}{51,102,153} 
\definecolor{lightwarmblue}{RGB}{105,141,198} 
\definecolor{sepia}{RGB}{112,66,20}

\usepackage[all]{xy}
\usepackage{tikz} 
\usepackage{tkz-euclide}
\usepackage{pgfplots}

\newcommand{\btkz}{\begin{tikzpicture}}
\newcommand{\etkz}{\end{tikzpicture}}

\newtheorem{theorem}{Theorem}[section]
\newtheorem{lemma}[theorem]{Lemma}
\newtheorem{proposition}[theorem]{Proposition}
\newtheorem{corollary}[theorem]{Corollary}
\newtheorem{definition}[theorem]{Definition}
\theorembodyfont{\rmfamily}
\newtheorem{example}[theorem]{Example}

\theorembodyfont{\rmfamily}
\newenvironment{proof}{{\flushleft \emph{Proof}:}}{\hfill\ding{110}}
\newenvironment{remark}{{\flushleft \bfseries Remark:}}{}
\newenvironment{remarks}{{\flushleft \bfseries Remarks:}}{}
\newenvironment{comment}{{\flushleft \bfseries Comment:}}{}

\usetikzlibrary{decorations.pathmorphing}
\usetikzlibrary{snakes}

\tikzset{snake it/.style={decorate, decoration=snake}}

\newcommand{\M}{M}

\newcommand{\N}{N}
\newcommand{\g}{g}
\newcommand{\h}{h}
\renewcommand{\P}{\mathcal{P}}

\newcommand{\MEps}{\M_{\e}}
\newcommand{\gEps}{\g_{\e}}
\newcommand{\PEps}{\P _{\e}}
\newcommand{\ZEps}{Z_ {\e}}

\newcommand{\euc}{\mathrm{e}}
\renewcommand{\r}{\mathfrak{r}}
\renewcommand{\b}{\mathfrak{b}}
\renewcommand{\v}{{\bf v}}
\renewcommand{\u}{{\bf u}}
\newcommand{\bxi}{\boldsymbol{\xi}}

\newcommand{\hM}{\hat{\M}}
\newcommand{\hg}{\hat{\g}}
\newcommand{\hP}{\hat{\P}}

\newcommand{\hZ}{\hat{Z}}
\newcommand{\hI}{\hat{I}}
\newcommand{\hE}{\hat{E}}

\newcommand{\Elin}{E^{\text{quad}}}
\newcommand{\Ilin}{I^{\text{quad}}}

\newcommand{\W}{{\Omega}}
\newcommand{\Vol}{\textup{Vol}}
\newcommand{\dVol}{\textup{dVol}}
\newcommand{\Volume}{\dVol_\P}
\newcommand{\hVolumeV}{\dVol_{\hP_\v}}
\newcommand{\hVolumeEV}{\dVol_{\hP_{\e\v}}}
\newcommand{\hVolumeEVE}{\dVol_{\hP_{\e\v_\e}}}
\newcommand{\VolumeE}{dx}
\newcommand{\VolumeEps}{\dVol_{\PEps}}

\newcommand{\cof}[1]{\vartheta^{#1}}
\newcommand{\weakly}{\rightharpoonup}
\newcommand{\weakstar}{\stackrel{*}{\weakly}}

\newcommand{\calW}{\mathcal{W}}
\newcommand{\bbW}{\mathbb{W}}
\newcommand{\calE}{\mathcal{E}}
\newcommand{\calM}{\mathcal{M}}
\newcommand{\frakX}{\mathfrak{X}}
\newcommand{\bbS}{\mathbb{S}}
\newcommand{\frakn}{\mathfrak{n}}
\newcommand{\bbN}{\mathbb{N}}

\newcommand{\LimEps}{\lim_{\e\to0}}
\newcommand{\LiminfEps}{\liminf_{\e\to0}}
\newcommand{\LimsupEps}{\limsup_{\e\to0}}

\newcommand{\curl}{\operatorname{curl}}
\renewcommand{\div}{\operatorname{div}}
\newcommand{\tr}{\operatorname{tr}}
\newcommand{\TR}{\tr_\euc}

\newcommand{\IdRtwo}{\id_{\R^2}}

\newcommand{\ind}{\mathds{1}}
\newcommand{\pl}{\partial}
\newcommand{\dM}{\pl\M}

\newcommand{\tf}{\tilde{f}}
\newcommand{\tmu}{\tilde{\mu}}
\newcommand{\teta}{\tilde{\eta}}

\newcommand{\kmax}{k_{\max}}
\newcommand{\LogEps}{\log(1/\e)}

\newcommand{\ip}[1]{\langle #1 \rangle}

\newcommand{\HOne}{H^1}
\newcommand{\HOneZero}{H^1_0}
\newcommand{\HminusOne}{H^{-1}}

\newcommand{\SEF}{\Sigma_{\mathbb{S}}}

\newcommand{\torsion}{\mathbb{T}}

\newcommand{\Emph}[1]{{\bfseries #1}}

\newcommand{\brk}[1]{\left(#1\right)}          
\newcommand{\BRK}[1]{\left\{#1\right\}}        

\newcommand{\Abs}[1]{\left| #1 \right|}        
\newcommand{\Norm}[1]{\left\| #1 \right\|}     

\newcommand{\Cases}[1]{\begin{cases} #1 \end{cases}}

\newcommand{\secref}[1]{Section~\ref{#1}}
\newcommand{\figref}[1]{Figure~\ref{#1}}
\newcommand{\thmref}[1]{Theorem~\ref{#1}}
\newcommand{\defref}[1]{Definition~\ref{#1}}

\newcommand{\propref}[1]{Proposition~\ref{#1}}
\newcommand{\lemref}[1]{Lemma~\ref{#1}}
\newcommand{\corrref}[1]{Corollary~\ref{#1}}

\newcommand{\beq}{\begin{equation}}
\newcommand{\eeq}{\end{equation}}
\newcommand{\bsplit}{\begin{split}}
\newcommand{\esplit}{\end{split}}

\providecommand{\e}{\varepsilon}
\providecommand{\half}{\frac{1}{2}}

\providecommand{\R}{\mathbb{R}}

\newcommand{\textand}{\quad\text{ and }\quad}
\newcommand{\Textand}{\qquad\text{ and }\qquad}
\providecommand{\vp}{\varphi}

\newcommand{\SO}{{\operatorname{SO}}}
\newcommand{\dist}{{\operatorname{dist}}}
\newcommand{\Hom}{{\operatorname{Hom}}}
\newcommand{\id}{{\operatorname{Id}}}

\numberwithin{equation}{section}

\begin{document}

\title{From Volterra dislocations to strain-gradient plasticity}
\author{Raz Kupferman* and Cy Maor\footnote{Einstein institute of Mathematics, Hebrew University of Jerusalem.}}
\date{}
\maketitle

\begin{abstract}
We rigorously derive a strain-gradient model of plasticity as a $\Gamma$-limit of continuum bodies containing finitely-many edge-dislocations (in two dimensions).
The key difference from previous such derivations is the elemental notion of a dislocation:
we work in a continuum framework in which the lattice structure is represented by a smooth frame field, and the presence of a dislocation manifests in a circulation condition on that frame field; the resulting model is a Lagrangian approach with a multiplicative strain decomposition.
The multiplicative nature of the geometric incompatibility generates many technical challenges,
which require a systematic study of the geometry of bodies containing multiple dislocations, the definition of new notions of convergence, and the derivation of new geometric rigidity estimates pertinent to dislocated bodies.
Our approach places the strain-gradient limit in a unified framework with other models of dislocations, which cannot be addressed within the ``admissible strain'' approach used in previous works.
\end{abstract}

\tableofcontents

\section{Introduction}

Dislocations are among the most important and well-studied defects in crystals.
They were studied in early 20th century by Volterra, who examined elastic equilibria of multiply-connected elastic bodies, obtained from ``stress-free'' bodies by cut-and-weld procedures mimicking plastic deformations.
In the 1930s, this theory was used to explain plastic deformation in crystalline materials \cite{Tay34, Oro34, Pol34}.
Ever since, there has been growing literature on bodies with finitely-many dislocations, collective behavior of clusters of dislocations, and on a larger scale, classical plasticity. 
Within this literature, two main threads are central to this work: 

The mechanical literature, starting from the early 1950s, has addressed kinematics issues, such as the effective fields describing materials symmetries, with and without dislocations, notions of material uniformity, and  the incorporation of these notions into mechanical models, i.e., into constitutive relations \cite{Nol58,Nol67,Wan67}.
Its elemental object is a body manifold, which is a continuum representation of a material structure.
This thread, which for the sake of establishing a nomenclature we will call the \emph{rational-mechanics approach}, often uses a geometric language, and has traditionally put less focus on the rigorous derivation of effective models from more elemental ones. 

The other thread concerns the rigorous derivation of models for the collective behavior of dislocations from models of finitely-many ones (e.g., \cite{GLP10,DGP12,MSZ14,MSZ15,CGO15,CGM16,Gin19,Gin19b,KO20,CGM23}).
In this thread, the elemental model of a single dislocation does not depart from the same premises as the rational mechanics approach. Moreover, all the aforementioned literature addresses a low-energetic regime. For the sake of nomenclature, we will call this approach the \emph{admissible strain approach}---the terminology will be clarified below.

In earlier work, we initiated a program of combining both threads, so far in a high-energetic regime  \cite{KM15,KM16,EKM20}. The present work is a continuation of this program, and its scope is twofold:
\begin{itemize}[itemsep=0pt]
\item[-] The derivation of low-energetic plasticity models from elemental models that are consonant with the rational mechanics approach. That is, our basic model is a Lagrangian model with multiplicative decomposition of the strain gradient.
\item[-] Lay a setting for a unified and rigorous study of various problems involving a wider range of material defects.
\end{itemize}

As for the second point, our motivation in studying dislocations based on notions of uniformity and 
symmetry, is the flexibility that this approach offers in addressing a wide variety of mechanical problems involving various types of defects. There are several situations in which it is not clear how to apply the  admissible strain approach. For example:
\begin{enumerate}[itemsep=0pt,label=(\alph*)]
\item High energetic regimes, in which the accumulation of dislocations induces a substantial change in the intrinsic geometry of the material.
\item Defects in slender bodies, such as graphene monolayers.
\item Defects of mixed types, such as coexistence of disclinations, dislocations and point defects.
\end{enumerate}
Item (a) has been partly addressed in \cite{KM15,KM16,EKM20}. In this work we address low-energetic limits that were studied within the admissible strain approach, using a framework enabling the future study of problems such as Items (b) and (c). 

In this introduction, we describe our model  of bodies containing finitely-many dislocations, comparing it  with the admissible strain approach, describe the strain-gradient plasticity model, survey some of the relevant literature, and present the results and structure of this paper.
We focus on two-dimensional systems, and comment only briefly on extensions to three-dimensional ones.

\paragraph{Modeling a dislocation}
A dislocation in a crystal is created by a gliding mechanism along a lattice direction; atomic bonds are broken and new bonds are formed after indentation.
Once this gliding has taken place, one has a medium having a perfect local lattice structure (except at a core), which nevertheless does not embed in a global lattice structure.
In a continuum theory, the lattice structure is replaced by smooth fields. A defect-free crystal is modeled as a Euclidean domain, possibly endowed with a frame field representing the lattice directions. 
As described by Volterra, a dislocation can be simulated by perforating the continuum by a cylindrical hole of atomic-size diameter (the core of the defect), cutting the domain across a half-plane terminating at the cylindrical hole,
translating one of the sides of the cut along one of the lattice directions, and then gluing the two sides of the cut (see Figure~\ref{fig:intro} for the discrete and continuum modeling of an edge-dislocation). 

\begin{figure}
\begin{center}
\btkz
	\foreach \i in {0,1,2,3,4,5}
	{
		\tkzDefPoint(0.5*\i,0){A};
		\tkzDrawPoint[color=blue](A);
		\tkzDefPoint(0.5*\i,0.5){A};
		\tkzDrawPoint[color=blue](A);
		\tkzDefPoint(0.5*\i,1){A};
		\tkzDrawPoint[color=blue](A);
	}	
	\foreach \i in {0,1,2,3,4}
	{
		\tkzDefPoint(0.62*\i,1.5){A};
		\tkzDrawPoint[color=blue](A);
		\tkzDefPoint(0.62*\i,2.0){A};
		\tkzDrawPoint[color=blue](A);
	}	
	\draw[dashed, color=brown] (1.25,1.20) circle (12pt);
	\draw[->,color=magenta] (0.5,0.5) -- (1.0,0.5);
	\draw[->,color=magenta] (1.0,0.5) -- (1.5,0.5);
	\draw[->,thick,color=cyan] (1.5,0.5) -- (2.0,0.5);
	\draw[->,color=magenta] (2.0,0.5) -- (2.0,1.0);
	\draw[->,color=magenta] (2.0,1.0) -- (1.85,1.5);
	\draw[->,color=magenta] (1.85,1.5) -- (1.85,2.0);
	\draw[->,color=magenta] (1.85,2.0) -- (1.25,2.0);
	\draw[->,color=magenta] (1.25,2.0) -- (0.65,2.0);
	\draw[->,color=magenta] (0.65,2.0) -- (0.65,1.5);
	\draw[->,color=magenta] (0.65,1.5) -- (0.5,1.0);
	\draw[->,color=magenta] (0.5,1.0) -- (0.5,0.5);
	\draw[->,thick,color=red] (0.,0.) -- (0.5,0.0);
	\draw[->,thick,color=red] (0.,0.) -- (0.0,0.5);
	\tkzText(1.25,2.75){Lattice model}
	\tkzText[color=red](0.0,-0.5){lattice}
	\tkzText[color=red](0.0,-0.85){directions}
	\tkzText[color=cyan](2.5,-0.5){Burgers}
	\tkzText[color=cyan](2.5,-0.85){vector $\v$}
	\tkzText[color=magenta](1.25,-1.5){Burgers}
	\tkzText[color=magenta](1.25,-1.85){circuit}
	\tkzText[color=brown](-1.0,1.5){defect}
	\tkzText[color=brown](-1.0,1.15){core}
	\draw[->,dashed,color=brown] (1.25,1.25) -- (-0.6,1.22);
	\draw[->,dashed,color=magenta] (1.25,0.4) -- (1.25,-1.2);
	\draw[->,dashed,color=cyan] (1.75,0.4) -- (2.50,-0.3);
	
	\begin{scope}[xshift=6cm]
		\tkzText(2.8,2.75){Continuum model (Volterra)}
		\tkzText[color=magenta](3.0,-1.5){defect is encoded in the}
		\tkzText[color=magenta](3.0,-1.85){frame field}
		\draw[color=blue] (0,-0.1) -- (2.0,-0.1) -- (2.0,2.1) -- (0,2.1) -- cycle;
		\draw[color=blue] (0.8,0.8) -- (1.2,0.8) -- (1.2,1.2) -- (0.8,1.2) -- cycle;
		\draw[color=blue] (1.2,1.0) -- (2.0,1.0);
		\tkzText[color=red](0.25,-0.50){frame}
		\tkzText[color=red](0.25,-0.85){field}
		\draw[->,thick,color=red] (0.1,0.) -- (0.6,0.0);
		\draw[->,thick,color=red] (0.1,0.) -- (0.1,0.5);
		\draw[->,color=brown] (2.5,1.0) -- (3.0,1.0);
		\begin{scope}[xshift=3.5cm]
			\draw[color=blue] (0,-0.1) -- (2.0,-0.1) -- (2.0,1) -- (2.2,1) -- (2.2,2.1) -- (0,2.1) -- cycle;
			\draw[color=blue] (0.8,0.8) -- (1.2,0.8) -- (1.2,1.0) -- (1.4,1.0) -- (1.4,1.2) -- (0.8,1.2) -- cycle;
			\draw[->,thick,color=cyan] (2.0,1) -- (2.2,1);
			\draw[color=blue] (1.2,1.0) -- (2.0,1.0);	
			\draw[dashed, color=brown] (1.1,1.0) circle (12pt);
			\draw[->,thick,color=red] (0.1,0.) -- (0.6,0.0);
			\draw[->,thick,color=red] (0.1,0.) -- (0.1,0.5);
			\tkzText[color=cyan](2.5,-0.5){Burgers}
			\tkzText[color=cyan](2.5,-0.85){vector $\v$}
			\draw[->,dashed,color=cyan] (2.1,0.8) -- (2.50,-0.3);
		\end{scope}
	\end{scope}
\etkz

\end{center}
\label{fig:intro}
\caption{\footnotesize{Illustration of edge-dislocations in the discrete (left) and continuum settings (right). In the discrete setting, the lattice directions are defined everywhere except in the dislocation core, and the Burgers vector  ({\color{cyan} cyan}) can be recovered by conducting a burgers circuit ({\color{magenta} magenta}) around the core.
The continuum setting depicts a Volterra cut-and-weld process: on the left is a Euclidean annular domain, cut along a ray. The cut is then shifted and glued (right). The depicted domain does not represent a stress-free configuration as the dislocated body cannot be isometrically immersed in the Euclidean space. Still, the Euclidean frame-field at each point (representing the lattice directions in the continuum model) are still well defined in the dislocated body, and from which the Burgers vector ({\color{cyan} cyan}) can be obtained by integrating the frame field along a loop, which is a continuum version of the Burgers circuit \eqref{eq:Burgers_circuit}.}}
\end{figure}

The outcome of such a cut-and-weld procedure is a body endowed with a geometry that is locally-Euclidean, however does not embed (isometrically) into a global Euclidean structure. This formalism is valid for all types of dislocations in two and three dimensions (as well as for disclinations); henceforth we focus on parallel edge-dislocations, which can be described by a two-dimensional model.
%
 
As detailed in Section~\ref{subsec:coord_free}, a body $\M$ containing an edge-dislocation can be viewed as a 
locally-Euclidean Riemannian manifold, endowed with a global frame field representing the lattice directions at every point.
Equivalently, we can replace the frame field by its dual,
which is a non-degenerate  $\R^2$-valued one-form $\P:T\M\to \R^2$, which is known in elasto-plasticity as the \Emph{plastic strain}; the locally-Euclidean structure implies that $\P$ is closed (curl free).
The defect is encoded in a \Emph{Burgers vector} $\v \in \R^2$, obtained by integrating $\P$ along a simple closed curve $C$ encircling the  core of the dislocation,
\beq
\label{eq:Burgers_circuit}
\oint_C \P = \v.
\eeq
This integral is the continuum counterpart of counting lattice sites along a Burgers circuit.

The mechanics under study concern configurations of the dislocated body in the ambient Euclidean space $\R^2$, i.e., maps $f:\M\to \R^2$; to each configuration  is associated an elastic energy accounting for how distorted is the embedded body relative to its intrinsic geometry.
An elementary notion in rational mechanics is that of \Emph{material uniformity}, which in our case amounts to the energy density being ``the same everywhere". 
As explained in Section~\ref{sec:energy}, since the lattice directions encoded by $\P$ define how different points in the material correspond to each other, the energy of a uniform material with dislocations has the form
\beq
\label{eq:EofMQ}
E^{\text{NW}}(f) = \int_\M \calW(df\circ \P^{-1})\,\Volume.
\eeq
where $\calW:\R^2\otimes \R^2 \to [0,\infty]$ is an elastic energy density, and $\Volume$ is the volume form induced by $\P$ (the superscript NW stands for Noll--Wang; see below).
The map $df \circ \P^{-1}$ can be viewed as the elastic distortion in Kr\"oner's multiplicative decomposition of the strain.
We refer to this construction, in which the dislocation is encoded in the circulation of a global frame field, as a \Emph{Volterra model of dislocations}, since Volterra was the first to describe defective bodies using cut-and-weld procedures (although the terminology used by Volterra differs from ours substantially). 

We next briefly describe the admissible strain model, as presented, e.g., in \cite{GLP10, SZ12,MSZ14,MSZ15}, and explain in which sense it is  an approximation of the Volterra model; a more detailed account is given in \secref{sec:admissible}. In a low energetic regime, there exists a coordinate system (typically referred to as reference configuration), in which $\P$  (an $\R^2\otimes\R^2$-valued function in coordinates) is close to the identity, whereas $df$ is close (also in coordinates) to a rotation $U$. 
One can formally linearize $U^T df \circ \P^{-1}$ about the identity and obtain the additive decomposition of the strain,
\[
U^T df \circ \P^{-1} \approx U^T df + (I-\P) \equiv \beta,
\]
so that $\beta:\M\to \R^2\otimes \R^2$ is a closed (curl free) matrix field satisfying
\beq\label{eq:circulation_beta}
\oint_C \beta = -\v.
\eeq

The \Emph{admissible strain} approach considers an elastic energy of the form
\[
E^{\text{as}}(\beta) = \int_\M \calW(\beta)\,dx,
\]
defined over all curl-free fields $\beta$ satisfying the circulation condition \eqref{eq:circulation_beta} (usually changing the righthand side in \eqref{eq:circulation_beta} to $\v$; in this presentation, we retain $-\v$ for consistency).
This derivation, approximating Volterra's model of dislocations by the admissible strain model, requires $\beta$ to be close to $I$, hence the linear-elasticity version of the admissible strain model replaces $\calW(\beta)$ with $\bbW(\beta-I)$, where $\bbW$ is the Hessian of $\calW$ at the identity. 
An alternative derivation of an admissible strain model departs from an Eulerian approach, where the strain $\beta$ is a map from the tangent space of a deformed (or actual/spatial) configuration in Euclidean space to the reference lattice \cite[p.~180]{MSZ15} \cite[Sec.~3.1]{CGM23}; in this approach the circulation condition for $\beta$ is exact, however, being an Eulerian approach in which the deformed configuration is given, the variational problem that this model represents is quite different.
In all of these approaches, one can further define bodies having many dislocations.

\paragraph{Modeling dislocation fields}
Bodies containing ``macroscopically-many" dislocations are ubiquitous in nature, and several models for bodies with distributed dislocations were derived along the years.
In the 1950's, Nye \cite{Nye53}, Bilby \cite{BBS55}, Kondo \cite{Kon55} and others, modeled bodies with distributed dislocations as Riemannian manifolds $(\M,\g)$, endowed with a curvature-free, metric affine connection $\nabla$.
Up to choosing a basis at a single point, the joint choice of $\g$ and $\nabla$ is equivalent to choosing an implant map $\P:T\M\to \R^2$ as above. 
However, when describing a distribution of dislocations, $\P$ needs not to be closed (curl free).
An energetic model of the form \eqref{eq:EofMQ} for non-closed $\P$ was proposed by Noll and Wang \cite{Nol58,Wan67} for describing continuously distributed dislocations; see \cite{EKM20} for a summary of the Kondo--Bilby and Noll--Wang approaches.

Later on, the so-called \Emph{strain-gradient models} were  
developed by Fleck--Hutchinson \cite{FH93} and Gurtin \cite{Gur00,GA05}. 
In these models, the energy is of the form
\[
E^{\text{sg}}(u,\beta^p) = \int_\W \bbW(\nabla u - \beta^p)\,dx + \int_\W \Sigma(\curl \beta^p) \,dx,
\]
where $u:\W\to \R^2$ is a displacement field relative to a reference configuration, $\beta^p$ is a plastic strain, whose curl represents the distribution of dislocations, $\bbW$ is a quadratic elastic energy density, and $\Sigma$ is a model-dependent function.
The first term is a linear elastic energy of the elastic strain, whereas the second term, which is independent of the displacement field, is the self-energy contribution of the plastic strain.

Comparing the two approaches, it is apparent that the Noll-Wang model \eqref{eq:EofMQ} for a non-curl-free $\P$ concerns systems subject to higher energy/stress---there is no a priori reference (zero energy) configuration and the decomposition of the strain is multiplicative rather than additive (compare $df \circ \P^{-1}$ in $E^{\text{NW}}$ and $\nabla u - \beta^p$ in $E^{\text{sg}}$).
This observation will be made precise in the next part, where we describe the derivation of these models from models of finitely-many dislocations, each in a different energy scaling.

\paragraph{Rigorous homogenization of dislocations: previous results}
In two dimensions, the Kondo--Bilby geometric model, and the Noll--Wang energetic model were obtained as limits of the Volterra  model of finitely-many dislocations: 
in the Kondo--Bilby model this reduces to showing that manifolds $(\M,\P)$ can be obtained as limits of manifolds $(\M_n,\P_n)$ with finitely-many dislocations (that is, $d\P_n = 0$), as the magnitude of the dislocations tends to zero and their number tends to infinity \cite{KM15,KM16}; the  
Noll--Wang model was obtained by taking the $\Gamma$-limit of the associated energies $E^{\text{NW}}$ of $(\M_n,\P_n)$, under some additional technical assumptions \cite{KM16a,EKM20}.
In both cases, the total Burgers vector in $(\M_n,\P_n)$ is $O(1)$ (the small parameter is the typical magnitude of a dislocation), as is the associated elastic energy. 

The strain-gradient model was first derived by Garroni--Leoni--Ponsiglione (GLP) \cite{GLP10} as a $\Gamma$-limit of the admissible strain model, as the magnitude of the dislocations tends to zero, in the case where the underlying energy density is a quadratic energy density $\bbW$ (i.e., a linear elastic model).
For a parameter $\e\to0$, they considered systems with (roughly) $n_\e\to \infty$ dislocations, each of magnitude $\e$. They showed that the energy contribution of each dislocation is of order $\e^2 \log(1/\e)$, summing up to a \Emph{self-energy} of order $n_\e \e^2 \log(1/\e)$.  
Another energy contribution is an \Emph{interaction energy} of order $n_\e^2 \e^2$.
In the low energy regime considered, both these terms are assumed to tend to zero as $\e\to 0$.

The identification of two distinct energy contributions gives rise to different energy regimes: 
\Emph{Subcritical} for $n_\e \ll \log(1/\e)$, \Emph{critical} for $n_\e = \log(1/\e)$ and \Emph{supercritical} for $n_\e \gg \log(1/\e)$.
GLP considered the energy densities
\[
\calE^{\text{GLP}}_\e(\beta_\e,\mu_\e) = \frac{1}{h_\e^2} \int_{\M_\e} \bbW(\beta_\e-I)\,dx,
\]
where $h_\e^2 = \max\{n_\e \e^2 \log(1/\e),n_\e^2 \e^2\}$, $\mu_\e$ is a sum of $\delta$-functions representing the locations and magnitudes of the dislocations, $\M_\e$ is a subset of a domain $\W$ obtained by removing discs of radius $\e$ around the support of $\mu_\e$, and $\beta_\e\in L^2(\W;\R^2\otimes \R^2)$ is a strain field satisfying $\curl \beta_\e = \mu_\e$.
GLP showed (under some technical assumptions) that $\calE^{\text{GLP}}_\e$ $\Gamma$-converges to $\calE_0: L^2(\W;\R^2\otimes \R^2) \times  \calM(\W;\R^2) \to [0,\infty]$ given by
\beq
\label{eq:limenergy}
\calE_0(J,\mu) =
\Cases{
 \int_\W \bbW (J)\,dx + \int_\W \Sigma(\frac{d\mu}{d|\mu|}) \,d|\mu| & \parbox[t]{.4\textwidth}{subcritical and $\curl J =0$} \\
 \int_\W \bbW (J)\,dx  + \int_\W \Sigma(\frac{d\mu}{d|\mu|}) \,d|\mu| & \parbox[t]{.4\textwidth}{critical, $\mu\in \HminusOne(\W;\R^2)$ and $\curl J = -\mu$} \\
\int_\W \bbW (J)\,dx  & \parbox[t]{.4\textwidth}{supercritical, $\mu\in \HminusOne(\W;\R^2)$ and $\curl J = -\mu$} \\
\infty & \text{otherwise},
}
\eeq
where $\Sigma$ is a one-homogeneous, convex function given by an appropriate cell formula.
The topology with respect to which the $\Gamma$-limit is obtained is induced by the limits $\h_\e^{-1} (\beta_\e-I) \weakly J$ in $L^2$ 
and $\frac{1}{n_\e\e}\mu_\e \weakstar \mu$ in $\calM$.

Obtaining strain-gradient plasticity from a \emph{non-linear} energy density $\calW$, i.e., for
\beq\label{eq:MSZ}
\calE^{\text{SZ}}_\e(\beta_\e,\mu_\e) = \frac{1}{h_\e^2} \int_{\M_\e} \calW(\beta_\e)\,dx,
\eeq
was first done by Scardia--Zeppieri \cite{SZ12} for the case of finitely-many dislocations of order $\e$ (that is, a constant $n_\e$), in fixed, non-variable, positions.
A generalization for $n_\e\to \infty$ was then obtained by M\"uller--Scardia--Zeppieri \cite{MSZ14}, who considered the critical regime $n_\e = \log(1/\e)$ and obtained the limiting energy
$\calE_0^{\text{MSZ}}: L^2(\W;\R^2\otimes \R^2) \times \SO(2) \times \calM(\W;\R^2) \to [0,\infty]$ given by
\beq
\label{eq:limenergy_MSZ}
\calE_0^{\text{MSZ}}(J,U,\mu) =
\Cases{
 \int_\W \bbW (J)\,dx  + \int_\W \Sigma\brk{U ,\frac{d\mu}{d|\mu|}} \,d|\mu| & \parbox[t]{.3\textwidth}{$\mu\in \HminusOne(\W;\R^2)\quad\qquad$ and $\curl J = -U^T\mu$,} \\
\infty & \text{otherwise},
}
\eeq
where $\bbW$ is the Hessian of $\calW$ at the identity.
In \cite{MSZ14}, the topology is induced by the limits $h_\e^{-1} (U_\e^T \beta_\e - I)\weakly J$ in $L^2$ for some $U_\e\in \SO(2)$ converging to $U$, and $\frac{1}{n_\e\e}\mu_\e \weakstar \mu$ in $\calM$. 

Several improvements of these results were obtained throughout the years:
the removal of a non-physical upper-bound assumption on the energy density \cite{MSZ15}, the relaxation of the assumption that the dislocations are well-separated (\cite{DGP12} in the subcritical regime, \cite{Gin19} in the critical regime, both in the linear model $\calE^{\text{GLP}}_\e$), and more.
Other recent results regarding three-dimensional models \cite{CGO15,GMS21,CGM23} are beyond the scope of this paper. 

\paragraph{Main results}
As exposed above, a first contribution of this work is the 
establishment of a nonlinear framework for the analysis of solids with defects, along with the notion of material uniformity, which is consistent with the microscopic models of lattice defects. 
While it is difficult to claim novelty when it comes to ideas that have been thoroughly discussed in the rational mechanics literature in the past 75 years, our framework, as far as we know, is the first one to combine these ideas with a rigorous calculus of variations approach, and thus it is a starting point for a wide range of future analyses, including the homogenization of media with mixed types of defects, and thin sheets containing  defects.
Similar kinematic considerations, with multiplicative strain decomposition, were recently presented in \cite{HR22b} in the context of elasto-plastic evolution.

The geometry of a single edge-dislocation is characterized 
in section~\ref{sec:single_dislocation}.
We give an axiomatic, coordinate-free definition of a two-dimensional body $(\M,\P)$ containing a single edge-dislocation of Burgers vector $\v$ (Definition~\ref{def:single_disloc}). 
A natural question is whether our axiomatic definition characterizes a unique object;  we show that it defines the body uniquely up to the shape of the core of the dislocation (Theorem~\ref{thm:disloc_unique}).
This extends the uniqueness result \cite[Theorem~3]{KMS15} in the dislocation-free case $\v=0$; the method of proof here is quite different and requires new ideas.
The importance of this result is that in different contexts, it is useful to describe the geometry of a dislocation in different ways---compare, for example, the different metrics used in \cite{KMS15}, in \cite{Kup17} and in the current work. Theorem~\ref{thm:disloc_unique} establishes that they all describe the same object.

Finally, we obtain the strain-gradient model as a homogenization limit of the Volterra model of dislocations.
Combining with \cite{EKM20}, this establishes both the strain-gradient and the Noll--Wang models under the same framework---as homogenization limits of Volterra's dislocations under different energy scalings.

Without getting into technical details, we show that the elastic energies (rescaled Noll--Wang energies)
\beq\label{eq:rescaled_Ee}
\calE_\e(f_\e,\PEps) =  \frac{1}{h_\e^2} \int_{\MEps} \calW(df_\e\circ \PEps^{-1})\,\VolumeEps,
\eeq
where $(\MEps,\PEps)$ is a sequence of bodies with (roughly) $n_\e$ dislocations of order $\e$, and $f_\e \in H^1(\MEps,\R^2)$, $\Gamma$-converges to the GLP limiting energy \eqref{eq:limenergy}. 
Here, the $\Gamma$-convergence is with respect to an appropriate notion of convergence of the bodies  $(\MEps,\PEps)$ to $(\W,\mu)$ defined in Section~\ref{sec:conv_disloc}, and with respect to the convergence of the scaled displacements $\h_\e^{-1}(U_\e^T df_\e - \PEps) \weakly J$ in $L^2$ for some $U_\e\in \SO(2)$.
A detailed formulation of the result appears in the beginning of Section~\ref{sec:GammaConv}.

We note that the self-energy term in \eqref{eq:limenergy}, in contrast to the self-energy in \eqref{eq:limenergy_MSZ}, does not depend on a limiting rotation $U$, even though the energies in both our model and \eqref{eq:MSZ} used in \cite{SZ12,MSZ14} are rotation-invariant.
This difference occurs because in our model, unlike in the admissible strain model, the Burgers vector is independent of the action of the rotation group: indeed, the rotation group only acts on $f_\e$ and not on $\PEps$, whereas in the admissible strain model one cannot rotate $\beta_\e$ without rotating $\mu_\e$. The limiting dislocation density $\mu$ is a limit of $\PEps$ alone, and thus the self-energy term associated with it is independent of a global rotation closest to the strains.

\paragraph{Structure of the paper and intermediate results}
\begin{itemize}
\item
\textbf{Modeling an elastic body:}
Section~\ref{sec:energy} presents a short introduction to the Noll--Wang energy \eqref{eq:EofMQ} associated with a general elastic body $(\M,\P)$, and lists our assumptions on the energy density $\calW$.

\item \textbf{Geometry of a dislocation:}
In section~\ref{sec:single_dislocation} we characterize axiomatically the geometry of a single edge-dislocation, which lays the basis for the subsequent analysis. 
We prove that this coordinate-free definition fully characterizes a geometry, up to the shape of a core region.

\item \textbf{Energy of a single dislocation:}
In order to obtain sharp energy estimates, we need an explicit coordinate representation of  a body containing a dislocation. 
We derive such a representation  by constructing a model body $(\hM_\v,\hP_\v)$ (Section~\ref{sec:coordinate_construction}), and we analyze its deviation from a Euclidean annulus (Section~\ref{sec:dev_from_euc}).
In Section~\ref{sec:energy_single}, we analyze the (infimal) elastic energy associated with $(\hM_\v,\hP_\v)$, and show that it is of order $|\v|^2\log(1/|\v|)$.
A more detailed analysis of the energy of $(\hM_{\e\v},\hP_{\e\v})$ as $\e\to 0$, shows that after rescaling the energy by $\e^2\log(1/\e)$, it tends to a quadratic energy functional $\Ilin_0(\v)$, whose convex relaxation, as in pervious work, yields the self-energy function $\Sigma$ appearing in the $\Gamma$-limit.

\item 
\textbf{Relation to admissible strain model:}
In Section~\ref{sec:admissible_strain} we further elaborate on how the admissible strain model can be formally obtained from the Volterra model via linearization.

\item \textbf{Geometry of multiple dislocations:} 
In Section~\ref{sec:def_bodies}, we define bodies $(\M,\P)$ containing multiple edge-dislocations---locally-flat manifolds that look locally like a body with a single edge-dislocation.
Following ideas of Epstein--Segev (\cite{ES14,ES15}, see also \cite{KO20}) we view the implant map $\P$ as a measure $\torsion\in \calM(\M;R^2)$; this alternative point of view is important when discussing convergence of such bodies.
In Section~\ref{sec:disloc_construction}, we construct bodies containing multiple dislocations, and estimate their deviation from a multiply-punctured Euclidean plane. 
This construction is essential for the construction of a recovery sequence in the  $\Gamma$-convergence analysis.

\item \textbf{Convergence of bodies with many dislocations:}
The $\Gamma$-convergence of the energy associated with bodies containing multiple dislocations, must rely on a primal notion of convergence of bodies containing multiple dislocations.
Such a notion is defined in Section~\ref{sec:conv_disloc}, in which we present a few examples, which also form the basis for the recovery sequence in the $\Gamma$-convergence section.
In lay terms, a sequence of bodies with dislocations $(\M_\e,\PEps)$ of magnitude $\e$ converges,  with respect to a parameter $n_\e$,  to a domain $\W\subset \R^2$ and a measure $\mu \in \calM(\W;\R^2)$, if 
we can embed $\MEps$ as a subset of $\W$, such that:

\begin{enumerate}[itemsep=0pt,label=(\alph*)]
\item The volume of $\W\setminus\MEps$ tends to zero. 
\item \Emph{Distortion bounds}: $|I - \PEps|$ tends to zero uniformly, except in the vicinity of the dislocations, and $\|I - \PEps\|_{L^2} = O(h_\e)$. 
\item \Emph{Burgers vector convergence}: The measures $\frac{1}{n_\e\e} \torsion_\e$ weakly converge to $\mu$, where $\torsion_\e$ are the measures associated with $\PEps$. 
\end{enumerate}

\item \textbf{Geometric rigidity:}
As always in low-energy limits of non-linear elasticity, one needs a Friesecke--James--M\"uller-type geometric rigidity estimate (henceforth FJM).
In Theorem~\ref{thm:rigidity}, we prove an asymptotic FJM result for converging bodies with dislocations:
If $(\MEps,\PEps) \to (\W,\mu)$, then for every $f_\e \in \HOne(\MEps;\R^2)$, there exists a matrix $U_\e\in \SO(2)$, such that
\[
\|df_\e - U_\e \PEps\|_{L^2(\MEps)}^2 \le C \int_{\MEps} \dist^2(df_\e\circ \PEps^{-1},\SO(2))\,\VolumeEps + Ch_\e^2,
\]
where the constant $C$ depends on $\W$ and on the uniform bound of $| I - \PEps|$.
We show a-posteriori (using both $\Gamma$-convergence and compactness results) that for $\e$ small enough, the extra $h_\e^2$ term can be removed at the expense of $C$ depending (in an explicit way) on the limiting dislocation density $\mu$, see Theorem~\ref{thm:rigidity2}. 
It would be interesting to know whether a similar statement (without the $h_\e^2$ error term) holds for a fixed $\e$, with a constant that is independent of the number of dislocations. 
See the discussion in the end of Section~\ref{sec:rigidity}.

\item \textbf{Compactness:}
Using the rigidity theorem, we prove in Theorem~\ref{thm:compactness} that if $(\MEps,\PEps) \to (\W,\mu)$ and $\calE_\e(f_\e) = O(1)$, where $\calE_\e$ is given in \eqref{eq:rescaled_Ee}, then  there exists a subsequence $U_\e \in \SO(2)$ such that the rescaled displacements $h_\e^{-1}(U_\e^T df_\e - \PEps)$ weakly converge in $L^2$ to $J\in L^2(\W;\R^2\otimes \R^2)$, where $J$ satisfies $\curl J= 0$ (subcritical) or $\curl J = -\mu$ (critical or supercritical).
This compactness of the displacement is analogous to the one in \cite{GLP10}, except that in the supercritical case, only the compactness of the symmetric part of the displacement is obtained in \cite{GLP10}; here full compactness is obtained due to the stronger rigidity estimate Theorem~\ref{thm:rigidity2}, which stems from our distortion bounds assumptions on $\PEps$.
Furthermore, we show that in the critical and subcritical regimes, and under a mild separation assumption between the dislocations, if a sequence of bodies $(\MEps,\PEps)$ satisfies the distortion bounds with respect to a limiting domain $\W$, then $\frac{1}{n_\e\e} \torsion_\e\weakstar\mu$ for a subsequence, and thus $(\MEps,\PEps)\to (\W,\mu)$.
This compactness of the dislocation measures is analogous to the one in \cite{GLP10, MSZ14}.

\item \textbf{$\Gamma$-convergence:}
In Section~\ref{sec:GammaConv}, we gather all the ingredients, in particular the asymptotic estimates for a single dislocation (Section~\ref{sec:asymp_estimates}) and the construction of bodies containing multiple dislocations (Section~\ref{sec:disloc_construction}), and prove the $\Gamma$-limit result stated above, under the assumptions that the dislocations are well-separated (the infimal distance between dislocations $\rho_\e$ satisfies $\log(1/\rho_\e)\ll \log(1/\e)$), and that the energy is not too high ($\log n_\e \ll \log(1/\e)$).

\end{itemize}

\paragraph{Main challenges}
While the $\Gamma$-convergence proof follows eventually a course similar to \cite{GLP10, SZ12,MSZ14}, e.g., separating the energy $\calE_\e$ into core and far-field regimes, there are significant challenges, both conceptual and technical, in applying this course to our nonlinear geometric setting.

First, as the model of a dislocation is encoded in a section of a vector bundle on a manifold, rather than by a measure over a fixed domain, the correct setting of the problem had to be identified, including the correct definitions of manifolds containing multiple dislocations and their convergence.
$\Gamma$-convergence of elastic models over convergent manifolds appeared in our previous work on the Noll--Wang limit \cite{KM16a,EKM20}, but in this paper the notions of convergence are much more refined.

The construction of bodies containing multiple dislocations also poses a new challenge---while measures can be added, frame fields over a manifold cannot (this can be interpreted as a geometric nonlinearity, in addition to the energetic nonlinearity of a nonlinear-elastic energy density $\calW$).  
One approach to overcome this difficulty is by gluing bodies containing single dislocations; this approach was adopted in \cite{KM15,KM16,EKM20}.
However, the energy estimates obtained for these compound manifolds are not sharp enough for obtaining recovery sequences in both critical and supercritical regimes.
Instead, we adapted ideas from the construction of recovery sequences of strains in \cite{MSZ14} to construct frame fields $\P$.
The price we have to pay in this construction is that the precise shape of the cores of the dislocations is not known, hence we need to estimate how they differ from the core in the single-dislocation model manifold $(\hM_\v,\hP_\v)$, for which we have explicit calculations.

In previous derivations of the strain-gradient limit, the rigidity estimates (FJM-like or Korn-like) compared matrix fields that are not curl-free to a fixed rotation (or an infinitesimal rotation).
In this work (Theorems~\ref{thm:rigidity_single_disloc} and \ref{thm:rigidity}) the variables are configurations $f:\M\to \R^2$, whose associated strains $df:T\M\to \R^2$ are curl-free. However, the role of a ``constant rotation" is played by the parallel 1-form $\P$ (which is not exact!).
To the best of our knowledge, this is the first time that such rigidity estimates appear in the literature (although their proofs are simpler than the ones for incompatible strains, and rely on the standard FJM estimate).

Finally, in the admissible strain approach one can choose the size of the core of the dislocation regardless of the magnitude of the Burgers vector; in the geometric approach, as in the lattice model, the size of the core is bounded from below by the magnitude of the Burgers vector (see Comment 5 after Definition~\ref{def:single_disloc}).
This makes various estimates throughout the work (in particular in Section~\ref{sec:asymp_estimates}) more challenging, as one cannot take the core size to zero independently of the Burgers vector.

\paragraph{Future extensions}
We now list some potential extensions, which, given the length of this manuscript, are not addressed in this work:
\begin{itemize}
\item
\textbf{Energy upper bound:} 
Our $\Gamma$-convergence analysis (Theorems~\ref{thm:liminf} and \ref{thm:limsup}) hinges on a non-physical upper bound  \eqref{eq:bound_calW} on the growth of the energy density $\calW$. 
This upper bound was also assumed in \cite{SZ12,MSZ14}. 
Specifically, this bound is used to bound the energies of specific constructions (Corollary~\ref{corr:Escaling}, which is not required for the $\Gamma$-limit result, Proposition~\ref{prop:cell_liminf} and Lemma~\ref{lem:rec_seq}).
We expect that this upper bound assumption can be relaxed---in Proposition~\ref{prop:cell_liminf} we only need it to hold in a neighborhood of $\SO(2)$, whereas in Lemma~\ref{lem:rec_seq}, an improved ansatz, similar the one in \cite{MSZ15}, should probably enable us to relax the upper bound assumption considerably.

\item
\textbf{Dislocation separation:}
The second technical assumption that is used throughout the $\Gamma$-convergence analysis is that the minimal separation $\rho_\e$ between dislocations tends to zero slower than any positive power of $\e$, and that $n_\e$ tends to infinity slower than any negative power of $\e$.
This means that the number of dislocations grows slower than any negative power of $\e$, that they are well-separated, and that the magnitude of all Burgers vectors is of order $\e^{1-o(1)}$.
Similar assumptions appear also in the admissible strain derivations of strain-gradient plasticity \cite{GLP10,MSZ14} and were subsequently relaxed \cite{DGP12,Gin19}; we expect that a similar relaxation can be obtained also in the geometric model. 

\item
\textbf{Compactness of lattice structure:}
Our $\Gamma$-convergence result is with respect to a topology of a joint convergence $(\MEps,\PEps,f_\e)\to(\W,\mu,J)$. Our compactness result, however, not only assumes that the energy $\calE_\e(f_\e,\PEps)$ is of order 1, but also that $\PEps$ satisfies the global distortion bounds described above.
This additional assumption is needed since the variable entering the energy is $df_\e\circ\PEps^{-1}$, whereas the information on the curl is only for $\PEps$; we needed the global distortion assumption to obtain the associated geometric rigidity result. 
No such extra assumption is necessary in the admissible strain approach, since the curl condition is explicitly given for the variable entering the energy. It would be of interest to relieve the extra assumption in our setting. 

\item
\textbf{Lower energy results in the supercritical regime:}
Our compactness result for the dislocation measures, as in \cite{GLP10}, only holds in the critical and subcritical regimes, as the self-energy is not strong enough to control the norms of the measures $\torsion_\e$ in the supercritical regime $n_\e\gg \log(1/\e)$.
If one considers a lower energetic regime --- namely, $\e^2 n_\e \log(1/\e)$ rather than $h_\e^2 = \e^2n_\e^2$ --- then compactness of $\torsion_\e$ follows as well (see remark below Theorem~\ref{thm:measure_comp}).
In the linearized context, a full $\Gamma$-convergence result in this energy scaling was obtain \cite{FPP19}; it would be interesting to adapt their results to our settings. 
Note that the compactness and lower bounds in this case are essentially covered in this work (in fact, we get a better rate of convergence of the skew-symmetric part of the strain compared to \cite{FPP19}, similar to the discussion about the supercritical compactness above), and the main challenge is in the upper bound --- to construct bodies with dislocations in the supercritical regime with this lower energy.

\item
\textbf{Three-dimensional dislocations:}
The results of this paper concern edge-dislocations in two dimensions.
A next step would be to extend this analysis to dislocations in three-dimensional bodies (following the work on the admissible strain model, e.g., \cite{CGO15,GMS21,CGM23}).
Note that parts of the settings and the constructions prevail in three dimensions. For example, the geometric, coordinate-free definition of a body with dislocations and the construction of bodies with multiple dislocations (Section~\ref{sec:disloc_construction}) can be applied in three dimensions, for dislocation of edge- and screw-type; such a construction would provide the first coordinate expression for Riemannian metrics of bodies with multiple dislocations in three dimensions.

\end{itemize}

\paragraph{Notations} 
We denote the Euclidean metric on $\R^2$ by $\euc$, the standard inner-product on $\R^2$ by $\ip{\cdot,\cdot}$, the standard frame by $\{\pl_1,\pl_2\}$ and the corresponding coframe by $\{dx^1,dx^2\}$. We denote by $B_R(x)$ the Euclidean ball of radius $R$ centered at $x$, and use the shorter notation $B_R = B_R(0)$ for balls centered at the origin.

Let $(V,\g)$ and $(W,\h)$ be inner-product spaces and let $A\in\Hom(V,W)\simeq V^*\otimes W$ be a linear operator. We denote the operator norm of $A$ by $|A|_{\g,\h}$. For a subset $K\subset \Hom(V,W)$, we denote the distance of $A$ from $K$ by
\[
\dist_{\g,\h}(A,K).
\]
In cases where no confusion should arise, we will omit the subscript $\g,\h$.
If $V$ and $W$ are oriented and of same dimension, we denote the set of orientation-preserving isometries by $\SO(\g,\h)\subset \Hom(V,W)$. These notations carry on for Riemannian vector bundles over a manifold.

Let $(\M,\g)$ be a Riemannian manifold. We denote the space of vector fields on $\M$ by $\frakX(\M) = \Gamma(T\M)$, and the space of $k$-forms by $\W^{k}(\M)$. For a vector bundle $E\to\M$, we denote the space of $k$-forms on $\M$ taking values in $E$ by $\W^k(\M;E)$, and by $L^2\W^k(\M;E)$ the space of $k$-forms of $L^2$-regularity.

For a $d$-dimensional manifold $\M$, an important role in this paper is played by non-degenerate $\R^d$-valued one-forms, that is, sections of $T^*\M\otimes \R^d$.
Given such a section $\P$, at each point $p$, $\P_p$ is a linear map $T_p\M\to \R^d$.
We say that $\P$ is non-degenerate if this linear map is invertible at each point, in which case we denote its pointwise inverse by $\P^{-1} \in \Gamma(\R^d\otimes T\M)$.
The set $\{\P_p^{-1}(\pl_1),\ldots, \P_p^{-1}(\pl_d)\}$ is thus a basis (a frame) for $T_p\M$.
Therefore we can think of $\P^{-1}$ (and thus of $\P$) as a smooth choice of lattice directions at each point.

Given an $\R^2$-valued function $\psi\in C^\infty(\M;\R^2)$ and $\R^2$-valued 1-forms $A,B\in\W^1(\M;\R^2)$, we define the $\R^2\otimes\R^2$-valued forms
\[
\psi\otimes A \in \W^1(\M;\R^2\otimes\R^2)
\Textand
A\wedge B  \in \W^2(\M;\R^2\otimes\R^2),
\]
via the coordinate expressions 
\[
(\psi\otimes A)_i^{\alpha\beta} = \psi^\alpha A_i^\beta
\Textand
(A\wedge B)_{ij}^{\alpha\beta} = A_i^\alpha B_j^\beta - B_i^\alpha A_j^\beta,
\]
where $\alpha,\beta$ denote Euclidean coordinates and $I,j$ denote entries with respect to an orthonormal frame of $T^*\M$. 
Given a $k$-form $\alpha$ taking values in $\R^2\otimes\R^2$, we denote by $\TR\alpha$ the real-valued $k$-form obtained by contracting the Euclidean components (in the above expressions, the indices $\alpha$ and $\beta$).

Let $(\N,\h)$ be another Riemannian manifold and let $\vp:\M\to\N$ be a smooth map. For a vector bundle $F\to\N$, we denote the pullback vector bundle over $\M$ by $\vp^*F$, with the canonical identification $(\vp^*F)_p \simeq F_{\vp(p)}$. For a section $\eta\in\Gamma(F)$, we denote by $\vp^*\eta\in\Gamma(\vp^*F)$ the pullback section, $(\vp^*\eta)_p = \eta_{\vp(p)}$. The pullback of sections differs from the pullback of forms: for $\omega\in\W^k(\N)$, we denote by $\vp^\#\omega\in\W^k(\M)$ the $k$-form on $\M$ defined by
\[
(\vp^\#\omega)_p(X_1,\dots,X_k) = \omega_{\vp(p)}(d\vp_p(X_1),\dots,d\vp_p(X_k)).
\]
Similarly, $\P\in \W^1(\M;\R^d)$, pulls back multilinear maps $A: (\R^d)^k \to \R$ via
\[
(\P^\# A)_p (X_1,\dots,X_k) = A(\P_p(X_1),\dots,\P_p(X_k)).
\]
In particular, a non-degenerate $\P$ defines an inner product $\P^\#\euc$ on $\M$ by pulling back the Euclidean inner-product $\euc$.

Finally, throughout this work, we use the symbols $\lesssim$ and $\gtrsim$ to denote inequalities up to a multiplicative constant, i.e.,
\[
f(x) \lesssim g(x)
\]
means that there exists a constant $C>0$ such that $f(x) \le C\, g(x)$ for all $x$. 
If $f(x)\lesssim g(x)$ and $f(x)\gtrsim g(x)$, we write $f(x)\simeq g(x)$.  
Whenever needed, we specify on which parameters the multiplicative constant does or does not depend.

\section{Material-uniform elastic models}
\label{sec:energy}

In this section we present a short exposition of the modeling of nonlinear elasticity using the geometric formalism of Riemannian geometry. 

\begin{definition}
\label{def:elastic_body}
A \textbf{$d$-dimensional elastic body} is a pair $(\M,\P)$, where $\M$ is a connected, $d$-dimensional manifold, possibly having a Lipschitz boundary, and  $\P$ is a global non-degenerate section of $T^*\M\otimes\R^d$ (i.e., an $\R^d$-valued one-form).
A \textbf{configuration} of the elastic body is a map
\[
f: \M \to \R^d
\]
into the ambient Euclidean space. 
\end{definition}

The intrinsic geometry of the body in encoded in the section $\P$.
A Riemannian metric $\g$ on $\M$ is defined by $\g_p(u,v) = \ip{\P_p(u),\P_p(v)}$, that is $\g = \P^\#\euc$.
Note that by construction, $\P$ is a section of $\SO(\g,\euc)$.
 
The map $\P^{-1}$ is sometimes called an \textbf{implant map}, for reasons that will be clarified soon; with a slight abuse of terminology, we sometimes refer to $\P$ itself as the implant map.
In the context of plasticity (i.e., when $\P$ is thought of as a kinematic variable), it is known as the \Emph{plastic strain}.
We can think of $\P$ as a frame of $T^*\M$, $(u\mapsto\ip{\P u,\pl_i})_{i=1}^d$; its dual frame $(\P^{-1}(\pl_i))_{i=1}^d$ is sometimes called the \textbf{scaffold} of the body \cite{HR22b}, and can be thought as a continuum field representing the lattice directions at each point (see Fig.~\ref{fig:intro}).
The existence of a global section $\P$ imposes some topological restrictions on $\M$, which in some cases can be relieved by replacing $\P$ with a covering of sections satisfying compatibility conditions. In this paper the above simpler definition is sufficient.

We endow elastic bodies with an elastic energy functional, which quantifies an amount by which the geometry of a configuration is deformed in comparison with the intrinsic geometry of the body. We assume that the energy density is uniform, in the sense that it behaves ``the same way'' at all points.
Mathematically speaking, given a function $\calW:\R^d\otimes\R^d\to\R$, called the \Emph{archetypal energy density} (see \cite{EKM20} for details), we define the \textbf{elastic energy} of a configuration $f$ associated with $\calW$ and $(\M,\P)$ by
\beq
E(f,\P) = \int_\M \calW(df \circ \P^{-1})\, \Volume,
\label{eq:E(f)}
\eeq
where $\Volume$ is the volume form associated with the metric induced by $\P$. 
Whether $\P$ is fixed or can be treated as a variable depends on the problem at hand.
The differential $df:T\M\to \R^d$ determines how tangent vectors in $\M$ map under $f$ into tangent vectors in $\R^d$. The right-composition of $df$ with $\P^{-1}:\R^d \to T\M$ yields a linear endomorphism of $\R^d$, which can be viewed as an elastic distortion. In this sense, $\P^{-1}$ is a plastic distortion which ``implants" the archetypal energy density into the body see \figref{fig:2}.
This definition also enables us to discuss different elastic bodies having the ``same'' elastic behavior.

\begin{figure}
\begin{center}
\btkz
	\draw[thick] plot [smooth cycle] coordinates {(1,1.5)(3.5,0.9)(4.8,0)(6,1.3)(7.2,2)(7.5,3.5)(6,4)(4,5)(2.6,3.5)(1,3)} node at (2,3.6) {$\M$};
	\tkzDefPoint(2.5,2.2){x};
	\tkzDefPoint(5.2,2.8){y};
	\tkzDrawPoints(x,y);
	\tkzLabelPoint[below](x){$x$};
	\tkzLabelPoint[below](y){$y$};
	\begin{scope}[xshift=2.5cm,yshift=2.2cm]
	\draw[->,thick, color=cyan]  (0,0) -- ++(0.4,0);
	\draw[->,thick, color=cyan]  (0,0) -- ++(0,0.4);
	\draw[->,thick, rotate = 20, color=red]  (0,0) -- ++(0.5,0);
	\draw[->,thick, rotate = 50, color=red]  (0,0) -- ++(0,0.8);
	\tkzText[color=red](0.7,0.7){$F_x\circ \P_x^{-1}$}
	\tkzText[color=cyan](0.95,0.1){$\P_x^{-1}$}
	\end{scope}

	\begin{scope}[xshift=5.2cm,yshift=2.8cm]
	\draw[->,thick, rotate=30, color=cyan]  (0,0) -- ++(0.4,0);
	\draw[->,thick, rotate=30, color=cyan]  (0,0) -- ++(0,0.4);
	\draw[->,thick, rotate = 50, color=red]  (0,0) -- ++(0.5,0);
	\draw[->,thick, rotate = 80, color=red]  (0,0) -- ++(0,0.8);
	\tkzText[color=red](1,0.6){$F_y \circ \P_y^{-1}$}
	\tkzText[color=cyan](0.9,-0.1){$\P_y^{-1}$}
	\end{scope}
	
	\begin{scope}[xshift=10cm,yshift=2.4cm]
	\tkzText(0.0,0.8){$F_x\circ \P_x^{-1} = F_y\circ \P_y^{-1}$}
	\tkzText(0.0,0.0){$\Downarrow$}
	\tkzText(0.0,-0.8){$W(x,F_x) = W(y,F_y)$}
	\end{scope}

\etkz

\end{center}
\label{fig:2}
\caption{
\footnotesize{Material uniformity: A materially-uniform body is a body in which the mechanical response is ``the same at all points''. Given a frame ({\color{cyan} cyan}) at each point (in our case $\{\P^{-1}(\pl_i)\}_{i=1}^d$), the elastic energy density of a deformation $F_x:T_x\M\to \R^d$ at a point $x$ and $F_y:T_y\M\to \R^d$ at a point $y$ is the same if the way they operate on the given frame is the same ({\color{red} red}).
This amounts to $F_x\circ \P^{-1}_x = F_y \circ \P^{-1}_y: \R^d\to \R^d$.
Thus, the elastic energy density at a point $x$ is given by $W(x, \cdot) = \calW(\cdot \circ \P_x^{-1})$ for some function $\calW:\R^d\times \R^d \to \R$.}}
\end{figure}

We assume that the archetypal energy density $\calW$ satisfies the following properties:
\begin{enumerate}[itemsep=0pt,label=(\alph*)]
\item
\Emph{Regularity}:
$\calW$ is continuous and twice differentiable in a neighborhood of $\SO(d)$.
\item 
\Emph{Frame-indifference}: $\calW(UA) = \calW(A)$ for every $U\in\SO(d)$.
\item
\Emph{Upper and lower bounds}:
\beq
\dist^2(A,\SO(d)) \lesssim \calW(A) \lesssim \dist^2(A,\SO(d)).
\label{eq:bound_calW}
\eeq
\end{enumerate}
As noted in the introduction, the upper bound in \eqref{eq:bound_calW} is a non-physical assumption; it is needed for technical reasons, and can likely be relaxed.
Throughout this work, the quadratic form associated with $D_I^2 \calW$ will play a prominent role; we denote
\[
\bbW(A) = \half D_I^2 \calW(A,A).
\]
The lower bound in \eqref{eq:bound_calW} implies in particular that
\beq
\bbW(A) \gtrsim |A + A^T|^2.
\label{eq:D2calW_bound}
\eeq

We now present some key classes of elastic bodies:
\begin{enumerate}
\item \Emph{A body with (finitely-many) dislocations} is a body in which the implant map $\P$ is closed, i.e., $d\P =0$. 
This implies that we have a well-defined, global, frame that locally looks like the (undistorted) standard frame in Euclidean space.
This is the focus of this paper, and discussed in detail below. 

\item In the case there are disclinations, one can also locally define a frame field that looks like the standard frame in Euclidean space. 
However, due to the curvature charge of the disclinations, this frame field cannot be defined globally; thus \Emph{a body containing both dislocations and disclinations} is a slight generalization of the above model: 
A body $\M$, a finite open cover $\M_i$, and a collection of sections $\P_i$ of $T\M_i\otimes \R^d$ (frame fields on each patch $\M_i$), satisfying $d\P_i=0$.
In order for the energy \eqref{eq:E(f)} to be well defined, the maps $\P_i$ need to be compatible, namely that, for each $i,j$, $\P_j \circ \P_i^{-1}: \M_i\cap \M_j \to \R^{d\times d}$ obtains values in the isotropy group $\Gamma < \SO(d)$ of $\calW$
\[
\Gamma = \{U \in \R^{d\times d} ~:~ \calW(AU) = \calW(A)\,\, \text{for all}\,\, A\in \R^{d\times d}\}.
\]
This model is described in detail in the follow-up work \cite{Mao24}.

\item Given a $2$-dimensional body $(\M,\P)$, one can obtain its \Emph{associated plate model} by considering the $3$-dimensional body $(\M\times (-t/2,t/2), \P \oplus dx^3 \otimes \pl_3)$ for a small thickness parameter $t>0$.
This corresponds to the metric $G = \brk{\begin{matrix} \g & 0 \\ 0 & 1 \end{matrix}}$, hence the manifold $(\M,\g)$ has zero second fundamental form in $(\M\times (-t/2,t/2), G)$, justifying the term "plate model".

A \Emph{reduced plate model} can be obtained from $(\M,\P)$, by considering configurations to be embeddings $f:\M \to \R^3$, an energy density $\calW:\R^3 \otimes \R^2 \to \R$, and the energy
\[
E^{\text{plate}}(f) = \int_\M \calW (df \circ \P^{-1})\, \Volume + t^2 \int_\M |\nabla n_f|^2 \, \Volume,
\]
where $n_f(p)$ is the unit normal to $f(\M)$ at $p$.
This model is a formal expansion of the three-dimensional associated plate model described above \cite{ESK09a}.

\item This setting also includes the so-called \Emph{incompatible} or \Emph{non-Euclidean elasticity} model, in which $\P$ is sometimes called the \Emph{pre-strain} and its associated metric $\g$ the \Emph{reference metric}.
In this setting the main interest is often the case $d\P\ne 0$; in this case the associated metric $\g$ is non-flat, and its curvature is a source of incompatibility.
If the body is isotropic, then the energy can be written in terms of $\g$ alone and thus the choice of $\P$ is often omitted in this case.
This model has been thoroughly studied in recent years, see e.g., \cite{ESK09a,LP11,KMS15,Lew23} and the references therein. 
\end{enumerate}

\section{The geometry of an edge-dislocation}
\label{sec:single_dislocation}

In this section we introduce two-dimensional elastic bodies modeling cross-sections of bodies with straight edge-dislocations. In Subsection~\ref{subsec:coord_free} we provide an intrinsic, coordinate-free definition of a body with an edge-dislocation.
In Subsection~\ref{sec:coordinate_construction} we construct such bodies using polar coordinates; this construction is useful for subsequent calculations, which are more easily carried out in coordinates. In Subsection~\ref{sec:dev_from_euc}, we quantify the ``defectiveness" of such bodies by a geometric comparison with defect-free bodies.

\subsection{Coordinate-free construction of dislocated bodies}
\label{subsec:coord_free}

An \Emph{edge-dislocation} is a material defect, in which a perfect lattice structure is perturbed by the presence of an extra half-plane, whose boundary is called a \Emph{dislocation line}. This extra half-plane is usually created by a gliding mechanism, as described in the introduction.  

A continuum-mechanical viewpoint of crystalline defects was classified by Volterra by means of cut-and-weld protocols. Geometrically, Volterra's procedures yield Riemannian manifolds, which are \emph{locally} Euclidean (i.e., locally isometrically-embeddable in the ambient space) with dislocations encoded in the topology and the global metric of the manifold.  

A three-dimensional body with a single (straight) edge-dislocation has an axial symmetry, and can therefore be  described by a two-dimensional cross-section. We define a body with an edge-dislocation as follows: 

\begin{definition}
\label{def:single_disloc}
\Emph{A body with an edge-dislocation} having Burgers vector $\v$ is a two-dimensional elastic body $(\M,\P)$ satisfying the following additional properties:
\begin{enumerate}[itemsep=0pt,parsep=0pt,label=(\alph*)]
\item $\M$ is diffeomorphic to $\R^2\setminus B_1$.
\item The implant map $\P$ is closed, $d\P=0$.
\item 
For every positively-oriented loop $C$ homotopic to $\dM$, 
\beq
\oint_C \P = \v,
\label{eq:def_burgers}
\eeq
where $\v\in\R^2$ is interpreted as a \Emph{Burgers vector}. 
\item The boundary $\dM$ has winding number 1.
\end{enumerate}
\end{definition}

The first assumption imposes an annular topology, where the hole represents the ``core'' of the dislocation. 
The second assumption amounts to stating that the body is locally devoid of defects. 
The third assumption asserts that the core (which is not part of the body) contains a defect of dislocation-type. 
The fourth assumption implies that being homotopic to the boundary amounts to ``encircling the core'' exactly once: 
if $\dM$ is of regularity $C^{1,1}$, and $\gamma(s)$ is a unit length oriented parametrization of $\dM$ with geodesic curvature $\kappa$, the winding number condition can be defined as the condition $\frac{1}{2\pi} \int_\gamma \kappa(s)\,ds = 1$.
We elaborate on this interpretation in the following list of comments:

\begin{enumerate}
\item The metric induced by a closed implant map is locally-Euclidean:
Indeed, a closed $\P$ is locally-exact. It follows that every point $p\in\M$ has a neighborhood $p\in U\subset\M$ and a map $f: U\to\R^2$, such that $\P|_U=df$. 
Thus, within $U$, the metric $\g$ induced by the implant map $\P$ is given by
\[
\g|_U(u,v) = \ip{df(u),df(v)},
\] 
i.e., $\g|_U$ equals the pullback of the Euclidean metric by $f$, which implies that $\g$ is locally-Euclidean, which we interpret as $\M$ being locally defect-free. 

\item 
A global implant map also induces a global notion of parallelism, or equivalently, a path-independent parallel transport map, 
$\Pi^p_q:T_q\M\to T_p\M$ given by 
\beq
\Pi^p_q = \P^{-1}_p \P_q.
\label{eq:Pi_pq}
\eeq
By construction, $\P$ is parallel with respect to the parallel transport it induces.
Denoting by $\nabla^\g$ the Riemannian connection of $\g$ and by $\nabla^\euc$ the Euclidean connection in $\R^2$, it follows from the previous item that locally $\nabla^\g \P|_U = \nabla^{f^\#\euc} df = df(\nabla^\euc \id) = 0$, i.e., $\P$ is parallel with respect to the Riemannian connection of $\g$, namely, the parallel transport induced by $\P$ coincides with the parallel transport induced by $\g$.

\item 
Any section of the form $U\P$ for some $U\in\SO(2)$ is parallel with respect to the above parallel-transport.
These are the only parallel, orientation-preserving maps that are isometries from $(T\M,\g)$ to the Euclidean space. 
In this sense, $\g$ carries all the information about $\P$, up to a global choice of rotation.

\item 
The integration \eqref{eq:def_burgers} is the integral of an $\R^2$-valued 1-form over a one-dimensional curve. By definition, if $\gamma:I\to\M$ is a positively-oriented parametrization of $C$, i.e., counter-clockwise, then
\[
\oint_C \P = \int_I \P_{\gamma(t)}(\dot{\gamma}(t))\, dt.
\]
The fact that this integral only depends on the homotopy class of $C$ results from $\P$ being closed. 

\item
For every parametrized, non-contractible loop $\gamma:I\to\M$, 
\[
|\v| = \Abs{\int_I \P_{\gamma(t)}(\dot{\gamma}(t))\, dt} \le 
\int_I |\P_{\gamma(t)}(\dot{\gamma}(t))|\, dt = 
\int_I |\dot{\gamma}(t)|_\g\, dt,
\]
where the last equality follows from $\P$ being, by definition, an isometry.
The right-hand side is the length of the loop. It follows that every loop surrounding the dislocation, and in particular $\dM$, has a length bounded from below by the magnitude of the Burgers vector. In other words, in this geometrically-nonlinear setting, as in the atomistic viewpoint, \emph{there is no such thing as a point dislocation}---the size of the core is bounded from below by the magnitude of the dislocation (as described in \cite[p.~365]{KM15}, the size of the core cannot be shrunk below a segment of length $|\v|/2$). This is consistent with the discrete picture, where the region in which  the lattice structure is imperfect cannot be smaller than the Burgers vector.

The fact that there is not such a thing as a point dislocation implies that there is no intrinsic meaning to the distance of a point $p\in \M$ from ``the dislocation''. 
In this work we define the distance between $p$ and the dislocation by
\beq\label{eq:dist_disloc}
\r(p) = \dist(p,\dM) + |\v|.
\eeq

\item
Condition  \eqref{eq:def_burgers} on the circulation of $\P$ can be replaced by an equivalent condition: let $p\in\M$ be an arbitrary reference point and denote by $\Pi^p\in\W^1(\M;T_p\M)$ the $T_p\M$-valued 1-form whose value at $q\in\M$ is $\Pi^p_q$ ($\Pi^p$ translates tangent vectors to the point $p$).  Using \eqref{eq:Pi_pq},
\beq
\oint_C \Pi^p = \oint_C \P^{-1}_p \P =  \P^{-1}_p \oint_C  \P = \P_p^{-1}(\v).
\label{eq:def_burgers2}
\eeq
The right-hand side is a tangent vector at $p$, and can be viewed as the value at $p$ of the parallel vector field $\b = \P^{-1}(\v) \in\Gamma(T\M)$ (see \cite{KMS15} for detail). 

\item Condition  \eqref{eq:def_burgers} can be replaced by yet another equivalent condition: 
Consider the space of continuous, compactly-supported functions $C_c(\M;\R^2)$ (the way we defined $\M$ as a manifold with a boundary, they need not vanish in a 
neighborhood of $\dM \subset\M$). 
Define the bounded linear functional $\torsion:C_c(\M;\R^2) \to \R$,
\[
\torsion(\psi) = \int_{\dM}  \TR  (\P\otimes \psi),
\]
where $\P\otimes \psi$ and its trace were defined in the Notations section. For every $\psi\in C_c(\M;\R^2)$ satisfying $\psi|_{\dM}=\u\in \R^2$,
\[
\torsion(\psi) =  \TR \brk{\int_{\dM}    \P \otimes \u} =
\TR(\v\otimes\u) = \ip{\v,\u}.
\]
If $\psi\in \HOneZero(\M;\R^2)$, since $\P$ is closed,
\[
\torsion(\psi) = - \int_{\M}  d\TR (\P\otimes\psi) = \int_{\M}  \TR(\P\wedge d\psi),
\]
where the change of sign in the first equality is due to $\dM$ being an inner boundary.
The distribution $\torsion$ (which can be viewed as a de-Rham $0$-current, as a measure in $\calM(\M;\R^2)$, or as an element of $H^{-1}(\M;\R^2)$) was studied in \cite{KO20} and identified as encoding the torsion of the connection induced by the parallel implant map $\P$.

\item
Dislocations are often ``quantized'', due to an underlying lattice structure; the Burgers vector $\v$ can only assume certain magnitudes and directions. This fact is encompassed in the following definition:

\begin{definition}
\label{def:disloc_structure}
Let $S\subset \R^2$ be a basis.
The set $\bbS = \operatorname{span}_{\mathbb{Z}} S\setminus \{0\}$ is called \Emph{a dislocation structure}.
A collection of bodies with edge-dislocations is said to have \Emph{a dislocation structure $\bbS$ at length-scale $\e$} if the corresponding collection of Burgers vectors is a subset of $\e\bbS$. 
\end{definition}

\item
At this point we did not make any assumption on the size of the body. We will soon consider bodies with dislocations,  having finite diameter (and hence finite volume).
If $\M$ is such a body, then there exists a compact manifold with boundary $\bar{\M}$, diffeomorphic to $\bar{B}_2\setminus B_1$, such that $\bar{\M}\setminus \M$ is the outer-boundary of $\bar{\M}$.
We refer to the set $\bar{\M}\setminus \M$ as the \Emph{outer-boundary} of $\M$, and typically denote it by $\Gamma$ and assume that it is Lipschitz.
This should not be confused with $\pl \M$, which is the \Emph{inner boundary}.

\item The reason for the winding number assumption is to exclude, say, a double cover of a Euclidean annulus, which would have a zero Burgers vector but cannot be isometrically embedded in the Euclidean plane (in fact, it can be viewed as a body containing a disclination of magnitude $-2\pi$). See also the proof of the uniqueness theorem below.

\item
Finally, the annular topology can be replaced with a simply-connected topology, with a dislocation core having a geometry which is not locally-Euclidean (representing a region where the lattice structure is defective, for example, containing a so-called 5-7 pair in an hexagonal lattice). 
In this case, the winding number condition has to be replaced with the condition that the total Gaussian curvature in the dislocation core vanishes. 
This can be seen as a different kind of regularization of the core, and is expected to lead to similar results.
\end{enumerate}

The following theorem asserts that
\defref{def:single_disloc} defines a body manifold uniquely in the following sense:

\begin{theorem}
\label{thm:disloc_unique}
Let $(\M,\P)$ and $(\M_1,\P_1)$ be metrically-complete bodies with edge-dislocations (\defref{def:single_disloc}) having identical Burgers vectors $\v$. Then, there exist annular submanifolds $\M'\subset\M$ and $\M_1'\subset \M_1$ with $\Vol_\P(\M\setminus\M')<\infty$ and $\Vol_{\P_1}(\M_1\setminus\M_1')<\infty$, such that $(\M',\P)$ and $(\M_1',\P_1)$ are isometric: there exists a diffeomorphism $f:\M'\to\M_1'$ such that $\P  = f^\#\P_1$, that is, $(\P_1)_{f(p)} \circ df_p = \P_p$.
\end{theorem}

The idea of the proof is as follows:
If $\dM$ is convex, in the sense that the shortest path in $\M$ connecting any two point on $\dM$ lies in $\dM$, we show that $(\M,\P)$ can be obtained by a Volterra cut-and-weld procedure in $\R^2\setminus D$ for some convex set $D$ (this is not necessarily true if the boundary is not convex, even if we allow $D$ to be non-convex).
Thus, two such manifolds are isometric if they are obtained by the same cut-and-weld procedure from the same set $\R\setminus D$.
The cut-and-weld procedure is completely determined by $\v$, which is the same for both manifolds; 
by enlarging the cores, i.e., by taking $\M'\subset \M$ and $\M_1'\subset \M_1$, we can make the corresponding sets $D$ the same.
In order to follow this strategy, we first need to study the geometry of $(\M,\P)$, and show that we can remove from $\M$ a compact set, resulting in a body with an edge-dislocation having the same Burgers vectors and a convex inner boundary.
This is done in Lemmas~\ref{lem:geometry_dislocation}--\ref{lem:convex_hull2}, after which we prove the theorem.

In the following, a \Emph{geodesic} $\gamma$ is a simple curve which is locally length minimizing; if its endpoints are $p$ and $q$ and $d(p,q) = \text{len}(\gamma)$ we say that $\gamma$ is a \Emph{minimizing geodesic}, or a \Emph{segment} (here $d$ stands for the distance induced by $\P$, and $\text{len}(\gamma)$ is the length of $\gamma$).
A geodesic that does not intersect $\dM$ at more than one point is a geodesic in the usual Riemannian sense (i.e., its tangent vector field is parallel), and a general geodesic is a concatenation of such curves and a simple open curve in $\dM$.

\begin{lemma}\label{lem:geometry_dislocation}
Let $(\M,\P)$ be as in Theorem~\ref{thm:disloc_unique}.
Then it is metrically unbounded, and every maximal geodesic $\eta\subset \M$ that does not intersect $\dM$ (which we call a \Emph{line}) splits $\M$ into two complete manifolds with boundary $\M_\eta^{\pm}$, with the following properties: 
\begin{enumerate}[itemsep=0pt,label=(\alph*)]
\item $\M_\eta^+\cup \M_\eta^- = \M$ and $\M_\eta^+\cap \M_\eta^- = \eta$.
\item $\M_\eta^+$ is isometric to a half plane.
\item $\M_\eta^-$ contains $\dM$ and is geodesically-convex, i.e., every geodesic (minimizing or not) in $\M$ connecting two points in $\M_\eta^-$ lies in $\M_\eta^-$.
\end{enumerate}
\end{lemma}

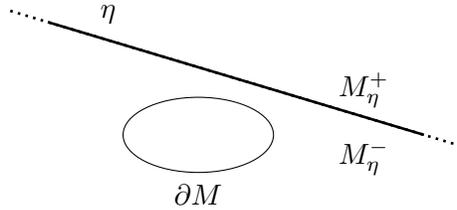
\begin{figure}[H]
\[
\btkz
\draw (0,0) ellipse (1cm and 0.5cm);
\draw [line width = 1pt](-2,1.5) -- (3,0);
\draw [dotted, line width = 1pt] (-2.5,1.65) -- (3.5,-0.15);
\tkzText(-1.2,1.6){$\eta$}
\tkzText(0,-0.8){$\dM$}
\tkzText(2.2,0.6){$\M_\eta^+$}
\tkzText(2.2,-0.3){$\M_\eta^-$}
\etkz
\]
\caption{The setting of \lemref{lem:geometry_dislocation}.}
\end{figure}

\begin{proof}
Since $(\M,\P)$ is a complete manifold, every closed bounded subset of $\M$ is compact.
As $\M$ is homeomorphic to $\R^2\setminus B_1(0)$, i.e., non-compact, it follows that $(\M,\P)$ is unbounded.
This homeomorphism also induces a global coordinate system on $\M$.
In this coordinate system, the line $\eta$ is a simple, open curve without boundary that does not intersect $\overline{B_1(0)}$, and thus splits $\M$ into two complete manifolds with boundary, one of which (denoted by $\M_\eta^-$) contains $\dM$.

The manifold $\M_\eta^+$ is a simply-connected complete, locally-Euclidean, smooth manifold with boundary, whose boundary is a geodesic line. It is thus isometric to a half plane (this follows, for example, by doubling it, obtaining a complete, simply-connected, flat two-dimensional manifold without boundary, which must be a plane by the uniqueness of constant-curvature complete simply-connected surfaces).
In particular, between any two points in $(\M_\eta^+,\P)$ there exists a unique geodesic in $\M_\eta^+$ (although it might not be the only connecting geodesic in $(\M,\P)$).
This implies that $\M_\eta^-$ is geodesically-convex in $\M$: 
Indeed, let $p,q\in \M_\eta^-$, and suppose that $\gamma$ is a geodesic in $\M$ (either minimal or not) connecting $p$ and $q$ and intersecting $\M_\eta^+$.
Then the endpoints of any connected component of the geodesic in $\M_\eta^+$ is a geodesic whose endpoints are in $\eta$, and therefore, by the uniqueness of geodesics in $\M_\eta^+$, a subset of $\eta$, and thus contained in $\M_\eta^-$.
\end{proof}

\begin{lemma}\label{lem:convex_hull}
Let $(\M,\P)$ be as in Theorem~\ref{thm:disloc_unique}.
Then the convex hull $K$ of $\dM$ is bounded. 
\end{lemma}

Here, by a convex hull, we mean the intersection of all sets $C\subset \M$ containing $\dM$, such that any geodesic between two points in $C$ is contained in $C$.
This is not the standard definition in Riemannian geometry (in which the assumption is only on minimal geodesics), but makes the proofs below slightly easier, and is sufficient for the proof of Theorem~\ref{thm:disloc_unique}.
This claim is not trivial, as there exist complete metrics on $\R^2$, such that the convex hull of a compact set is the whole of $\R^2$.

\begin{proof}
Denote by $\ell$ the length of $\dM$, let $p\in \dM$, and consider the metric ball $B=B_{100\ell}(p)\subset \M$.
We will show that $K\subset \bar{B}$.
Let $p_0\in \pl B$, and let $\gamma$ be a minimizing geodesic from $p_0$ to $p$.
Let $\eta$ be a complete geodesic at $p_0$, in a direction perpendicular to $\gamma$.
We now show that $\eta$ is a line, i.e., that it does not intersect $\dM$.
Let $q\in \dM$, and let $\sigma$ be a geodesic (minimizing or not) from $p_0$ to $q$.
Let $p' \in \gamma$ be the first intersection of $\gamma$ and $\dM$, and let $q' \in \sigma$ be the first intersection of $\sigma$ and $\dM$.
Let $\gamma'$ and $\sigma'$ be the parts of $\gamma$ and $\sigma$ starting at $p_0$ and ending at $p'$ and $q'$, respectively.
There exists a curve $c\subset \dM$ connecting $p'$ and $q'$ such that the domain $D\subset \M$ enclosed by $\gamma'$, $\sigma'$ and $c$ is simply-connected (see figure below). 

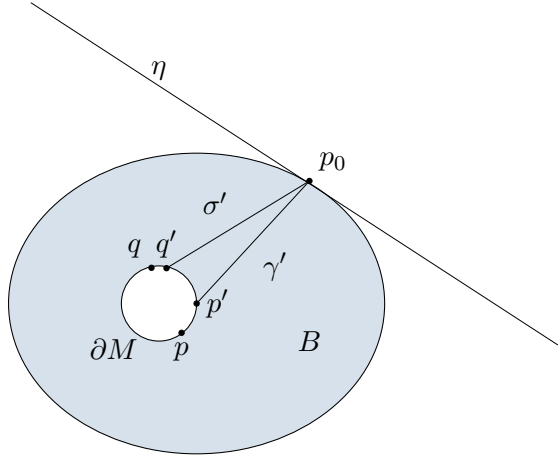
\begin{figure}[H]
\begin{center}
\btkz
\draw [fill=warmblue!20] (0.5,0) ellipse (2.5cm and 2cm);
\draw [fill=white] (0,0) circle (0.5cm);
\tkzDefPoint(0,0.){o};
\tkzDefPoint(2,1.63){p0};
\tkzDefPoint(0.5,0){p};
\tkzDefPoint(0.1,0.47){q};
\tkzDefPoint(0.3,-0.39){pp};
\tkzDefPoint(-0.1,0.475){qq};
\tkzDefPoint(0.3,2.7){q1};
\tkzDefPoint(2.3,1.4){q2};
\tkzDrawPoint(p0);
\tkzDrawPoint(p);
\tkzDrawPoint(q);
\tkzDrawPoint(pp);
\tkzDrawPoint(qq);
\tkzLabelPoint[right](p){$p'$};
\tkzLabelPoint[below](pp){$p$};
\tkzLabelPoint[above right](p0){$p_0$};
\tkzLabelPoint[above](q){$q'$};
\tkzLabelPoint[above left](qq){$q$};
\tkzDrawSegment(p0,p);
\tkzDrawSegment(p0,q);
\tkzLabelSegment[below right](p0,p){$\gamma'$};
\tkzLabelSegment[above left](p0,q){$\sigma'$};
\tkzText(-0.6,-0.6){$\dM$};
\tkzText(2,-0.5){$B$};
\tkzDrawLine[add=1 and 1.5](q1,q2);
\tkzText(0.,3.1){$\eta$};
\etkz
\end{center}
\caption{Notations used in the proof of \lemref{lem:convex_hull}.}
\label{fig:4}
\end{figure}

Note that 
\[
\begin{split}
\text{len}(c) &\le \text{len}(\dM) = \ell,\\
\text{len}(\gamma') &= d(p_0,p') \ge d(p_0,p) - d(p,p') \ge  99\ell ,\\
\text{len}(\sigma')&\ge d(p_0,q') \ge d(p_0,p) - d(p,q') \ge 99\ell.
\end{split}
\]

Since $(D,\P)$ is a simply-connected, locally-flat manifold, it can be immersed isometrically in the Euclidean plane.
The image of $\gamma'$ and $\sigma'$ under this isometric immersion are straight lines of length $\ge 99\ell$, whereas the image of $c$ is a curve of length $\le \ell$.
It follows that the angle between these straight lines, and therefore also between $\gamma'$ and $\sigma'$ is less than $\pi/2$. In particular, $\sigma'$ is not $\eta$, from which we conclude that $\eta\cap \dM = \emptyset$.

It follows that $\dM\subset \M_\eta^-$, which by Lemma~\ref{lem:geometry_dislocation} is a geodesically-convex set.
Thus, the convex hull $K$, being the intersection of all geodesically-convex sets containing $\dM$, is a subset of $\M_\eta^-$.

That fact that $K\subset\M_\eta^-$ for every line $\eta$ constructed this way implies that $K\subset \bar{B}$, hence bounded.
Indeed, assume that $p_1 \in \M\setminus \bar{B}$.
Then, there exists a minimizing geodesic $\gamma$ connecting $p$ and $p_1$, of length $r>100\ell$.
After time $100\ell$, this geodesic intersects some $p_0\in \pl B$; construct the perpendicular geodesic $\eta$ as before.
By construction, the part of $\gamma$ connecting $p_0$ and $p_1$ is in $\M_\eta^+$, and intersects $\eta$ only at $p_0$.
Since $K\subset \M_\eta^-$, it follows that $p_1\notin K$.
\end{proof}

\begin{lemma}\label{lem:convex_hull2}
Let $(\M,\P)$ be as in Theorem~\ref{thm:disloc_unique}.
Let $K$ be the convex hull of $\dM$. 
Then $(\overline{\M\setminus K},\P)$ is a body with an edge-dislocation according to \defref{def:single_disloc}, whose (possibly only Lipschitz continuous) boundary is convex. 
\end{lemma}

\begin{proof}
Identify $\M$ with $\R^2\setminus B_1(0)$ under the coordinate system mentioned in Lemma~\ref{lem:geometry_dislocation}.
It is sufficient to prove that $K\cup B_1(0)$ is a connected, simply-connected domain, since then the (topological) boundary $\pl (K\cup B_1(0))$ is a curve homotopically equivalent to $\dM = \pl B_1(0)$, and it coincides with the (manifold) boundary of $\overline{\M\setminus K}$.
The fact that in that case the boundary of $\overline{\M\setminus K}$ is Lipschitz continuous follows from the regularity of convex curves in the plane, as $\M$ is locally-Euclidean, and Lipschitz continuity over a compact set is a local property.

The fact that $K\cup B_1(0)$ is connected follows from the facts that $B_1(0)$ is connected, that $K$ is connected as a geodesically-convex set, and that $\pl B_1(0)\subset K$.

Assume by contradiction that $K\cup B_1(0)\subset \R^2$ is not simply-connected; this implies that its complement is not connected, hence there exists a bounded connected component $\W\subset (K\cup B_1(0))^c\subset \M$.
Let $p\in \W$, let $v\in T_p\M$, and consider the map $\gamma(t)= \exp_p(tv)$.
By the completeness of $\M$, it is defined for all $t\in \R$, unless $\gamma(t_0)\in \dM\subset K$ for some $t_0\in \R$.
The boundedness of $\W$, and the fact that $\W\subset \M\setminus K$, imply that there exist $t_1<0<t_2$ such that $\gamma$ is well-defined on $[t_1,t_2]$ and $\gamma(t_1),\gamma(t_2) \in K$.
However, this is a contradiction to the convexity of $K$, as $\gamma|_{[t_1,t_2]}$ is a geodesic between points in $K$ that goes through the point $p\notin K$ (here we used the definition of $K$ as containing all geodesics between points, as we do not know a priori that $\gamma|_{[t_1,t_2]}$ is a minimizing geodesic).
\end{proof}

We now prove Theorem~\ref{thm:disloc_unique}.
\begin{proof}
Consider $(\M,\P)$.
Since we allow removing bounded neighborhoods of the boundary of $\M$, we can, using Lemma~\ref{lem:convex_hull2}, assume that $\dM$ is geodesically-convex.
The boundary of $\M$ then may only be Lipschitz-continuous; by further removing an $\e$-neighborhood of it we can obtain a $C^{1,1}$ boundary, which is at least locally-convex:
Indeed, since the regularity of the boundary and being locally-convex are local properties, and since $\M$ is locally-Euclidean, this follows from the fact that an $\e$-neighborhood of a convex set in the plane is a convex set with a $C^{1,1}$ boundary \cite{Kis92}.

Thus, assume henceforth that $(\M,\P)$ has a locally-convex $C^{1,1}$ boundary, and let $\gamma:[0,\ell]\to \dM$ be an oriented arclength parametrization of $\dM$. 
Then, $\P\circ \dot{\gamma }:[0,\ell]\to \R^2$ is a closed $C^{1,1}$ curve. 
In particular, there exists an $s_0\in [0,\ell]$ such that $\v = - |\v| \P\circ \dot{\gamma}(s_0)$.
Without loss of generality, we take $s_0=0$ and denote $p = \gamma(0)$.

By Comment 6 following \defref{def:single_disloc}, the Burgers vector $\v$ induces via the implant map $\P$ a parallel vector field $\b = \P^{-1}(\v) \in\frakX(\M)$, which implies that $\b_{p} = -|\v|\dot{\gamma}(0)$.

Denote by $\b^\perp$ the parallel vector field such that the basis $(-\dot{\gamma}(0) ,\b^\perp_p)$ is orthonormal and oriented (by parallelism, $\b^\perp$ is orthogonal to $\b$ everywhere).
Since $\dM$ is an inner boundary and $\gamma$ is oriented, it follows that $\b^\perp_p$ points into $\M$.
Let $C$ be the geodesic ray emanating from $p$, in direction $\b^\perp_p$.
By the convexity of the boundary, $C$ does not intersect $\dM$ and extends indefinitely.

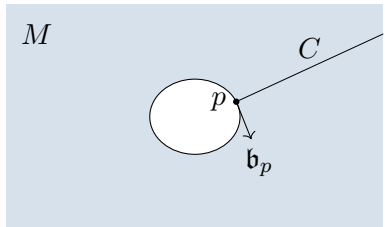
\begin{figure}[H]
\begin{center}
\btkz
\fill[color=warmblue!20] (-2,-1.5) -- (3,-1.5) -- (3,1.5) -- (-2,1.5) -- cycle;
\draw [fill=white!20] (0.5,0) ellipse (0.6cm and 0.5cm);
\tkzDefPoint(1.05,0.2){p};
\tkzDefPoint(3.0,1.1){q};
\tkzDrawPoint(p);
\tkzLabelPoint[left](p){$p$};
\tkzDrawSegment(p,q);
\draw[->] (1.05,0.2) -- (1.25,-0.3); 
\tkzLabelSegment[above](p,q){$C$};
\tkzText(1.35,-0.6){$\b_p$};
\tkzText(-1.6,1.1){$\M$};
\etkz
\end{center}
\caption{Cutting $\M$ along the ray perpendicular to the Burgers vector.}
\end{figure}

Next, ``cut" $\M$ along $C$; denote by $C_1$ and $C_2$ the two connected component of  $\pl(\M\setminus C)\setminus \pl\M$.
Define a map $f: \M\setminus C\to\R^2$ by
\[
f(q) = \int_{\gamma_{p,q}} \P,
\]
where $\gamma_{p,q}$ is a path in $\M\setminus C$ connecting $p$ and $q$. 
Since $\M\setminus C$ is simply-connected and $\P$ is closed, the integral only depends on the end points. 
A direct calculation shows that $df_q = \P_q$, i.e., $f$ is an isometric immersion. 
Furthermore, every point $q\in C$ can be identified with two points $q_1,q_2$ on the boundary of $\M\setminus C$. 
If paths connecting $q_1$ and $q_2$ are positively-oriented (as loops in $\M$), then
\[
f(q_2) - f(q_1) = \oint \P = \v.
\]
In particular, applying this to $p=\gamma(0)=\gamma(\ell)$ we obtain that $\sigma = f\circ \gamma:[0,\ell] \to \R^2$ is a $C^{1,1}$ path with $\sigma(0) = 0$ and $\sigma(\ell) = \v$.
If $\v= 0$, then $\sigma$ is closed, and the rest of the argument is similar but simpler; we assume from now that $\v\ne 0$.

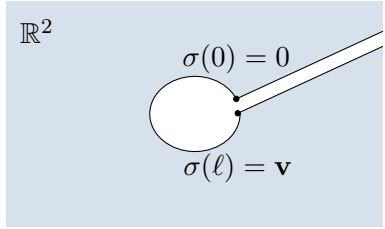
\begin{figure}[H]
\begin{center}
\btkz
\fill[color=warmblue!20] (-2,-1.5) -- (3,-1.5) -- (3,1.5) -- (-2,1.5) -- cycle;
\draw [fill=white!20] (0.5,0) ellipse (0.6cm and 0.5cm);
\fill[color=white!20] (1.05,0.2) -- (3.0,1.1) -- (3.0,0.9) -- (1.05,0.) -- cycle;
\draw (1.05,0.2) -- (3.0,1.1);
\draw (1.05,0.) -- (3.0,0.9);
\tkzDefPoint(1.05,0.2){p};
\tkzDefPoint(1.07,0.009){q};
\tkzDrawPoint(p);
\tkzDrawPoint(q);
\tkzLabelPoint[above=0.2cm](p){$\sigma(0)=0$};
\tkzLabelPoint[below=0.4cm](q){$\sigma(\ell)=\v$};
\tkzText(-1.6,1.1){$\R^2$};
\etkz
\end{center}
\caption{Image of the maps $f$ and $\sigma$.}
\end{figure}

Let $\kappa$ be the geodesic curvature of $\gamma$; it is defined almost everywhere, and satisfies $\int_0^\ell \kappa(s)\,ds= 2\pi$ by the winding number assumption (see the explanation of the winding number condition below Definition~\ref{def:single_disloc}).
The local-convexity of $\gamma$ implies that $\kappa\ge 0$.
Since $f$ is an isometric immersion, $\kappa$ is also the geodesic curvature of $\sigma$.
Since 
\[
\dot\sigma(0) = \dot\sigma(\ell) = \P_p(\dot \gamma(0)) = - \frac{\v}{|\v|},
\]
we can extend $\sigma$ to $\sigma:[0,\ell +|\v|] \to \R^2$ by
\[
\sigma(\ell + s) = \brk{1-\frac{s}{|\v|}}\v,
\]
for $s\in (0,|\v|]$, so that $\sigma$ is a $C^{1,1}$ closed curve with non-negative geodesic curvature that sums to $2\pi$.
Thus $\sigma$ is a simple curve that encloses a convex domain $\W\subset \R^2$.

Denote by $R\subset \R^2$ the closed domain bounded by the segment $[0,\v]$ and the rays $f(C_1) = \{t \v_\perp ~:~ t\in [0,\infty)\}$ and $f(C_2) = \{\v + t \v_\perp ~:~ t\in [0,\infty)\}$, where $\v_\perp = \P(\b_\perp)$ (in particular, $(\v,\v_\perp)$ is a positive orthogonal basis).
We now show that the image of $f:\M\setminus C \to \R^2$ is $\R^2 \setminus (\W\cup R)$, and that $f:\M\setminus C \to \R^2 \setminus (\W\cup R)$ is an isometry.
We then construct $\M$ by gluing the two rays $f(C_1), f(C_2)$ in $\R^2 \setminus (\W\cup R)^o$, which will complete the proof.

Let $q\in \M\setminus C$. 
Since $(\M,\P)$ is complete, $q$ can be connected by a minimal geodesic $\alpha$ to $\dM$.
Since $\dM$ is $C^1$, $\alpha$ intersects $\dM$ perpendicularly, at some point $\gamma(s)$ for $s\in (0,\ell)$.
Parametrize $\alpha$ so that $\alpha (0) = \gamma(s)$.
Then $t\mapsto f(\alpha(t))$ is a straight line in $\R^2$, starting at $\sigma(s)$ and perpendicular to $\W$.
Since $\W$ is convex and $s\ne (0,\ell)$, this straight line does not intersect $f(C_1)$ and $f(C_2)$.
It follows that $\alpha$ does not intersect $C$, and that $f(q)$ is indeed in $\R^2 \setminus (\W\cup R)$.
The fact that $f:\M\setminus C \to \R^2 \setminus (\W\cup R)$ is a bijection now follows by a similar argument, using the convexity of $\W$.
Since $f$ is an isometric immersion, it follows that it is an isometry.

Construct a manifold $N$ by gluing $t \v_\perp$ and $\v + t \v_\perp$ in $\R^2 \setminus (\W\cup R)^o$ for each $t\in [0,\infty)$.
Since these rays are parallel and are perpendicular to $\sigma$ at $0$ and at $\v$, it is a smooth manifold with a $C^{1,1}$ boundary.
Extend $f$ as a map $\M\to N$ by defining $f(C(t)) = t\v_\perp \sim \v + t\v_\perp$, where $C(t)$ is an arclength parametrization of $C$.
From the construction, it follows that $f$ is an isometry.

The same construction can be done for $(\M_1,\P_1)$, obtaining a manifold $N_1$ by gluing the rays in $\R^2 \setminus (\W_1\cup R)^o$ for some convex $\W_1$ with $[0,\v]\subset \pl\W_1$; the only difference is that $\W_1$ is not necessarily $\W$.
Take a convex set $\tilde{\W}\supset \W, \W_1$ such that $[0,\v]\subset \pl \tilde{\W}$, and perform the same gluing on $\R^2 \setminus (\tilde{\W}\cup R)^o$, obtaining a manifold $\tilde{N}$ which is a submanifold of both $N$ and $N_1$ with $\Vol(N\setminus \tilde{N}),\Vol(N_1\setminus \tilde{N})<\infty$.
Since $N$ and $N_1$ are isometric to $\M$ and $\M_1$, respectively, the proof is complete.
\end{proof}

\subsection{Coordinate construction of body manifolds}
\label{sec:coordinate_construction}

In this section we construct a specific family of complete bodies with edge-dislocations $(\hM_\v,\hP_\v)$ having Burgers vector $\v$.
By \thmref{thm:disloc_unique}, every such body is essentially unique up to the precise form of its boundary.
We endow the manifold $\hM_\v$ with polar coordinates:
\[
\hM_\v = \{(r,\vp) ~:~ r\ge |\v|, \quad \vp\in S^1 \},
\]
and an implant map
\beq
\begin{split}
\hP_\v &= dx\otimes \pl_1 + dy\otimes \pl_2 + \frac{d\vp}{2\pi}\otimes \v \\
&=dx \otimes \brk{\pl_1 - \frac{y}{2\pi r^2} \v} + dy \otimes \brk{\pl_2 + \frac{x}{2\pi r^2} \v},
\end{split}
\label{eq:hQv}
\eeq
where $x = r\,\cos\vp$ and $y = r\,\sin\vp$.

A simple calculation shows that $\hP_\v$ is non-degenerate on $\{r \ge |\v|\}$, hence $(\hM_\v,\hP_\v)$ is an elastic body. 
By construction, it satisfies Assumptions (a) and (d) of \defref{def:single_disloc}.
Furthermore, since each of the summands in \eqref{eq:hQv} is closed, so is $\hP_\v$, hence Assumption (b) is satisfied as well.
Finally, a direct calculation shows that 
\[
\oint_{r=|\v|} \hP_\v = \v, 
\]
i.e.,  Assumption (c) is also satisfied, hence $(\hM_\v,\hP_\v)$ is a body with an edge-dislocation with Burgers vector $\v$ .

Writing $\v = v_1 \,\pl_1 + v_2 \,\pl_2$, 
the coframe
\beq
\begin{aligned}
\nu^1 &= dx + \frac{v_1}{2\pi} d\vp  = \cos\vp\, dr + \brk{-r\sin\vp + \frac{v_1}{2\pi}} d\vp \\
\nu^2 &= dy + \frac{v_2}{2\pi} d\vp = \sin\vp\, dr + \brk{r\cos\vp + \frac{v_2}{2\pi}} d\vp
\end{aligned}
\label{eq:nu12}
\eeq
is by construction orthonormal and parallel with respect to the metric $\hg_\v = \hP_\v^\#\euc$.
The latter is given explicitly by
\beq
\begin{split}
\hg_\v(r,\vp) = \hP_\v^\#\euc (r,\vp) 
&= dr\otimes dr \\
&\quad + \tfrac{1}{2\pi}(v_1 \cos \vp + v_2 \sin \vp) ( dr \otimes d\vp + d\vp \otimes dr) \\
&\quad + r^2 \brk{1 +  \frac{-v_1\,\sin\vp + v_2\,\cos\vp}{\pi r} + \frac{|\v|^2}{4\pi^2 r^2}}  d\vp \otimes d\vp.
\end{split}
\label{eq:hg}
\eeq

For future reference, we define the submanifold of finite diameter, 
\[
\hM_\v^R = \{(r,\vp) \in \hM_\v ~:~ r< R, \quad \vp\in S^1 \}.
\]

We proceed to derive some geometric properties of the manifold $(\hM_\v,\hP_\v)$.
Note that 
\[
(\cof1,\cof2) = (\cos\vp\,\nu^1 + \sin\vp\,\nu^2, -\sin\vp\,\nu^1 + \cos\vp\,\nu^2) 
\]
is also an orthonormal (but not parallel) coframe, which by \eqref{eq:nu12} satisfies
\beq
|\cof1 - dr|_{\hg_\v} \le \frac{|\v|}{2\pi r} |r\,d\vp|_{\hg_\v} 
\Textand
|\cof2 - r\,d\vp|_{\hg_\v} \le \frac{|\v|}{2\pi r} |r\,d\vp|_{\hg_\v},
\label{eq:bnd_for_volume}
\eeq
hence, by a straightforward calculation,
\beq
||dr|_{\hg_\v} - 1| \le \frac{|\v|}{2\pi r - |\v|}
\Textand
||r\,d\vp|_{\hg_\v} - 1| \le \frac{|\v|}{2\pi r - |\v|}.
\label{eq:bnd_dr_dvp}
\eeq

Even though the $r$-coordinate lines have unit speed, they are not geodesics of $\hg_\v$, and thus the coordinate $r$ does not coincide with the distance $\r$ to the dislocation as defined in \eqref{eq:dist_disloc}   (only on $\dM$, $r = \r = |\v|$). The following lemma estimates the discrepancy between the two:

\begin{lemma}
\label{lem:frakr=r}
In $(\hM_\v,\hP_\v)$,
the distance $\r(p)$ of a point $p = (r,\vp)$ to the dislocation as defined in \eqref{eq:dist_disloc} satisfies
\[
\brk{1 - \frac{1}{2\pi}} r +  \frac{1}{2\pi} |\v| \le \r(p) \le r.
\]
\end{lemma}

\begin{proof}
For the upper bound, the curve $\gamma(t) = (r-t,\vp)$ for $t\in[0,r-|\v|]$ connects $p$ to $\dM$. Then,
\[
\text{Length}(\gamma) = \int_0^{r-|\v|} |\pl_r|_{\hg_\v}\, dt = r - |\v|,
\]
where we used the fact that $\pl_r$, as evident from \eqref{eq:hg}, is a unit vector.  Hence,
\[
\r(p) \le \text{Length}(\gamma)  + |\v| = r.
\]
For the lower bound, let $\gamma(t)$, $t\in [0,1]$ be any curve connecting $p$ to $\dM$. Then,
\[
|\v| - r  = r(1) - r(0) = \int_\gamma dr.
\]
Using the fact that $r\ge |\v|$,
it follows from \eqref{eq:bnd_dr_dvp} that $|dr|_{\hg_\v} \le 1 + \frac{1}{2\pi-1}$, hence
\[
r - |\v| \le \Abs{\int_\gamma dr} \le \text{Length}(\gamma)  \brk{1 + \frac{1}{2\pi  - 1}},
\]
i.e.,
\[
\text{Length}(\gamma) \ge \frac{2\pi - 1}{2\pi} (r - |\v|).
\]
Taking the infimum over all such curves $\gamma$,
\[
\r(p) 
= \inf_\gamma \text{Length}(\gamma)  + |\v|
\ge \brk{1 - \frac{1}{2\pi}} r +  \frac{1}{2\pi} |\v|.
\]
\end{proof}

In the sequel, when considering a body with dislocation $(\M,\P)$, we will need to assume some geometric restrictions on the inner boundary $\dM$. 

\begin{definition}
\label{def:regular_dM}
A body with a dislocation $(\M,\P)$ with Burgers vector $\v$ is said to have a \Emph{regular inner boundary} if there is an annular neighborhood $A$ of the inner boundary such that 
\begin{enumerate}[itemsep=0pt,label=(\alph*)]
\item 
The following inclusion holds
\[
\{p\in\M ~:~ \dist(p,\dM) < |\v|\} \subset A.
\]
\item $(A,\P)$ can be embedded isometrically in $(\hM_\v^{4|\v|},\hP_\v)$.
\item $A$ is Lipschitz equivalent to the annulus $B_{2|\v|}\setminus B_{|\v|}\subset\R^2$, with bilipschitz constant 10.
\end{enumerate}

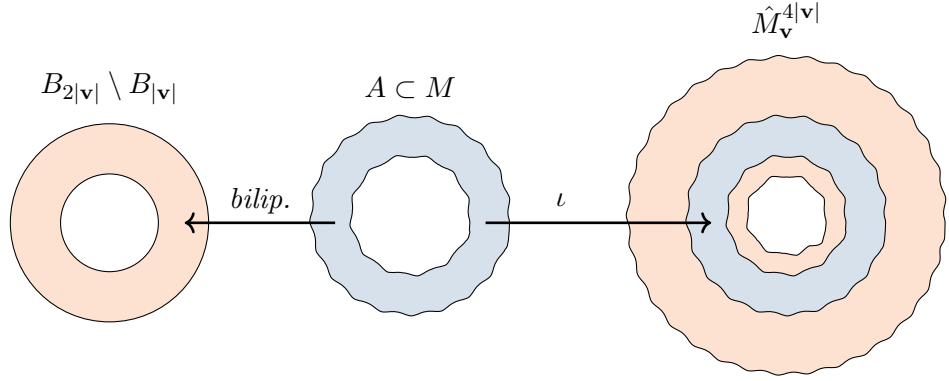
\begin{figure}[H]

\begin{center}
\btkz
\begin{scope}[xshift=0cm]
\node[circle,draw=black,fill=warmblue!20, minimum size=75, decorate,decoration={snake,amplitude=.2mm,segment length=4mm,post length=1mm}]  at (0,0) {};
\node[circle,draw=black,fill=white, minimum size=45, decorate,decoration={snake,amplitude=.2mm,segment length=4mm,post length=1mm}]  at (0,0) {};
\tkzText(0,1.8){$A\subset\M$};
\end{scope}

\begin{scope}[xshift=5cm]
\node[circle,draw=black,fill=ocre!20, minimum size=120, decorate,decoration={snake,amplitude=.2mm,segment length=4mm,post length=1mm}]  at (0,0) {};
\node[circle,draw=black,fill=warmblue!20, minimum size=75, decorate,decoration={snake,amplitude=.2mm,segment length=4mm,post length=1mm}]  at (0,0) {};
\node[circle,draw=black,fill=ocre!20, minimum size=45, decorate,decoration={snake,amplitude=.2mm,segment length=4mm,post length=1mm}]  at (0,0) {};
\node[circle,draw=black,fill=white, minimum size=30, decorate,decoration={snake,amplitude=.2mm,segment length=4mm,post length=1mm}]  at (0,0) {};
\tkzText(0,2.7){$\hM^{4|\v|}_\v$};
\end{scope}

\begin{scope}[xshift=-4cm]
\node[circle,draw=black,fill=ocre!20, minimum size=75]  at (0,0) {};
\node[circle,draw=black,fill=white, minimum size=37]  at (0,0) {};
\tkzText(0,1.8){$B_{2|\v|}\setminus B_{|\v|}$};
\end{scope}

\draw[->, line width=1pt](1,0) -- (4,0);
\tkzText(2,0.3){$\iota$};
\draw[->, line width=1pt](-1,0) -- (-3,0);
\tkzText(-2,0.3){bilip.};

\etkz
\end{center}
\caption{Illustration of \defref{def:regular_dM}.}
\end{figure}

(The constants $2$, $4$ and $10$ are not important as long as they are independent of $\v$.)
\end{definition}

Conditions (a) and (b) essentially assert that the core is not too large compared to the Burgers vector; condition (c) guarantees that the core geometry is regular enough. 

\defref{def:regular_dM} is not vacuous, as $\hM_\v$ itself satisfies these assumptions. 
Take for example $A=\hM_\v^{3|\v|}$. 
By \lemref{lem:frakr=r}, $\r(p) = |\v|$ implies $r(p) = |\v|$ and $\r(p) = 2|\v|$ implies that  
\[
r(p) \le \frac{2 - 1/2\pi}{1 - 1/2\pi} < 3,
\]
i.e.,
\[
\begin{split}
\{p\in\M ~:~ \dist(p,\dM) < |\v|\} & =  \{p\in\hM_\v~:~ |\v| \le \r(p) \le 2|\v|\} \\
& \subset \{p\in\hM_\v~:~ |\v| \le r(p) \le 3|\v|\} = \hM_\v^{3|\v|},
\end{split}
\]
which implies that the first item holds. The second item holds trivially if we take the inclusion map $\hM_\v^{3|\v|} \hookrightarrow \hM_\v^{4|\v|}$. 
The third item follows from the fact that the Euclidean metric $dr^2 + r^2\,d\vp^2$ on $\hM_\v$ induced by the coordinates $(r,\vp)$ is equivalent to the metric $\g_\v$, with equivalence constant $5/4$ (see \eqref{eq:bilipschitz} below).
Thus $A=\hM_\v^{3|\v|}$ is $5/4$-Lipschitz equivalent to $B_{3|\v|}\setminus B_{|\v|}$ with respect to the intrinsic metric of the latter, and thus $4$--Lipschitz equivalent to $B_{3|\v|}\setminus B_{|\v|}$ with its induced metric from $\R^2$ (as the intrinsic metric is larger by a factor of $\pi$ at most). Thus, $A$ is Lipschitz equivalent to $B_{2|\v|}\setminus B_{|\v|}$ with an equivalence constant of $6$.

These geometric assumptions are easily seen to yield the following double-embedding property:

\begin{proposition}
\label{prop:M_hM_M}
Let $(\M,\P)$ be a body with an edge-dislocation of finite diameter having Burgers vector $\v$ and regular inner boundary. 
Suppose that every point on the outer-boundary of $\M$ satisfies $\r\in(0.9R,1.1R)$, for some $R>10|\v|$. 
Then, 
\beq
(\hM_\v^{R/2}\setminus \hM_\v^{3|\v|},\hP_\v) 
\hookrightarrow
(\M,\P)
\hookrightarrow
(\hM_\v^{2R},\hP_\v),
\label{eq:Menclosed}
\eeq
where $\hookrightarrow$ stands here for an isometric embedding preserving the implant map.
\end{proposition}

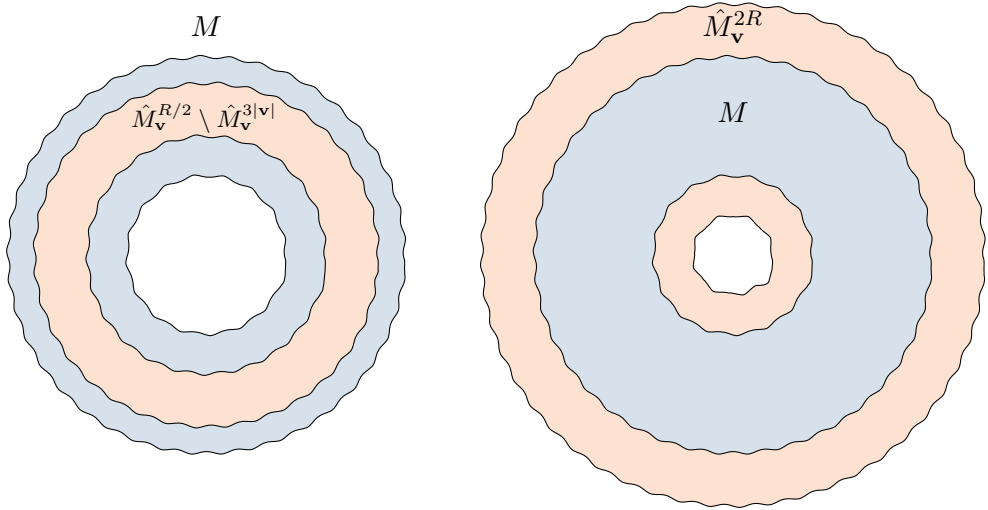
\begin{figure}[H]
\begin{center}
\btkz

\begin{scope}[xshift=0cm]
\node[circle,draw=black,fill=warmblue!20, minimum size=150, decorate,decoration={snake,amplitude=.2mm,segment length=4mm,post length=1mm}]  at (0,0) {};
\node[circle,draw=black,fill=ocre!20, minimum size=130, decorate,decoration={snake,amplitude=.2mm,segment length=4mm,post length=1mm}]  at (0,0) {};
\node[circle,draw=black,fill=warmblue!20, minimum size=90, decorate,decoration={snake,amplitude=.2mm,segment length=4mm,post length=1mm}]  at (0,0) {};
\node[circle,draw=black,fill=white, minimum size=60, decorate,decoration={snake,amplitude=.2mm,segment length=4mm,post length=1mm}]  at (0,0) {};
\end{scope}

\begin{scope}[xshift=7cm]
\node[circle,draw=black,fill=ocre!20, minimum size=190, decorate,decoration={snake,amplitude=.2mm,segment length=4mm,post length=1mm}]  at (0,0) {};
\node[circle,draw=black,fill=warmblue!20, minimum size=150, decorate,decoration={snake,amplitude=.2mm,segment length=4mm,post length=1mm}]  at (0,0) {};
\node[circle,draw=black,fill=ocre!20, minimum size=60, decorate,decoration={snake,amplitude=.2mm,segment length=4mm,post length=1mm}]  at (0,0) {};
\node[circle,draw=black,fill=white, minimum size=30, decorate,decoration={snake,amplitude=.2mm,segment length=4mm,post length=1mm}]  at (0,0) {};
\end{scope}

\tkzText(7,2.0){$\M$};
\tkzText(0,3.15){$\M$};
\tkzText(-0.03,1.92){\footnotesize{$\hM_\v^{R/2}\setminus \hM_\v^{3|\v|}$}};
\tkzText(7,3.15){$\hM_\v^{2R}$};

\etkz
\end{center}
\caption{Illustration of \propref{prop:M_hM_M}.}
\end{figure}

\begin{proof}
Let $\iota:(A,\P)\to  (\hM_\v^{4|\v|},\hP_\v)$ be the assumed isometric embedding. Since both $(\M,\P)$ and $(\hM_\v,\hP_\v)$ are locally-Euclidean, $\iota$ has a unique isometric outward extension, extending to the outer-boundary of $\M$. Since the distance of the outer-boundary from $A$ is at most $1.1R$, and since the distance of the outer-boundary of $\hM_\v^{2R}$ from $\iota(A)$ can be bounded from below, using a similar argument as in Lemma~\ref{lem:frakr=r}, by 
\[
\brk{1 - \tfrac{1}{2\pi}} (2R - 4|\v|) > \brk{1 - \tfrac{1}{2\pi}} (2R - 4R/10) > 1.1R, 
\]
it follows that $\iota(\M)\subset \hM_\v^{2R}$, thus proving the right embedding. 
The left embedding is established similarly, noting that 
the distance of the outer-boundary of $\M$ from $A$ is at least $0.9R - 2|\v|$, and the distance of the outer-boundary of $\hM_\v^{R/2}$ from $\iota(A)$ is at most $R/2$, whereas
\[
R/2 < 0.9R - 2|\v|.
\]
This shows that the outer-boundary of $\hM_\v^{R/2}$ is indeed inside $\iota(\M)$.
If $\hM_\v^{R/2}\setminus \hM_\v^{3|\v|} \nsubseteq \iota(\M)$, this would imply that $\iota(A) \subsetneq \hM_\v^{4|\v|}\setminus \hM_\v^{3|\v|}$, but this is impossible since the distance between any point at the  inner boundary of $\iota(A)$ and the outer-boundary of $\iota(A)$ is $|\v|$, by definition, whereas for $\hM_\v^{4|\v|}\setminus \hM_\v^{3|\v|}$ it is at most $|\v|$.
\end{proof}

As a result of this inner- and outer-enclosure of $(\M,\P)$ by submanifolds of the model manifold $(\hM_\v,\P_\v)$,  energy estimates for  $(\M,\P)$ can be bounded from above and from below by energy estimates for the model manifold. 
This fact will be used repeatedly in Section~\ref{sec:energy_single}.

\subsection{Deviation of $(\hM_\v^R,\hP_\v)$ from a Euclidean annulus}
\label{sec:dev_from_euc}

As the magnitude of the Burgers vector $\v$ tends to zero, the model manifolds $(\hM_\v,\hP_\v)$ approach a punctured Euclidean plane. 
In this section, we quantify the geometric discrepancy between $(\hM_\v^R,\hP_\v)$ and the Euclidean annulus $(B_R\setminus B_{|\v|},\IdRtwo)$, where $\IdRtwo$ is the canonical trivial implant map in $\R^2$.

\begin{proposition}
\label{prop:conv_manifolds}
Let $\hZ_\v:\hM_\v\to\R^2$ be the inclusion map in coordinates:
\beq
\hZ_\v(r,\vp) = (r\,\cos\vp,r\,\sin\vp).
\label{eq:defZe}
\eeq
Its restriction to $\hM_\v^R$ satisfies:
\begin{enumerate}[itemsep=0pt,parsep=0pt,label=(\alph*)]
\item 
$\hZ_\v(\hM_\v^R) = B_R\setminus B_{|\v|}$.

\item The differential of the embedding satisfies the following bounds:
\beq\label{eq:asymp_par}
|d\hZ_\v - \hP_\v|_{\hg_\v,\euc} \lesssim \frac{|\v|}{\r}
\Textand
|d\hZ_\v^{-1} - \hP_\v^{-1}|_{\euc,\hg_\v} \lesssim \frac{|\v|}{\r}.
\eeq
(Note that $d\hZ_\v = \hZ_\v^\# \IdRtwo$.) 

\item The embedding is bilipschitz, with a constant independent of $\v$,
\beq
|d\hZ_\v|_{\hg_\v,\euc} \le 6/5
\Textand 
|d\hZ_\v^{-1}|_{\euc,\hg_\v} \le 5/4. 
\label{eq:bilipschitz}
\eeq
\end{enumerate}
\end{proposition}

\begin{proof}
The first statement is immediate.
Differentiating $\hZ_\v$ and substituting $\hP_\v$ given by \eqref{eq:hQv},
\[
d\hZ_\v -\hP_\v = \frac{1}{2\pi}d\vp\otimes \v.
\]
Since 
by \eqref{eq:bnd_dr_dvp},
\[
|r\,d\vp|_{\hg_\v}  \le \frac{2\pi r}{2\pi r - |\v|} \le \frac{2\pi}{2\pi -1} \le \frac{2\pi}{5},
\]
it follows that
\[
|d\hZ_\v -\hP_\v|_{\hg_\v,\euc} \le \frac{|\v|}{5r} \le \frac{|\v|}{5\r},
\]
where the last inequality follows from \lemref{lem:frakr=r}; this proves the first claim in \eqref{eq:asymp_par}.
In particular, $|d\hZ_\v -\hP_\v|_{\hg_\v,\euc} \le 1/5$.
This implies the bilipschitz bound \eqref{eq:bilipschitz}, since $\hP_\v$ is an isometry.
The second inequality in \eqref{eq:asymp_par} follows from 
\[
|d\hZ_\v^{-1} - \hP_\v^{-1}|_{\euc,\hg_\v} \le 
|d\hZ_\v^{-1}|_{\euc,\hg_\v} |d\hZ_\v - \hP_\v|_{\hg_\v,\euc} |\hP_\v^{-1}|_{\euc,\hg_\v} 
\]
using \eqref{eq:bilipschitz}, and the fact that $\hP_\v$ is an isometry.
\end{proof}

\section{The energetics of an edge-dislocation}\label{sec:energy_single}

In this section we consider the variational problem introduced in \secref{sec:energy} for the model body manifolds with edge-dislocations $(\hM_\v,\hP_\v)$ introduced in \secref{sec:single_dislocation}. 

\subsection{Relation to admissible strain models}\label{sec:admissible_strain}
\label{sec:admissible}

We start by explaining is which sense does the \Emph{admissible strain} approach, which is often used in the literature, constitute a ``geometric linearization" of the Volterra approach used in the present work (as well as in \cite{KM16a,EKM20}).

Consider the body with a dislocation $(\hM_\v^R,\hP_\v)$. The elastic energy of a configuration $f:\hM_\v^R\to\R^2$ is
\[
E(f) = \int_{\hM_\v^R} \calW(df \circ \hP_\v^{-1})\,\hVolumeV.
\]
Suppose that the dislocation is ``small", hence there exist configurations $f$ having ``small" energy. As shown below, an FJM argument shows that there exists a matrix $U\in\SO(2)$ such that the $L^2$-norm of $df - U\hP_\v$ is of the same order as the energy $E(f)^{1/2}$. 
This motivates the following representation,
\[
df\circ \hP_\v^{-1} = U + (df - U\hP_\v) \circ \hP_\v^{-1} .
\]
On the other hand, from \propref{prop:conv_manifolds}, $d\hZ_\v - \hP_\v^{-1}$ is also small for small dislocations (sufficiently far away from the core), hence formally, to leading order,
\[
df\circ \hP_\v^{-1} \simeq U + (df - U\hP_\v) \circ d\hZ_\v^{-1},
\]
and thus, since the energy is frame-indifferent,
\[
E(f) \simeq \int_{\hM_\v^R} \calW(I+ (U^T df - \hP_\v) \circ d\hZ_\v^{-1})\,\hVolumeV.
\]
Changing variables, $\hZ_\v: \hM_\v\to B_R\setminus B_{|\v|} \equiv \W$, using the fact that $(\hZ_\v^{-1})^\#\hVolumeV\simeq \VolumeE$,
\[
E(f) \simeq \int_{\W} \calW(\beta(x))\,dx,
\]
where $\beta$ is an $\R^2$-valued 1-form on $\hZ_\v(\hM_\v)$ given by
\[
\beta = I + (U^Tdf - \hP_\v) \circ d\hZ_\v^{-1}.
\]
For every simple, closed, oriented path $C\subset\W$ homotopic to the inner boundary,
\[
\int_C \beta = \int_{\hZ_\v^{-1}(C)} (d\hZ_\v + (U^T df - \hP_\v))  = -\v,
\]
which is the standard condition for admissible strains. Thus, within this approximation, which is a combination of geometric and small-strain approximation, the variational problem may be reformulated as finding a minimum for 
\[
E(\beta) = \int_\W \calW(\beta)\, \VolumeE,
\]
where $\beta: \W\to\Hom(\R^2,\R^2)$ satisfies the circulation constraint
\[
\int_C \beta = -\v
\]
for every closed oriented path $C\subset\W$. 
This is the admissible strain model with nonlinear energy considered in \cite{SZ12,MSZ14,MSZ15,Gin19b}.

The $\R^2$-valued 1-form $\beta$ can be rewritten as
\[
\beta = U^T d(f\circ \hZ_\v^{-1})+ (I - \hP_\v)\circ d\hZ_\v^{-1}.
\]
The first term is simply the differential of the configuration (after a change of variables), and as such can be interpreted as an elastic strain.  The second term carries the circulation, hence can be thought of as a plastic strain; unlike the starting point, this term is \emph{added} to the elastic strain. 
From this perspective, the above approximation can be identified with the approximation of a multiplicative decomposition of the strain by an additive decomposition. 

Another approach to the admissible strain model is to view $\beta$ as an Eulerian variable. This approach is presented in \cite[p.~180]{MSZ15} (although the Lagrangian approach is eventually used) and in \cite{CGM23}.
There $\beta$ is a map from a ``deformed" or ``spatial" configuration $\W^{sp}\subset \R^2$, to linear maps from $T\W^{sp}$ to $\R^2$, representing how tangent vectors at each point map to a locally defined lattice configuration (which we can think of as the body $(\R^2,\IdRtwo)$).
In this model, the circulation constraint
\[
\int_C \beta = \v
\]
for every closed oriented path $C\subset\W$ does not involve any approximation (it is equivalent to $\P$ in our formulation), however the variational problem considered is quite different, as the deformed configuration is given. 

In both cases, as stated in the introduction, it is not clear to us how to encode in this framework bending (in which the assumption $\beta \approx I$ clearly fails and the deformed configuration is clearly not given), large deformations or disclinations.

\subsection{Upper and lower bounds}

We start by proving that the infimal elastic energy of a body with an edge-dislocation of magnitude $|\v|$ and outer-radius $R$ scales like $|\v|^2\, \log (R/|\v|)$. We start with the upper bound:

\begin{proposition}
\label{prop:4.1}
Consider $(\hM_\v^R,\hP_\v)$ for some $R>|\v|$, and let $\delta>0$ be such that $\delta R\in[|\v|,R)$.
Then,
\[
\inf_{f\in \HOne(\hM_\v;\R^2)} \int_{\hM_\v^R\setminus\hM_\v^{\delta R}} \dist^2(df,\SO(\hg_\v,\euc))\,\hVolumeV \lesssim |\v|^2 \log \frac{1}{\delta},
\]
with constant independent of $\v$, $\delta$ and $R$.
\end{proposition}

\begin{proof}
The bound is obtained by setting $f=\hZ_\v$, using the fact that $\dist(d\hZ_\v,\SO(\g,\euc)) \le |d\hZ_\v - \hP_\v|$, and integrating the first estimate \eqref{eq:asymp_par}, noting that $\r \simeq r$ and that $\hVolumeV \simeq \hZ_\v^\# \VolumeE$.
\end{proof}

\begin{corollary}
\label{cor:upper}
Let $(\M,\P)$ be a body with a dislocation having a Burgers vector $\v$, a finite diameter and a regular inner boundary.
Let $\Gamma$ be the outer-boundary of $\M$, and assume that $\r|_{\Gamma} \subset (0.9R,1.1R)$ for $R\ge 10|\v|$.
Then,
\[
\inf_{f\in \HOne(\M;\R^2)} \int_{\M} \dist^2(df,\SO(\g,\euc))\,\Volume \lesssim |\v|^2 \log \frac{R}{|\v|},
\]
with constant independent of $\v$ and $R$.
\end{corollary}

\begin{proof}
By \propref{prop:M_hM_M}, $(\M,\P)$ embeds isometrically in $(\hM_\v^{2R},\hP_\v)$, hence, by \propref{prop:4.1} with outer-radius $2R$ and inner-radius $\delta R = |\v|$,
\[
\begin{split}
\inf_{f\in \HOne(\M;\R^2)} \int_{\M} \dist^2(df,\SO(\g,\euc))\,\Volume &\le \\ 
	&\hspace{-4cm} \le \inf_{f\in \HOne(\hM_\v;\R^2)} \int_{\hM_\v^{2R}} \dist^2(df,\SO(\hg_\v,\euc))\,\hVolumeV  \\
	&\hspace{-4cm} \lesssim |\v|^2 \log \frac{2R}{|\v|} \lesssim  |\v|^2 \log \frac{R}{|\v|},
\end{split}
\]
where in the last inequality we used that $R\ge 10|\v|$.
\end{proof}

We proceed with the lower bound. 
To this end, we establish a Friesecke--James--M\"uller-type (FJM) rigidity estimate for a body with a dislocation, in which parallel sections of $\SO(\g,\euc)$ replace the constant rotations in Euclidean space.  
We first note the existence of a uniform FJM constant for Euclidean half-annuli of large enough aspect ratio:

\begin{lemma}[Uniform FJM constant for half-annuli]
\label{lem:FJM_annulus}
Let $\W\subset\R^2$ be a (Euclidean) half-annulus,
\[
\W = \{(r\cos\vp, r\,\sin\vp) ~:~ r\in (R_1,R_2), \,\,\, \vp\in (0,\pi)\}.
\]
having aspect ratio $R_2/R_1\ge 3/2$. Then there exists a constant $C>0$ independent of $R_1,R_2$, and
there exists for every $f\in \HOne(\W;\R^2)$ a matrix $U\in \SO(2)$, such that
\[
\int_{\W} |df - U|^2\,\VolumeE \le C \int_{\W} \dist^2(df,\SO(2))\,\VolumeE.
\] 
\end{lemma}

\begin{proof}
By the FJM rigidity estimate \cite[Theorem~3.1]{FJM02b}, there exists for every domain $\W$  such a constant $C$.
Since the FJM constant is scale-independent \cite[Theorem~5]{FJM06}, we may assume without loss of generality that $R_2=1$. 
The existence of a uniform FJM constant follows from the fact that all half-annuli of inner-radius $R_1<2/3$ and outer-radius $R_2=1$ are uniformly Lipschitz-equivalent \cite[Theorem~5]{FJM06}. 
\end{proof}

\begin{theorem}
\label{thm:rigidity_single_disloc}
Consider $(\hM_\v^R,\hP_\v)$ with $R>3|\v|/2$. Let $\delta>0$ be such that $\delta R\in[|\v|,2R/3)$.
There exists for every $f\in \HOne(\hM_\v^R;\R^2)$ a matrix $U\in\SO(2)$, such that
\beq
\int_{\hM_\v^R\setminus\hM_\v^{\delta R}} |df - U\hP_\v|_{\hg_\v,\euc}^2 \,\hVolumeV  \lesssim \int_{\hM_\v^R\setminus\hM_\v^{\delta R}} \dist^2(df,\SO(\hg_\v,\euc))\,\hVolumeV ,
\label{eq:FJM_single}
\eeq
with constant independent of $\v$, $R$ and $\delta$. 
\end{theorem}

The equivalent statement for the admissible strain model, is given in \cite[Prop.~3.3]{SZ12}, in which the domain is cut along a line in order to make it simply-connected and consequently,  closed forms are exact.

\begin{proof}
By Comment~2 following \defref{def:single_disloc}, every parallel section of $\SO(\hg_\v,\euc)$ is of the form $U\hP_\v$ for some $U\in\SO(2)$. The idea behind the proof, as in \cite[Prop.~3.3]{SZ12}, is to cover $\hM_\v^R\setminus\hM_\v^{\delta R}$ with overlapping, simply-connected domains. Since $(\hM_\v^R,\hP_\v)$ is locally-flat, each subdomain embeds isometrically in Euclidean plane, hence restrictions of parallel sections of $\SO(\hg_\v,\euc)$ can be viewed as constant matrices. Applications of the standard rigidity theorem \cite[Theorem~3.1]{FJM02b} for each sub-domain, exploiting their overlap, yields the desired bound.

Specifically, consider the following covering of $\hM_\v^R\setminus\hM_\v^{\delta R}$ by half-annuli, as in the figure below:

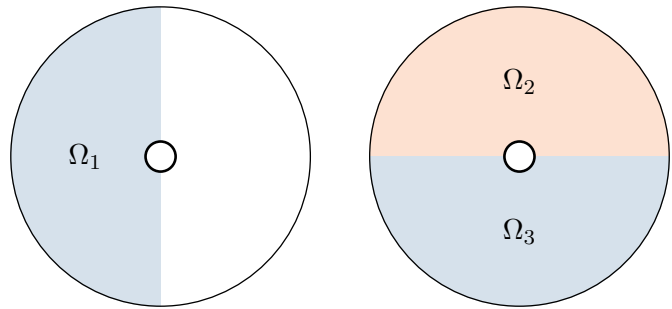
\begin{figure}[H]
\[
\btkz
\clip (0,0) circle (2cm);
\fill[warmblue!20] (0,-2) -- (0,2) -- (-2,2) -- (-2,-2) -- cycle;
\node at (-1,0) {$\W_1$};
\draw[line width=1pt] (0,0) circle (2cm);
\draw[line width=1pt, fill = white] (0,0) circle (0.2cm);
\etkz
\qquad
\btkz
\clip (0,0) circle (2cm);
\fill[ocre!20] (-2,0) -- (2,0) -- (2,2) -- (-2,2) -- cycle;
\node at (0,1) {$\W_2$};
\fill[warmblue!20] (-2,0) -- (2,0) -- (2,-2) -- (-2,-2) -- cycle;
\node at (0,-1) {$\W_3$};
\draw[line width=1pt] (0,0) circle (2cm);
\draw[line width=1pt, fill = white] (0,0) circle (0.2cm);
\etkz
\]
\caption{Coverage of $\hM_\v^R\setminus\hM_\v^{\delta R}$ by half-annuli.}
\end{figure}

Since $(\hM_\v^R, \hP_\v)$ is not a Euclidean annulus, we have to be more precise about how the half-annuli $\W_i$ are defined; to this end, we use the natural parameterization of $\hM_\v^R$, for example,
\[
\W_1 = \{(r,\vp) ~:~ r\in(\delta R,R), \,\,\, \vp\in(\pi/2,3\pi/2)\}.
\]
Now, $(\W_i,\hP_\v)$ are simply-connected flat manifolds that are isometrically-embeddable in $\R^2$ (meaning that we have an immersion $\Psi:\W_i\to \R^2$ such that $d\Psi = \hP_\v$), but under this embedding they are not Euclidean half annuli (since $\hP_\v$ is not the identity in the coordinates above). 
The important property of the $\W_i$, is that by the bounds on $R$ and $\delta$ and by the properties of the maps $\hZ_\v$ (\propref{prop:conv_manifolds}) they are Lipschitz-equivalent to Euclidean half-annuli of aspect ratio greater than $3/2$, with bilipschitz constants independent of $\v$, $R$ and $\delta$. 
Thus, \lemref{lem:FJM_annulus} applies to each of the $\W_i$.

Given $f\in \HOne(\hM_\v^R;\R^2)$, there exist three matrices $U_i\in\SO(2)$, $i=1,2,3$, such that
\beq
\int_{\W_i} |df - U_i \hP_\v|_{\hg_\v,\euc}^2\,\hVolumeV \lesssim \int_{\hM_\v^R\setminus\hM_\v^{\delta R}} \dist^2(df,\SO(\g_\v,\euc))\,\hVolumeV, 
\qquad i=1,2,3.
\label{eq:Q15}
\eeq
where we bounded the integrals over $\W_i$ on the right-hand side  by an integral over $\hM_\v^R\setminus\hM_\v^{\delta R}$.
Using the fact that 
\[
|U_i - U_j| = |U_i \hP_\v-U_j \hP_\v|_{\hg_\v,\euc}^2 \lesssim 
|df-U_i \hP_\v|_{\hg_\v,\euc}^2 + |df-U_j \hP_\v|_{\hg_\v,\euc}^2,
\]
where the left-hand side
is constant, 
\[
\begin{aligned}
\Vol_{\hg_\v}(\W_1\cap \W_2) \, |U_1 \hP_\v-U_2 \hP_\v|_{\hg_\v,\euc}^2 & \lesssim \int_{\hM_\v^R\setminus\hM_\v^{\delta R}} \dist^2(df,\SO(\hg_\v,\euc))\,\hVolumeV \\
\Vol_{\hg_\v}(\W_1\cap \W_3) |U_1 \hP_\v-U_3 \hP_\v|_{\hg_\v,\euc}^2  & \lesssim \int_{\hM_\v^R\setminus\hM_\v^{\delta R}} \dist^2(df,\SO(\hg_\v,\euc))\,\hVolumeV.
\end{aligned}
\]
Using the inequality $|a-b|^2 \lesssim a^2 + b^2$ several more times, noting that $\Vol_{\hg_\v}(\W_2\setminus \W_1) \simeq \Vol_{\hg_\v}(\W_1\cap \W_2)$ and $\Vol_{\hg_\v}(\W_3\setminus \W_1) \simeq \Vol_{\hg_\v}(\W_1\cap \W_3)$, we obtain that for every $i=2,3$,
\[
\int_{\W_i} |df - U_1 \hP_\v|_{\hg_\v,\euc}^2\,\VolumeE \lesssim \int_{\hM_\v^R\setminus\hM_\v^{\delta R}} \dist^2(df,\SO(\hg_\v,\euc))\,\hVolumeV. 
\]
Summing over $i=2,3$ we obtain the desired result with $U = U_1$.
\end{proof}

\begin{proposition}
\label{prop:4.4}
Consider $(\hM_\v^R,\hP_\v)$  for some $R>3|\v|/2$. Let $\delta>0$ be such $\delta R\in [|\v|,2R/3)$.
Then,
\[
\inf_{f\in \HOne(\hM_\v^R;\R^2)} \int_{\hM_\v^R\setminus \hM_\v^{\delta R}} \dist^2(df,\SO(\hg_\v,\euc))\,\hVolumeV \gtrsim |\v|^2 \log \frac{1}{\delta},
\]
with constant independent of $\v$, $R$ and $\delta$.
\end{proposition}

\begin{proof}
Let $f\in \HOne(\hM_\v^R;\R^2)$. By \thmref{thm:rigidity_single_disloc}, there exists a matrix $U\in\SO(2)$, such that
\[
\begin{split}
\int_{\hM_\v^R\setminus\hM_\v^{\delta R}} \dist^2(df,\SO(\hg_\v,\euc))\,\hVolumeV 
&\gtrsim   \int_{\hM_\v^R\setminus\hM_\v^{\delta R}} |df - U\hP_\v|_{\hg_\v,\euc}^2\,\hVolumeV \\
&\simeq \int_{\delta R}^R \brk{\int_{\{r=s\}}  |df - U\hP_\v|_{\hg_\v,\euc}^2 d\vp} s\, ds \\
&\gtrsim  \int_{\delta R}^R \frac{1}{s} \brk{\int_{\{r=s\}}   |df(\pl_\vp) - U\hP_\v(\pl_\vp)|^2 d\vp}\, ds \\
&\ge \int_{\delta R}^R  \frac{1}{2\pi s} \Abs{\int_{\{r=s\}} (df(\pl_\vp) - U\hP_\v(\pl_\vp)) d\vp}^2 \, ds \\
&=  \int_{\delta R}^R  \frac{1}{2\pi s} \Abs{\int_{\{r=s\}} (df - U\hP_\v)}^2 \, ds \\
&= |\v|^2 \int_{\delta R}^R  \frac{ds}{2\pi s} \\
&= \frac{|\v|^2}{2\pi}\log\frac{1}{\delta}.
\end{split}
\]
In the passage to the second line we used Fubini's theorem and the fact that \eqref{eq:bnd_for_volume} implies that $\hVolumeV = \nu^1\wedge \nu^2 \simeq r\,dr \wedge d\vp$; in the passage to the third line we used the fact that for a linear operator $A$ and a unit vector $x$, $|A|^2\ge |Ax|^2$, and that as a consequence of \eqref{eq:bnd_for_volume}, $r^{-1} \pl_\vp$ is approximately a unit vector; in the passage to the fourth line we used Jensen's inequality;  in the passage to the fifth line we used the definition of the line integral of a one-form; in the passage to the sixth line we used the fact that the integral of $df$ vanishes whereas the integral of $\hP_\v$ equals $\v$. 
\end{proof}

\begin{corollary}
\label{cor:low_bnd_single_disloc}
Let $(\M,\P)$ be a body with a single dislocation $\v$ having a regular inner boundary.
Let $\Gamma$ be the outer-boundary of $\M$, and assume that $\r|_{\Gamma} \subset (0.9R,1.1R)$ for some $R>10|\v|$. 
Then,
\[
\inf_{f\in \HOne(\M;\R^2)} \int_{\M} \dist^2(df,\SO(\g,\euc))\,\Volume \gtrsim |\v|^2\log\frac{R}{|\v|}
\]
with constant independent of $\v$ and $R$.
\end{corollary}

\begin{proof}
By \propref{prop:M_hM_M},  we have an isometric embedding,
\[
(\hM_\v^{R/2}\setminus \hM_\v^{3|\v|},\hP_\v) \hookrightarrow (\M,\P).
\]
Substituting the bound of \propref{prop:4.4}, using the fact that 
the aspect ratio is $\delta = (3|\v|)/(R/2)$, and that the condition $R>10|\v|$ implies that $\delta (R/2)= 3|\v| \in [|\v|, 2(R/2)/3)$,
\[
\begin{split}
\inf_{f\in \HOne(\M;\R^2)} \int_{\M} \dist^2(df,\SO(\g,\euc))\,\Volume  & \ge \\
&\hspace{-6cm}\ge  \inf_{f\in \HOne(\hM_\v;\R^2)} \int_{\hM_\v^{R/2}\setminus \hM_\v^{3|\v|}} \dist^2(df,\SO(\hg_\v,\euc))\,\hVolumeV \\
&\hspace{-6cm}\gtrsim |\v|^2\log\frac{R}{6|\v|} \gtrsim |\v|^2\log\frac{R}{|\v|},
\end{split}
\]
where in the last inequality we used again that $R>10|\v|$. 
\end{proof}

With that, we obtained lower and upper bounds for the infimal energy of an edge-dislocation:

\begin{corollary}[Single dislocation, energy bounds]
\label{corr:Escaling}
Let $(\M,\P)$ be a body with a dislocation $\v$ having a regular inner boundary.
Let $\Gamma$ be the outer-boundary of $\M$, and assume that $\r|_{\Gamma} \subset (0.9R,1.1R)$ for some $R>10|\v|$. 
Then,
\[
|\v|^2\log\frac{R}{|\v|} \lesssim \inf_{f\in \HOne(\M;\R^2)} \int_\M \dist^2(df,\SO(\g,\euc))\,\Volume \simeq E(f) \lesssim |\v|^2\log\frac{R}{|\v|}
\]
with constants independent of $\v$ and $R$.
\end{corollary}

\begin{proof}
This is an immediate consequence of \corrref{cor:upper}, \corrref{cor:low_bnd_single_disloc} and the lower and upper bounds in \eqref{eq:bound_calW}.
\end{proof}

\subsection{Asymptotic estimates for small dislocations}\label{sec:asymp_estimates}

In this section we derive more detailed estimates for the infimal energy of a body with an edge-dislocation. 
These estimates will be needed when we consider bodies with multiple edge-dislocations, in regimes where the magnitude of each dislocation tends to zero, whereas their number tends to infinity. 
In these regimes, the distance between neighboring dislocations can tend to zero, which requires us to examine the energetics of bodies with dislocations in which the outer-radius also shrinks to zero, albeit at a slower rate than the magnitude of the dislocation.

Let $\v\in\R^2$, 
let $R>0$ be the outer-radius (in coordinates) and let $\delta\in(0,1/2)$ be the aspect ratio between the inner and the outer radii.
Having set the dimensions of the annulus, we consider a dislocation having Burgers vector $\e\v$, where $\e>0$ is constrained by the geometric requirement that $\e |\v| < \delta R$. 
We define for every $\e\in(0,\delta R/|\v|)$ a \Emph{rescaled energy function}, $\hE_{\e,\delta}^R(\cdot;\v):  \HOne(\hM_{\e\v};\R^2) \to \R$,
\beq
\hE_{\e,\delta}^R(f;\v) = \frac{1}{\e^2\,\log(1/\delta)} \int_{\hM^R_{\e\v}\setminus \hM^{\delta R}_{\e\v}} \calW(df\circ \hP_{\e\v}^{-1})\,\hVolumeEV,
\label{eq:Ehat}
\eeq
and denote its infimum by
\[
\hI_{\e,\delta}^R(\v) = \inf_{f\in \HOne(\hM_{\e\v};\R^2)} \hE_{\e,\delta}^R(f;\v).
\]
It follows from Propositions~\ref{prop:4.1} and \ref{prop:4.4} (note that $R>2\e\v$), and the lower and upper bounds \eqref{eq:bound_calW} that 
\beq\label{eq:hI_e_delta}
\hI_{\e,\delta}^R(\v) \simeq |\v|^2,
\eeq
where the bounding constants are independent of $\e$, $\delta$, $R$ and $\v$.

The rescaled energy functional $\hE_{\e,\delta}^R(f;\v)$ is compared to another functional, which can be viewed as its linearization. We introduce the \Emph{quadratic energy functional},
\beq
\Elin_\delta(\beta;\v) = \frac{1}{\log(1/\delta)} \int_{B_1\setminus B_\delta} \bbW(\beta)\,\VolumeE ,
\label{eq:Ecell}
\eeq
where, as we recall, $\bbW(\beta) = \half D_I^2\calW(\beta,\beta)$,
defined over the set of so-called \Emph{admissible strains}, $X^1_\delta(\v) $, where for $\sigma>0$,
\beq
\label{eq:Xdeltav}
X^\sigma_{\delta\sigma}(\v) = \BRK{\beta\in L^2(B_{\sigma}\setminus B_{\delta\sigma}; \R^2\otimes\R^2) ~:~
\curl\beta=0,
\quad
\oint_C \beta = - \v},
\eeq
where $C$ is any positively-oriented curve homotopic to $\pl B_\delta$ (the vanishing of the distributive curl of $\beta$ guarantees the existence of line integrals of $\beta$, even though it only has $L^2$-regularity \cite{CL05}). We denote the infimum of the quadratic energy functional by
\[
\Ilin_\delta(\v) = \inf_{\beta\in X^1_\delta(\v)} \Elin_\delta(\beta;\v).
\]

The quadratic functional \eqref{eq:Ecell} was studied in \cite{GLP10} (and in references therein), with the following outcomes:
\begin{enumerate}[itemsep=0pt,label=(\alph*)]
\item $\Ilin_\delta$ is a quadratic form \cite[Eqs.~(27)--(28)]{GLP10} and satisfies $\Ilin_\delta(\v) \gtrsim |\v|^2$, where the constant is independent of $\delta$ \cite[Remark 3]{GLP10}.  
\item 
The quadratic variational problem is invariant under a scaling of the domain: more precisely,  for $\sigma>0$,
\[
\beta\in X^1_\delta(\v)
\qquad\text{if and only if}\qquad
\beta_\sigma \in X^{\sigma}_{\delta \sigma}(\v),
\]
where
\[
\beta_\sigma(x) = \frac{1}{\sigma}\beta(x/\sigma),
\]
and
\[
\int_{B_1\setminus B_{\delta}} \bbW(\beta)\,\VolumeE = 
\int_{B_\sigma\setminus B_{\delta\sigma}} \bbW(\beta_\sigma)\,\VolumeE.
\]
This justifies why taking the outer-radius equal to $1$ does not limit the generality of the quadratic energy functional.

\item Let the $\R^2$-valued 1-form $\beta_\v\in \Omega^1(\R^2\setminus \{0\}; \R^2)$ be the distributional solution of 
\beq
\curl\beta_\v = -\v\,\delta_0
\qquad
\div  (D_I^2\calW(\beta_\v,\cdot)) = 0.
\label{eq:betav}
\eeq
By \cite[Corollary 6]{GLP10}, 
\beq
0 \le \Elin_\delta(\beta_\v;\v)  - \Ilin_\delta(\v) \lesssim \frac{|\v|^2}{\log(1/\delta)} \lesssim \frac{\Ilin_\delta(\v)}{\log(1/\delta)},
\label{eq:Edelta_betav}
\eeq
where the constants are independent of $\delta$ and $\v$.
Moreover, by \cite[Eq.~(29)]{GLP10}, $\beta_\v$ is self-similar, that is, for every $\sigma>0$,
\[
\beta_\v(x) = \frac{1}{\sigma} \,\beta_\v(x/\sigma),
\]
and satisfies the bound
\beq\label{eq:beta_bound}
|\beta_\v(x)| \lesssim \frac{|\v|}{r}.
\eeq

\item The function $\Ilin_\delta$ converges pointwise as $\delta\to0$ to a limit
\[
\Ilin_0(\v) = \lim_{\delta\to0} \Ilin_\delta(\v),
\] 
and by \cite[Corollary 6]{GLP10}
\beq
\label{eq:Idelta_I0}
|\Ilin_\delta(\v) - \Ilin_0(\v)| \lesssim \frac{|\v|^2}{\log(1/\delta)},
\eeq
where $\Ilin_0$ is a positive-definite quadratic form \cite[Eq.~(36)]{GLP10}, and thus also\\ $\Ilin_\delta(\v)\simeq  \Ilin_0(\v) \simeq |\v|^2$ with constants independent of $\delta$.
\item It follows from the first three items that $\beta_\v$ (restricted to the relevant domain) satisfies 
\[
|\Elin_\delta(\beta_\v;\v) - \Ilin_0(\v)| \lesssim \frac{|\v|^2}{\log(1/\delta)}.
\]
\end{enumerate}

We proceed to relate between the minimization problems for the nonlinear energy functional \eqref{eq:Ehat} and the quadratic energy functional \eqref{eq:Ecell}. 
As mentioned in the beginning of this section, bounding the difference between $\hI_{\e,\delta}^{R}(\v)$ and $\Ilin_\delta(\v)$ as $\e\to0$ is not sufficient, as we need to account for settings in which $R$ and $\delta$ are $\e$-dependent.

\subsubsection{Lower bounds}

\begin{proposition}
\label{prop:cell_liminf}
Fix an aspect ratio $\delta\in (0,1/10)$ and a sequence $n_\e>1$, satisfying $n_\e\to\infty$ and $\e n_\e \to0$. 
There exists a non-negative sequence $\sigma_{\e,\delta}$ (depending on $n_\e$) satisfying 
\[
\lim_{\e\to 0} \sigma_{\e,\delta} = 0,
\]
such that for every sequence $\v_\e\in\R^2$ satisfying $|\v_\e| \le n_\e$ and 
 every sequence $R_\e>0$ satisfying $\delta R_\e \ge n_\e \e |\v_\e|$, 
\[
\hI_{\e,\delta}^{R_\e}(\v_\e) \ge \Ilin_\delta(\v_\e) - |\v_\e|^2\, \sigma_{\e,\delta}.
\]
In particular, for every $\v\in\R^2$,
\[
\liminf_{\e\to0} \hI_{\e,\delta}^{R_\e}(\v) \ge \Ilin_\delta(\v).
\]
\end{proposition}

\begin{proof}
First, note that the assumptions on $R_\e$ guarantee that the inner-radius $\delta R_\e$ is indeed greater than $\e|\v_\e|$, hence $\hE_{\e,\delta}^{R_\e}(\cdot;\v_\e)$ is well-defined.

The proof, whose strategy is similar to step 2 in the proof of \cite[Proposition 3.11]{SZ12}, is by contradiction: suppose that there exist sequences $|\v_\e| \le n_\e$ and  $\delta R_\e \ge n_\e \e \v_\e$, such that
\[
\liminf_{\e\to0} \frac{\hI_{\e,\delta}^{R_\e}(\v_\e) -  \Ilin_\delta(\v_\e)}{|\v_\e|^2} < 0.
\]
I.e., the exists a constant $c_\delta > 0$ and a (not relabeled) subsequence $\e\to0$, such that
\[
\frac{\hI_{\e,\delta}^{R_\e}(\v_\e) }{|\v_\e|^2} < \frac{\Ilin_\delta(\v_\e)}{|\v_\e|^2} - c_\delta.
\]
We may take a further subsequence such that $\v_\e/|\v_\e|\to\v$ in $\R^2$. Since $\Ilin_\delta$ is a continuous quadratic form, it follows that 
\beq
\label{eq:Iquad_contradiction}
\LiminfEps \frac{\hI_{\e,\delta}^{R_\e}(\v_\e) }{|\v_\e|^2} < \Ilin_\delta(\v) - c_\delta.
\eeq
We will show that
\[
\LiminfEps \frac{\hI_{\e,\delta}^{R_\e}(\v_\e) }{|\v_\e|^2} \ge \Ilin_\delta(\v),
\]
whence the contradiction.

Let 
$f_\e\in H^1(\hM_{\e\v_\e};\R^2)$
be a sequence of approximate minimizers for $\hE_{\e,\delta}^{R_\e}(\cdot;\v_\e)$, satisfying
\[
\LimEps \frac{\hE_{\e,\delta}^{R_\e}(f_\e;\v_\e) - \hI_{\e,\delta}^{R_\e}(\v_\e)}{|\v_\e|^2} = 0.
\]
Since from \eqref{eq:hI_e_delta} we have that $\hI_{\e,\delta}^{R_\e}(\v_\e)\simeq |\v_\e|^2$, it follows that $\hE_{\e,\delta}^{R_\e}(f_\e;\v_\e)/|\v_\e|^2$ is bounded as $\e\to0$.

From the lower bound in \eqref{eq:bound_calW} and \thmref{thm:rigidity_single_disloc} (using the bounds $\delta<1/10$ and $R_\e \ge 10 n_\e \e |\v_\e| > 10\e |\v_\e|$), there exist matrices $U_\e\in \SO(2)$, such that
\beq
\frac{1}{\e^2\log(1/\delta)} 
\int_{\hM^{R_\e}_{\e\v_\e}\setminus \hM^{\delta R_\e}_{\e\v_\e}} |df_\e - U_\e\hP_{\e\v_\e}|^2 \hVolumeEVE \lesssim
\hE_{\e,\delta}^{R_\e}(f_\e;\v_\e) \lesssim |\v_\e|^2,
\label{eq:FJM_cell}
\eeq
where the constants in both inequalities are independent of $\e$, $\delta$, $\v_\e$ and $R_\e$.
Since we can replace $f_\e$ with $U_\e f_\e$ without changing the energy, we can assume without loss of generality that $U_\e = \id$.

Denote
\beq
\eta_\e = \frac{df_\e \circ \hP^{-1}_{\e\v_\e} - \id}{\e |\v_\e|} \in L^2(\hM_{\e\v_\e};\R^2\otimes\R^2),
\label{eq:eta_e}
\eeq
and rewrite
\[
\hE_{\e,\delta}^{R_\e}(f_\e;\v_\e)  = 
 \frac{1}{\e^2\,\log(1/\delta)} \int_{\hM^{R_\e}_{\e\v_\e}\setminus \hM^{\delta R_\e}_{\e\v_\e}} \calW(I + \e |\v_\e| \eta_\e)\,\hVolumeEVE.
\]
Note that \eqref{eq:FJM_cell} implies that $\| \eta_\e\|_{L^2(\hM^{R_\e}_{\e\v}\setminus \hM^{\delta R_\e}_{\e\v})} \lesssim \log(1/\delta)$.

We proceed to change variables to a Euclidean domain,  using the maps $\hZ_{\e\v_\e}: \hM^{R_\e}_{\e\v_\e}\setminus \hM^{\delta R_\e}_{\e\v_\e} \to B_{R_\e} \setminus B_{\delta R_\e}$.
It follows from \eqref{eq:asymp_par} that
\beq
\label{eq:volume_change}
\Abs{\frac{\hVolumeEVE}{(\hZ_{\e\v_\e}^{-1})^\#\VolumeE}  -1} \lesssim \frac{\e |\v_\e|}{\delta R_\e} \le \frac{1}{n_\e},
\eeq
hence, since $n_\e\to\infty$,  there exists a  non-negative sequence $\alpha_{\e,\delta}$, satisfying
\[
\LimEps \alpha_{\e,\delta} = 0,
\]
such that
\[
\hE_{\e,\delta}^{R_\e}(f_\e;\v_\e) \ge  \frac{1}{\e^2\,\log(1/\delta)} \int_{B_{R_\e} \setminus B_{\delta R_\e}} \calW(I + \e |\v_\e| \teta_\e)\,\VolumeE
- \alpha_{\e,\delta} |\v_\e|^2,
\]
where $\teta: B_{R_\e} \setminus B_{\delta R_\e} \to \R^2\otimes \R^2$ is given by
$\teta_\e = \eta_\e\circ \hZ_{\e\v_\e}^{-1}$.
Since $\hZ_\e$ are uniformly bilipschitz, it follows that 
\[
\|\teta_\e\|_{L^2(B_{R_\e} \setminus B_{\delta R_\e})} \lesssim \log(1/\delta).
\]
We linearize the energy by introducing a cutoff function $\chi_\e:L^2(B_{R_\e} \setminus B_{\delta R_\e}) \to \R$,
\[
\chi_\e = \ind_{|\teta_\e|  \le (\sqrt{n_\e}\e |\v_\e|)^{-1}}.
\]
Note that by Chebyshev's inequality, 
\beq
\label{eq:supp_chi_e}
\frac{|\{|\teta_\e|  \ge (\sqrt{n_\e}\e |\v_\e|)^{-1}\}|}{|B_{R_\e} \setminus B_{\delta R_\e}|} 
\le \frac{n_\e \e^2 |\v_\e|^2\|\teta_\e\|_{L^2}^2}{R_\e^2(1-\delta^2) }
\le \frac{n_\e \e^2 |\v_\e|^2\|\teta_\e\|_{L^2}^2}{n_\e^2 \e^2 |\v_\e|^2(\delta^{-2}-1)} \lesssim \frac{1}{n_\e} \to 0.
\eeq
Then,
\[
\begin{split}
\frac{\hE_{\e,\delta}^{R_\e}(f_\e;\v_\e)}{|\v_\e|^2} &\ge  
\frac{1}{\e^2|\v_\e|^2\,\log(1/\delta)} \int_{B_{R_\e} \setminus B_{\delta R_\e}} \chi_\e \calW(I + \e |\v_\e| \teta_\e)\,\VolumeE - \alpha_{\e,\delta} \\
&\ge \frac{1}{\log(1/\delta)}  
\int_{B_{R_\e} \setminus B_{\delta R_\e}} \chi_\e 
\bbW(\teta_\e)\,\VolumeE \\
&\qquad -\frac{1}{\log(1/\delta)} \int_{B_{R_\e} \setminus B_{\delta R_\e}} \chi_\e |\teta_\e|^2 
\frac{\omega(\e |\v_\e| |\teta_\e|)}{\e^2 |\v_\e|^2 |\teta_\e|^2}\,\VolumeE - \alpha_{\e,\delta},
\end{split}
\]
where $\omega:\R_+\to\R$ satisfies $\omega(x)/x^2\to0$ as $x\to0$.
Since the condition that $n_\e\to \infty$ implies that $\e |\v_\e| |\teta_\e|\to0$ uniformly in the support of $\chi_\e$, and since the sequence $\teta_\e$ is bounded in $L^2$, we may modify the infinitesimal sequence $\alpha_{\e,\delta}$ to obtain that 
\[
\frac{\hE_{\e,\delta}^{R_\e}(f_\e;\v_\e)}{|\v_\e|^2} \ge
\frac{1}{\log(1/\delta)}  
\int_{B_{R_\e} \setminus B_{\delta R_\e}} 
\bbW(\chi_\e \teta_\e)\,\VolumeE  - \alpha_{\e,\delta},
\]
where we also used the fact that $\bbW$ quadratic and that $\chi_\e$ is an indicator function. Letting $\e\to0$,
\[
\LiminfEps \frac{\hI_{\e,\delta}^{R_\e}(\v_\e)}{|\v_\e|^2} \ge
\LiminfEps \frac{1}{\log(1/\delta)}  
\int_{B_{R_\e} \setminus B_{\delta R_\e}} 
\bbW(\chi_\e \teta_\e)\,\VolumeE.
\]

At this stage we focus on the right-hand side. 
We rescale the domain by a factor $R_\e$ in order to obtain an $\e$-independent domain, so that we can use properties of weak convergence and convergence in measure.
Denote by $m_\e: B_1 \to B_{R_\e}$ the rescaling function. 
Then,
\[
\frac{1}{\log(1/\delta)}  
\int_{B_{R_\e} \setminus B_{\delta R_\e}} 
\bbW(\chi_\e \teta_\e)\,\VolumeE = 
\frac{1}{\log(1/\delta)}  
\int_{B_1 \setminus B_\delta} 
\bbW(\tilde{\chi}_\e \tilde{\teta}_\e)\,\VolumeE,
\]
where $\tilde{\chi}_\e : B_1 \setminus B_\delta\to\R$ is given by
\[
\tilde{\chi}_\e = \chi_\e \circ m_\e,
\]
and $\tilde{\teta}_\e : B_1 \setminus B_\delta\to\R^2\otimes\R^2$ is given by
\[
\tilde{\teta}_\e = \teta_\e\circ dm_\e = R_\e\, \eta_\e\circ \hZ_{\e\v_\e}^{-1}\circ m_\e.
\]
From \eqref{eq:supp_chi_e} it follows that the cutoff function $\tilde{\chi}_\e$ converges to 1 boundedly in measure, whereas $\tilde{\teta}_\e$ is bounded in $L^2(B_1\setminus B_\delta)$, hence has a weakly-converging subsequence, converging to, say, $\eta_0$. 
Since the product of an $L^2$-weakly-converging sequence and a sequence converging boundedly in measure converges weakly in $L^2$ to the product of the limits, it follows that
\[
\tilde{\chi}_\e \tilde{\teta}_\e \weakly \eta_0 
\qquad \text{in $L^2(B_1\setminus B_{\delta};\R^2\otimes R^2)$}.
\]
It further follows from the weak lower-semicontinuity of quadratic functionals that
\[
\begin{split}
\LiminfEps  \frac{ \hI_{\e,\delta}^{R_\e}(\v_\e)}{|\v_\e|^2} 
& \ge \frac{1}{\log(1/\delta)}\int_{B_1\setminus B_\delta}  
\bbW\brk{\eta_0} \VolumeE.
\end{split}
\]

We now show that 
\[
d\eta_0 = 0
\Textand
\oint_{C} \eta_0 = - \v,
\]
for any positively-oriented loop $C \subset B_1\setminus B_\delta$ homotopic to $\pl B_\delta$.
That is, $\eta_0\in X_\delta^1(\v)$ as defined in \eqref{eq:Xdeltav}.
This will complete the proof as it will follow that
\[
\begin{split}
\LiminfEps  \frac{ \hI_{\e,\delta}^{R_\e}(\v_\e)}{|\v_\e|^2} 
& \ge \inf_{\beta\in X_\delta^1(\v)} \frac{1}{\log(1/\delta)}\int_{B_1\setminus B_\delta}  
\bbW\brk{\beta} \VolumeE = \Ilin_\delta(\v),
\end{split}
\]
which contradicts \eqref{eq:Iquad_contradiction}.

To show that $\eta_0\in X_\delta^1(\v)$, note first that from the definition \eqref{eq:eta_e} of $\eta_\e$, the $\R^2$-valued one forms $\eta_\e \circ \hP_{\e\v_\e}$ satisfy
\[
d(\eta_\e \circ \hP_{\e\v_\e}) = 0
\Textand
\oint_{C_\e} (\eta_\e \circ \hP_{\e\v_\e}) =  -\frac{1}{\e |\v_\e|} \oint_{C_\e} \hP_{\e\v_\e} = - \frac{\v_\e}{|\v_\e|},
\]
for any positively-oriented loop $C_\e$ homotopic to the inner boundary of $\hM^{R_\e}_{\e\v_\e}\setminus \hM^{\delta R_\e}_{\e\v_\e}$.
Changing variables, the same is true for the $\R^2$-valued one forms $\tilde{\eta}_\e \circ \hP_{\e\v_\e} \circ d\hZ_{\e\v_\e}^{-1}$ on $B_{R_\e} \setminus B_{\delta R_\e}$.
By the same argument, the $\R^2$-valued one forms $\teta_\e \circ \hP_{\e\v_\e} \circ d\hZ_{\e\v_\e}^{-1} \circ dm_\e$ on $B_{1} \setminus B_{\delta}$ satisfy
\[
d(\teta_\e \circ \hP_{\e\v_\e} \circ d\hZ_{\e\v_\e}^{-1} \circ dm_\e) = 0
\Textand
\oint_C \teta_\e \circ \hP_{\e\v_\e} \circ d\hZ_{\e\v_\e}^{-1} \circ dm_\e = - \frac{\v_\e}{|\v_\e|},
\]
for any positively-oriented loop $C$ homotopic to the inner boundary of $B_1\setminus B_\delta$.
From Proposition~\ref{prop:conv_manifolds}, 
\[
\begin{split}
&\|\tilde{\teta}_\e - \teta_\e \circ \hP_{\e\v_\e} \circ d\hZ_{\e\v_\e}^{-1} \circ dm_\e\|_{L^2(B_1\setminus B_\delta)} \\
&\qquad\le \|\tilde{\eta}_\e\|_{L^2(B_{R_\e}\setminus B_{\delta R_\e})}\,\|I - \hP_{\e\v_\e} \circ d\hZ_{\e\v_\e}^{-1}\|_{L^\infty(B_{R_\e}\setminus B_{\delta R_\e})} \|dm_\e\|_{L^\infty(B_1\setminus B_\delta)} \\
&\qquad \lesssim \frac{\e\v_\e}{\delta R_\e} R_\e\le \frac{\e n_\e}{\delta},
\end{split}
\]
which is negligible as $\e\to 0$.
Thus $\teta_\e \circ \hP_{\e\v_\e} \circ d\hZ_{\e\v_\e}^{-1} \circ dm_\e \weakly \eta_0$ in $L^2$.

Since being curl-free is preserved under weak-$L^2$-convergence, and since the circulation of a curl-free vector field on an annulus is weak-$L^2$-continuous, it follows that $\eta_0\in X_\delta^1(\v)$.
\end{proof}

The following proposition will be a main component in obtaining a lower bound in the $\Gamma$-convergence analysis in \secref{sec:GammaConv}. At this point, the roles of $n_\e$ and $R_\e$ may be obscure, however they will clarify when we consider bodies with multiple dislocations.

\begin{proposition}\label{prop:lower_bnd_single2}
Fix $\delta\in (0,1/10)$, fix a sequence $n_\e>1$ satisfying $n_\e\to\infty$ and $\log n_\e \ll \log(1/\e)$, and fix $s\in(1/2,1)$.
Then for any sequence $\v_\e \in \R^2$ satisfying $|\v_\e|\le n_\e$,
and every sequence $R_\e \ge \e n_\e^{(2-s)/(1-s)}$, 
\[
\hI_{\e,3\e|\v_\e|/R_\e}^{R_\e}(\v_\e) \ge s\, \Ilin_0(\v_\e)\brk{1 - C \brk{\frac{1}{\log(1/\delta)} + \tilde{\sigma}_{\e,\delta}}},
\]
for some universal constant $C>0$, where $\tilde{\sigma}_{\e,\delta}$ is non-negative sequence (depending only on $n_\e$) satisfying $\lim_{\e\to 0} \tilde{\sigma}_{\e,\delta} = 0$.
\end{proposition}

\begin{proof}
This proof also follows a strategy similar to step 2 in the proof of \cite[Proposition 3.11]{SZ12}.
Let $f_\e \in H^1(\hM_{\e \v_\e};\R^2)$.
Define the annuli of aspect ratio $\delta$,
\[
A_{\e}^k = \{(r,\vp) ~:~ R_\e \delta^k < r < R_\e \delta^{k-1}\},
\qquad k=1,\dots,\kmax,
\]
where
\[
\kmax = \left\lfloor s \frac{\log(R_\e/3\e|\v_\e|)}{\log(1/\delta)} \right\rfloor.
\]
That is, the union of these annuli covers the submanifold
\[
\{(r,\vp) ~:~ R_\e^{1-s} (3\e|\v_\e|)^s\lesssim r \le R_\e\} \subset \hM_{\e\v_\e}^{R_\e}.
\]
Note that by the bounds on $\v_\e$, $R_\e$ and $s$,
\beq
\kmax \ge \frac{s}{1-s}\frac{\log(n_\e/2)}{\log(1/\delta)} -1 \ge \frac{\log(n_\e/2)}{\log(1/\delta)} -1.
\label{eq:we_should_refer_to_it}
\eeq
Now,
\[
\begin{split}
\hE_{\e,3\e|\v_\e|/R_\e}^{R_\e}(f_\e;\v_\e) 
&> \frac{1}{\e^2\log(R_\e/3\e|\v_\e|)} \sum_{k=1}^{\kmax} \int_{A_\e^k} \calW(df_\e\circ\hP_{\e\v}^{-1})\,\hVolumeEVE \\
&= s \frac{\log(1/\delta)}{s\,\log(R_\e/3\e|\v_\e|)} \sum_{k=1}^{\kmax} \frac{1}{\e^2\,\log(1/\delta)} \int_{A_\e^k} \calW(df_\e\circ\hP_{\e\v_\e}^{-1})\,\hVolumeEVE \\
&\ge s \frac{1}{\kmax+1} \sum_{k=1}^{\kmax} \frac{1}{\e^2\,\log(1/\delta)} \int_{A_\e^k} \calW(df_\e\circ\hP_{\e\v_\e}^{-1})\,\hVolumeEVE \\
&\ge s  \frac{\kmax}{\kmax+1}\min_{k=1}^{\kmax} \frac{1}{\e^2\,\log(1/\delta)} \int_{A_\e^k} \calW(df_\e\circ\hP_{\e\v_\e}^{-1})\,\hVolumeEVE \\
&\ge s \frac{\kmax}{\kmax+1}\min_{k=1}^{\kmax} \frac{1}{\e^2\,\log(1/\delta)} \inf_{f\in \HOne(A_\e^k;\R^2)} \int_{A_\e^k} \calW(df\circ\hP_{\e\v_\e}^{-1})\,\hVolumeEVE \\
&\ge s \frac{\kmax}{\kmax+1}\min_{k=1}^{\kmax} \hI^{R_\e\delta^{k-1}}_{\e,\delta}(\v_\e) \\
&\ge s \brk{1- \frac{\log(1/\delta)}{\log(n_\e/2)}}\min_{k=1}^{\kmax} \hI^{R_\e\delta^{k-1}}_{\e,\delta}(\v_\e),
\end{split} 
\]
where in the last passage we used \eqref{eq:we_should_refer_to_it}.
By our assumptions on $\v_\e$ and $R_\e$, and since by the definition of $\kmax$,
\[
\log(1/\delta^{\kmax}) \le s\log(R_\e/3\e|\v_\e|),
\]
it follows that
\[
\delta(R_\e \delta^{\kmax-1})
\ge R_\e^{1-s}\e^s |\v_\e|^s 
\ge \e n_\e^{2-s} |\v_\e|^s 
\ge n_\e \e |\v_\e|,
\]
hence Proposition~\ref{prop:cell_liminf} can be applied to each $\hI^{R_\e\delta^{k-1}}_{\e,\delta}(\v_\e)$, $k=1,\ldots, \kmax$, yielding
\[
\hE_{\e,3\e|\v_\e|/R_\e}^{R_\e}(f_\e;\v_\e) \ge s\brk{1- \frac{\log(1/\delta)}{\log(n_\e/2)}}\brk{\Ilin_\delta(\v_\e) - |\v_\e|^2\, \sigma_{\e,\delta}} \ge s\brk{\Ilin_\delta(\v_\e) - |\v_\e|^2\, \tilde{\sigma}_{\e,\delta}},
\]
where 
\[
\tilde{\sigma}_{\e,\delta} = \brk{1- \frac{\log(1/\delta)}{\log(n_\e/2)}}\sigma_{\e,\delta} + \frac{\log(1/\delta)}{\log(n_\e/2)}\max_{\v\ne 0}\frac{\Ilin_\delta(\v)}{|\v|^2}.
\]
Using \eqref{eq:Idelta_I0} and the positive definiteness of $\Ilin_0$, 
\[
\begin{split}
\hE_{\e,3\e|\v_\e|/R_\e}^{R_\e}(f_\e;\v_\e) 
&\ge s\brk{\Ilin_0(\v_\e) - |\v_\e|^2\brk{\frac{C}{\log(1/\delta)}+ \tilde{\sigma}_{\e,\delta}}} \\
&\ge s\Ilin_0(\v_\e)\brk{1 - C\brk{\frac{1}{\log(1/\delta)}+ \tilde{\sigma}_{\e,\delta}}}.
\end{split}
\]
Taking the infimum over $f_\e \in H^1(\hM_{\e \v_\e};\R^2)$ completes the proof.
\end{proof}

\subsubsection{Upper bounds}

The following proposition shows that the lower bound obtained in Proposition~\ref{prop:cell_liminf} is asymptotically tight for small $\delta$:

\begin{proposition}
\label{prop:cell_limsup}
Let $\beta_\v \in \Omega^1(\R^2\setminus \{0\}; \R^2)$ be the distributional solution of \eqref{eq:betav}. 
Define $f_\e\in C^\infty(\hM_{\e\v};\R^2)$ by
\[
df_\e = \hP_{\e\v} + \e\beta_\v \circ d\hZ_{\e\v}.
\]
Then
\[
\LimEps \hE^R_{\e,\delta}(f_\e;\v)  - \Ilin_\delta(\v) \lesssim \frac{|\v|^2}{\log(1/\delta)},
\]
from which follows that
\[
\LimsupEps \hI_{\e,\delta}^R(\v)  - \Ilin_\delta(\v) \lesssim \frac{|\v|^2}{\log(1/\delta)}.
\]
\end{proposition}

\begin{proof}
The $\R^2$-valued 1-forms $\hP_{\e\v} + \e\beta_\v \circ d\hZ_{\e\v}$ are closed, since $\beta_\v$ is curl-free. Furthermore, for every curve $C$ homotopic to $\pl B_\delta$,
\[
\oint_{\hZ_{\e\v}^{-1}(C)} (\hP_{\e\v} + \e\beta_\v \circ d\hZ_{\e\v}) = \e\v + \e\oint_C\beta_\v = 0.
\]
It follows that $\hP_{\e\v} + \e\beta_\v \circ d\hZ_{\e\v}$ is exact, i.e., it is the differential of a mapping $f_\e$, which is smooth.

We write
\[
\begin{split}
\hE_{\e,\delta}^R(f_\e;\v) &= 
\frac{1}{\e^2\,\log(1/\delta)} \int_{\hM^R_{\e\v}\setminus \hM^{\delta R}_{\e\v}} \calW(I + \e \eta_\e)\,\hVolumeEV 
\end{split}
\]
where 
\[
\eta_\e = \beta_\v \circ d\hZ_{\e\v}\circ \hP_{\e\v}^{-1} \in C^\infty(\hM_{\e\v};\R^2\otimes \R^2). 
\]
Changing variables using $\hZ_{\e\v}:\hM^R_{\e\v}\setminus \hM^{\delta R}_{\e\v} \to B_R\setminus B_{\delta R}$, and denoting $\tilde\eta_\e = \eta_\e\circ \hZ_{\e\v}^{-1}$,
\[
\hE_{\e,\delta}^R(f_\e;\v)
= \frac{1}{\e^2\,\log(1/\delta)} \int_{B_R\setminus B_{\delta R}} \calW(I + \e \tilde\eta_\e)\,(\hZ_{\e\v}^{-1})^\#\hVolumeEV.
\]
By \eqref{eq:asymp_par}, $\tilde{\eta}_\e$ converges uniformly on $B_R\setminus B_{\delta R}$ to $\beta_\v$ as $\e\to0$, hence so does
\[
\frac{1}{\e^2}  \calW(I + \e \tilde\eta_\e) \to \bbW(\beta_\v).
\]
Since, furthermore,
\[
(\hZ_{\e\v}^{-1})^\#\hVolumeEV \to \VolumeE
\]
uniformly on $B_R\setminus B_{\delta R}$, we obtain, using the self similarity of $\beta_\v$, that
\[
\LimEps \hE_{\e,\delta}^R(f_\e;\v) = \Elin_\delta(\beta_\v;\v) \le \Ilin_\delta(\v) + \frac{Cb^2}{\log(1/\delta)},
\]
where the last inequality follows from \eqref{eq:Edelta_betav}.
This completes the proof.
\end{proof}

So far, we have been considering two energy functions: a nonlinear energy function (for configurations) with density $\calW$ on the dislocated body, and a linear energy function (for admissible strains) with density $\bbW$ on a Euclidean domain. We now analyze an intermediate energy with density $\bbW$ on the dislocated body, acting on displacements,
and obtain more detailed bounds for a fixed $\e$.
These will be needed for constructing a recovery sequence in Section~\ref{sec:recovery_seq}.

\begin{proposition}
\label{prop:cell_limsup_Quad}
For $\v_\e\in \R^2$, let $\beta_{\v_\e}\in \Omega^1(\R^2\setminus\{0\};\R^2)$ be the distributional solution of \eqref{eq:betav},
and define $f_\e\in \HOne(\hM_{\e\v_\e};\R^2)$ by
\[
df_\e = \hP_{\e\v_\e} + \e\beta_{\v_\e} \circ d\hZ_{\e\v_\e}.
\]
For every $\delta\in (0,1/10)$, for every sequence $\v_\e\in\R^2$ and for every sequence $R_\e>0$ satisfying $\delta R_\e \ge \e |\v_\e|$,
\[
\frac{1}{\e^2\log(1/\delta)}\int_{\hM_{\e\v_\e}^{R_\e}\setminus \hM_{\e\v_\e}^{\delta R_\e}} \bbW(df_\e\circ \hP_{\e\v_\e}^{-1}-I)\, \hVolumeEVE = \Ilin_\delta(\v_\e)\brk{1 + O\brk{\frac{\e|\v_\e|}{\delta R_\e} + \frac{1}{\log(1/\delta)}}}.
\]
\end{proposition}

\begin{proof}
First, by the uniform bilipschitz bound \eqref{eq:bilipschitz}  of $\hZ_{\e\v_\e}$ and the property \eqref{eq:beta_bound} of $\beta_{\v_\e}$,
\beq
\label{eq:beta_v_bound}
|df_\e - \hP_{\e\v_\e}|_{\hg_{\e\v_\e},\euc} = 
\e |\beta_\v|_{\euc,\euc} | d\hZ_{\e\v_\e}|_{\hg_{\e\v_\e},\euc} \lesssim \frac{\e|\v_\e|}{r}.
\eeq
Using the fact that $\bbW$ is quadratic, 
\[
\begin{split}
\frac{1}{\e^2}\bbW(df_\e\circ \hP_{\e\v_\e}^{-1}-I)
&= \bbW(\beta_{\v_\e} \circ d\hZ_{\e\v_\e}\circ \hP_{\e\v_\e}^{-1}) \\
&= \bbW(\beta_{\v_\e})\circ\hZ_{\e\v_\e} + 
O(|\beta_{\v_\e}|_{\euc,\euc}^2 |d\hZ_{\e\v_\e}-\hP_{\e\v_\e}|_{\hg_{\e\v_\e},\euc}) \\
&= \bbW(\beta_{\v_\e})\circ\hZ_{\e\v_\e} +
O(\e|\v_\e|^3 /r^3),
\end{split}
\]
where in the passage to the third line we used \eqref{eq:asymp_par}.
Integrating,
\[
\begin{split}
&\frac{1}{\e^2\log(1/\delta)}\int_{\hM_{\e\v_\e}^{R_\e}\setminus \hM_{\e\v_\e}^{\delta R_\e}} 
\bbW(df_\e\circ \hP_{\e\v_\e}^{-1}-I)\, \hVolumeEVE \\
&\qquad = \frac{1}{\log(1/\delta)}\int_{\hM_{\e\v_\e}^{R_\e}\setminus \hM_{\e\v_\e}^{\delta R_\e}} \bbW(\beta_{\v_\e}) \circ \hZ_{\e\v_\e}\, \hVolumeEVE + 
O\brk{\frac{1}{\log(1/\delta)} \frac{\e|\v_\e|^3}{\delta R_\e}} \\
&\qquad = \frac{1}{\log(1/\delta)}\int_{\hM_{\e\v_\e}^{R_\e}\setminus \hM_{\e\v_\e}^{\delta R_\e}} \bbW(\beta_{\v_\e}) \circ \hZ_{\e\v_\e} \, \hVolumeEVE + 
 \Ilin_\delta(\v_\e) O\brk{\frac{1}{\log(1/\delta)}},
\end{split}
\]
where in the passage to the last line we used the fact that $ \Ilin_\delta(\v_\e) \gtrsim |\v_\e|^2$ and the bound $\delta R_\e \ge \e |\v_\e|$.
Changing variables with $\hZ_{\e\v_\e}$, using the first inequality in \eqref{eq:volume_change} and \eqref{eq:Edelta_betav}, we obtain
\[
\begin{split}
&\frac{1}{\e^2\log(1/\delta)}\int_{\hM_{\e\v_\e}^{R_\e}\setminus \hM_{\e\v_\e}^{\delta R_\e}} \bbW(df_\e\circ \hP_{\e\v_\e}^{-1}-I)\, \hVolumeEVE \\
&\qquad = \frac{1+O(\e|\v_\e|/\delta R_\e)}{\log(1/\delta)}\int_{B_{R_\e}\setminus B_{\delta R_\e}} \bbW(\beta_{\v_\e})\, \VolumeE +  \Ilin_\delta(\v_\e) O\brk{\frac{1}{\log(1/\delta)}}\\
&\qquad = \Ilin_\delta(\v_\e)\brk{\brk{1 + O\brk{\frac{1}{\log(1/\delta)}}}\brk{1 + O\brk{\frac{\e|\v_\e|}{\delta R_\e}} }+ O\brk{\frac{1}{\log(1/\delta)}}},
\end{split}
\]
from which the claim follows.
\end{proof}

\begin{proposition}
\label{prop:cell_limsup_Quad2}
Fix $C>0$, fix $s\in (0,1)$ and fix a compact set $K\subset \R^2 \setminus \{0\}$.
Then for every sequence $\v_\e\in K$ and every bounded sequence $R_\e>0$ satisfying $\log(1/R_\e) \ll \log(1/\e)$,
\beq\label{eq:Quad_df_e_optimal}
\begin{split}
&\frac{1}{\e^2\log(1/\e)}\int_{\hM_{\e\v_\e}^{R_\e}\setminus \hM_{\e\v_\e}^{C\e^s}} \bbW(df_\e\circ \hP_{\e\v_\e}^{-1}-I)\, \hVolumeEVE \\
&\qquad = \Ilin_0(\v_\e)\brk{s + O\brk{\e^{1-s}+ \frac{|\log R_\e|+1}{\log(1/\e)}}},
\end{split}
\eeq
where $f_\e$ is as in Proposition~\ref{prop:cell_limsup_Quad}, and the constants are independent of $s$ (they depend  on $C$ and $K$). 
\end{proposition}

\begin{proof}
Let $b = \max_{\v \in K}|\v|$.
First, note that since $R_\e \gg \e^s \gg \e b \ge \e|\v|$, the domain $\hM_{\e\v_\e}^{R_\e}\setminus \hM_{\e\v_\e}^{C\e^s}$ is well-defined.
Let $\delta\in(0,1)$ be such that $\delta R_\e = C\e^s$, and note that
\[
\frac{\log(1/\delta)}{\log(1/\e)} = s + O\brk{\frac{|\log R_\e|+1}{\log(1/\e)}} = O(1).
\]
Applying Proposition~\ref{prop:cell_limsup_Quad} for this choice of $\delta$, we obtain
\[
\begin{split}
&\frac{1}{\e^2\log(1/\e)}\int_{\hM_{\e\v_\e}^{R_\e}\setminus \hM_{\e\v_\e}^{C\e^s} }\bbW(df_\e\circ \hP_{\e\v_\e}^{-1}-I)\, \hVolumeEVE\\
& \qquad = \Ilin_{\delta}(\v_\e)\brk{1 + O\brk{\e^{1-s} + \frac{1}{\log(1/\delta)}}}\brk{s + O\brk{\frac{|\log R_\e|+1}{\log(1/\e)}}} \\
& \qquad = \Ilin_{\delta}(\v_\e)\brk{s + O\brk{\e^{1-s} + \frac{|\log R_\e|+1}{\log(1/\e)}}} \\
& \qquad = \Ilin_{0}(\v_\e)\brk{s + O\brk{\e^{1-s} + \frac{|\log R_\e|+1}{\log(1/\e)}}},
\end{split}
\]
where the last line follows from \eqref{eq:Idelta_I0}.
\end{proof}

The following lemma allows us to adjust the boundary values of the asymptotically-optimal maps $f_\e$ of Proposition~\ref{prop:cell_limsup_Quad2}:

\begin{lemma}
Let $\v_\e \in \R^2$, and let $R_\e >0$ be a bounded sequence satisfying $R_\e > 10 |\v_\e|$.
Let $\ZEps\in H^1(\hM_{\e\v_\e}^{R_\e};\R^2)$ be a map satisfying
\beq\label{eq:dZ_e_bound}
|d\ZEps - \hP_{\e \v_\e}| \lesssim \frac{\e|\v_\e|}{\r}.
\eeq
Then, for every $f_\e\in \HOne(\hM_{\e\v_\e}^{R_\e};\R^2)$ satisfying
\beq\label{eq:df_e_bound}
|df_\e - \hP_{\e \v_\e}| \lesssim \frac{\e|\v_\e|}{\r},
\eeq
there exists a $\tf_\e\in \HOne(\hM_{\e\v_\e}^{R_\e};\R^2)$ satisfying the bound \eqref{eq:df_e_bound}, such that
\[
\tf_\e = \ZEps \qquad \text{for $r = R_\e$}, 
\]
and
\beq
\label{eq:tf_energy}
\int_{\hM_{\e\v_\e}^{R_\e}} \bbW(d\tf_\e\circ \hP_{\e\v_\e}^{-1}-I)\, \hVolumeEVE < 
\int_{\hM_{\e\v_\e}^{R_\e}} \bbW(df_\e\circ \hP_{\e\v_\e}^{-1}-I)\, \hVolumeEVE + C|\v_\e|^2,
\eeq
where $C$ is independent of $\e$, $R_\e$, $\v_\e$, and depends on $f_\e$ and $\ZEps$ only through the constants in their pointwise bounds \eqref{eq:dZ_e_bound}--\eqref{eq:df_e_bound}.
\end{lemma}

\begin{proof}
Partition the domain into annuli having (in coordinates) fixed aspect ratio,
\[
A_k = \{(r,\vp) ~:~ 2^{-k-1}R_\e < r < 2^{-k}R_\e\}, \qquad k= 0,\dots,k_{\max}.
\]
Suppose that for every $k$
\[
\int_{A_k} \bbW(\ZEps\circ \hP_{\e\v_\e}^{-1}-I)\,\hVolumeEVE \le \int_{A_k} \bbW(df_\e\circ \hP_{\e\v_\e}^{-1}-I)\,\hVolumeEVE.
\]
In such case, the claim follows trivially letting $\tf_\e = \ZEps$.
Otherwise, let $k$ be the smallest natural number for which
\[
\int_{A_k} \bbW(df_\e\circ \hP_{\e\v_\e}^{-1}-I)\,\hVolumeEVE < \int_{A_k} \bbW(d\ZEps\circ \hP_{\e\v_\e}^{-1}-I)\,\hVolumeEVE\lesssim \e^2 |\v_\e|^2,
\]
where the second inequality always holds and follows from \eqref{eq:dZ_e_bound} and the fact that $r\simeq \r$ by Lemma~\ref{lem:frakr=r}.
By \eqref{eq:dZ_e_bound}--\eqref{eq:df_e_bound}, 
\[
|df_\e - d\ZEps|{_{\hg_{\e\v_\e},\euc}} \lesssim \frac{\e|\v_\e|}{r},
\]
and thus, by possibly translating $f_\e$, we have
\[
\|f_\e - \ZEps\|_{L^\infty(A_k)}\lesssim \e|\v_\e|,
\]
since on $A_k$ we have $r \simeq 2^{-k} R_\e$, whereas the diameter of $A_k$ is of the same order.

Define $\tf_\e:\hM_{\e\v_\e}^{R_\e}\to\R^2$ via the requirement that
\[
d\tf_\e = \Cases{
df_\e & r < 2^{-k-1} R_\e  \\
df_\e + d(\vp(r) (\ZEps - f_\e)) &  r\in A_k \\
d\ZEps & r > 2^{-k} R_\e ,
}
\]
where $\vp(r)$ is smooth, equals zero in some neighborhood of $r = 2^{-k-1} R_\e$, equals one in some neighborhood of $r = 2^{-k} R_\e$, and satisfies $|\vp'(r)| \lesssim 2^{k}R_\e^{-1}$.  
It is easy to see that indeed
\[
|d\tf_\e - \hP_{\e\v_\e}| \lesssim \frac{\e|\v_\e|}{\r}.
\]
By the very definition of $k$,
\[
\int_{r\not\in A_k} \bbW(d\tf_\e\circ \hP_{\e\v_\e}^{-1}-I) \,\hVolumeEVE \le \int_{r\not\in A_k} \bbW(df_\e\circ \hP_{\e\v_\e}^{-1}-I) \,\hVolumeEVE.
\]
In the transition annulus, $A_k$, 
\[
d\tf_\e  = \hP_{\e\v_\e} + J_\e,
\]
where
\[
\begin{split}
J_\e &= (df_\e - \hP_{\e\v_\e})  + d(\vp(r) (\ZEps - f_\e)) \\
&= (df_\e - \hP_{\e\v_\e})  + (d\ZEps - df_\e)\vp(r) + (\ZEps - f_\e) \vp'(r) dr,
\end{split}
\]
hence
\[
|J_\e|^2 \lesssim 
|df_\e - \hP_{\e\v_\e}|^2  + |d\ZEps - df_\e|^2 +  \frac{1}{(2^{-k}\rho_\e)^2} |\ZEps - f_\e|^2,
\]
from which we obtain that
\[
\int_{A_k} \bbW(d\tf_\e\circ \hP_{\e\v_\e}^{-1}-I) \,\hVolumeEVE \lesssim \int_{A_k} |J_\e|^2 \,\hVolumeEVE \lesssim |\v_\e|^2 \e^2.
\]
Putting everything together, $\tf_\e$ satisfies the energy bound \eqref{eq:tf_energy}.
\end{proof}

\begin{corollary}
\label{cor:cell_limsup_Quad}
Fix $C>0$, fix $s\in (0,1)$ and fix a compact set $K\subset \R^2 \setminus \{0\}$.
Let $\v_\e\in K$ and let $R_\e>0$ be a bounded sequence satisfying $\log(1/R_\e) \ll \log(1/\e)$.
Let $\ZEps\in H^1(\hM_{\e\v_\e}^{R_\e};\R^2)$ satisfy 
\[
|d\ZEps - \hP_{\e \v_\e}| \lesssim \frac{\e|\v_\e|}{\r}.
\]
Then there exists a function $f_\e\in H^1(\hM_{\e\v_\e}^{R_\e};\R^2)$ satisfying 
\[
f_\e = \ZEps \qquad \text{for $r = R_\e$},
\]
and
\[
\begin{split}
&\frac{1}{\e^2\log(1/\e)}\int_{\hM_{\e\v_\e}^{R_\e}\setminus \hM_{\e\v_\e}^{C\e^s}} \bbW(df_\e\circ \hP_{\e\v_\e}^{-1}-I)\, \hVolumeEVE \\
&\qquad = \Ilin_0(\v_\e)\brk{s + O\brk{\e^{1-s}+ \frac{|\log R_\e|+1}{\log(1/\e)}}},
\end{split}
\]
where the constants are independent of $s$ (they depend on $C$, $K$ and the constant in the pointwise bound on $d\ZEps-\hP_{\e\v_\e}$). 
\end{corollary}

\subsection{The self-energy function}

\begin{definition}
Let $\bbS$ be a dislocation structure.
The \Emph{self-energy function of a dislocation structure} is a function $\SEF:\R^2\to [0,\infty)$ given by
\beq
\SEF(\v) = \inf \BRK{\sum_{i=1}^N \lambda_i \Ilin_0(\v_i) ~:~ \v_i \in \bbS, \, N\in \mathbb{N}, \, \lambda_i >0,\, \sum_{i=1}^N \lambda_i \v_i = \v}.
\label{eq:SigmaS}
\eeq
\end{definition} 

\begin{lemma}
\label{lem:4.15}
The function $\SEF$ satisfies the following properties:
\begin{enumerate}[itemsep=0pt,label=(\alph*)]
\item It is positively 1-homogeneous, i.e., $\SEF(\alpha \v) = \alpha \,\SEF(\v)$ for every $\alpha>0$.
\item It is convex.
\item There exists a constant $K>0$ such that the infimum in \eqref{eq:SigmaS} can be limited to $\v_i\in \bbS$ satisfying $|\v_i|<K$.
\item The infimum in \eqref{eq:SigmaS} is in fact a minimum (with $|\v_i|<K$).
\end{enumerate}
\end{lemma}

\begin{proof}
\begin{enumerate}[itemsep=0pt,label=(\alph*)]
\item For every $\v\in\R^2$ and $\alpha>0$,
\[
\begin{split}
\SEF(\alpha\v) &= \inf \Big\{\sum_{i=1}^n \lambda_i \Ilin_0(\v_i) ~:~ \v_i \in \bbS, \, n\in \mathbb{N}, \, \lambda_i >0,\, \sum_{i=1}^n \lambda_i \v_i = \alpha \v\Big\} \\
&= \alpha \inf \Big\{\sum_{i=1}^n \frac{\lambda_i}{\alpha} \Ilin_0(\v_i) ~:~ \v_i \in \bbS, \, n\in \mathbb{N}, \, \frac{\lambda_i}{\alpha} >0,\, \sum_{i=1}^n \frac{\lambda_i}{\alpha} \v_i = \v\Big\} \\
&= \alpha \, \SEF(\v).
\end{split}
\]
\item Let $\u,\v\in\R^2$, let $\e>0$ and let $\u_1,\dots,\u_m\in \bbS$, $\mu_1,\dots,\mu_m>0$,  $\v_1,\dots,\v_n\in \bbS$ and $\lambda_1,\dots,\lambda_n>0$, such that
\[
\sum_{i=1}^m \mu_i \Ilin_0(\u_i) < \SEF(\u) + \e
\Textand
\sum_{i=1}^n \lambda_i \Ilin_0(\v_i) < \SEF(\v) + \e,
\]
where
\[
\sum_{i=1}^m \mu_i \u_i = \u
\Textand
\sum_{i=1}^n \lambda_i \v_i = \v.
\]
Then, for $t\in(0,1)$,
\[
\sum_{i=1}^m t \mu_i \u_i + \sum_{i=1}^n (1-t) \lambda_i \v_i = t\u + (1-t)\v,
\]
whereas
\[
\sum_{i=1}^m t \mu_i \Ilin_0(\u_i) +  \sum_{i=1}^n (1-t) \lambda_i \Ilin_0(\v_i) < t \SEF(\u) + (1-t) \SEF(\v) + \e,
\]
i.e.,
\[
\SEF(t\u + (1-t)\v) < t \SEF(\u) + (1-t) \SEF(\v) + \e.
\]
Since this holds for every $\e>0$, it follows that $\SEF$ is convex. 

\item  
Since $\Ilin_0(\v)\simeq |\v|^2$, there exists a $c>0$ such that
\[
c|\v|^2  \le \Ilin_0(\v)
\qquad
\text{for every $\v\in\R^2$}.
\]
Let $\{\u_1,\u_2\}\subset\bbS$ be a basis for $\R^2$ and let $\v\in\R^2$ be given by
\[
\v = \sum_{i=1}^2 \lambda_i \xi_i \u_i, \qquad\lambda_i>0, \,\, \xi_i=\pm1.
\]
Note that the map $\v \mapsto \lambda_1 + \lambda_2$ is a norm on $\R^2$, and thus $|\v| \ge c_1 \sum_{i=1}^2\lambda_i$ for some $c_1>0$.
Then, 
\[
\begin{split}
\sum_{i=1}^2 \lambda_i \Ilin_0(\xi_i\u_i) 
&= \sum_{i=1}^2 \lambda_i \Ilin_0(\u_i) \\
&\le \max_i\{\Ilin_0(\u_i)\} \sum_{i=1}^{2} \lambda_i \\
&\le c_1^{-1}\max_i\{\Ilin_0(\u_i)\} |\v| \\
&\le \frac{\max_i\{\Ilin_0(\u_i)\}}{c_1 c |\v|} \Ilin_0(\v)
\end{split}
\]
Thus, 
\[
\sum_{i=1}^2 \lambda_i \Ilin_0(\u_i) < \Ilin_0(\v)
\]
for every $\v\in\bbS$ satisfying
\[
|\v| > \frac1{c_1c} \max_i\{\Ilin_0(\u_i)\} \equiv K.
\]
Hence, the infimum defining $\SEF$ can be taken over $\bbS\cap \overline{B_K}$.

\item Since every bounded subset of $\bbS$ is finite,  by the previous item, 
the infimum in \eqref{eq:SigmaS} can be limited to $\v_i\in \bbS$ in a finite set $\{\u_1,\dots,\u_k\}$. For every 
$\v \in \R^2$, the linear combinations
\[
\sum_{i=1}^k \lambda_i \u_i
\]
can be limited to the compact set
\[
0\le  \lambda_i \le \SEF(\v)/\Ilin_0(\u_i),
\]
hence the infimum is a minimum.
\end{enumerate}
\end{proof}

\section{The geometry of multiple edge-dislocations}

Having constructed body manifolds containing one edge-dislocation, we generalize this construction to bodies containing multiple edge-dislocations, and eventually take their number to infinity.

\subsection{Bodies with multiple edge-dislocations}
\label{sec:def_bodies}

\begin{definition}\label{def:body_w_disloc}
\Emph{A body with $m$ edge-dislocations} 
is a two-dimensional elastic body $(\M,\P)$ (\defref{def:elastic_body}) having finite diameter, and satisfying the following additional properties:
\begin{enumerate}[itemsep=0pt,parsep=0pt,label=(\alph*)]
\item $M$ is a manifold with boundary, diffeomorphic to a plane with $m$ open holes,
\[
\R^2 \setminus \brk{\bigcup_{i=1}^{m} B_{r}(x_i)}.
\]
\item $\P$ is closed, $d\P=0$.
\item Any hole in $\M$ has a neighborhood $\M' \subset \M$ such that $(\M',\P|_{\M'})$
is a body with an edge-dislocation (\defref{def:single_disloc}), having an regular inner boundary (\defref{def:regular_dM}).
In particular, every hole is associated with a Burgers vector $\v^i\in\R^2$, $i=1,\ldots,m$. 

\item Denote by $r_{ij}$, $1\le i < j < m$ the distance between the boundary of the $i$-th and $j$-th hole, and by $\ell_i$ the distance between the $i$-th hole and the outer-boundary of $\M$. Let
\[
\rho =  \min \{r_{ij},\ell_i ~:~ i,j\}.
\]
We assume that 
\[
\rho > \max_i 20|\v^i|.
\]
\end{enumerate}
\end{definition}

To every point $p\in M$, we associate by \eqref{eq:dist_disloc} its distance $\r_i(p)$ to the $i$-th dislocation.
We further denote by $i(p)$ the index of the dislocation closest to $p\in M$ (which is well-defined for almost every $p\in M$), and by $\r(p) = \r_{i(p)}(p)$ the distance to the closest dislocation.

In line with the comments following \defref{def:single_disloc} we note the following:

\begin{enumerate}
\item The Burgers vector associated with the $i$-th dislocation can alternatively be represented by a parallel vector field $\b^i$, where 
\[
\oint_C \Pi^p = \b^i_p
\]
for every simple, closed, oriented loop $C$ surrounding (only) the $i$-th dislocation, and every reference point $p\in\M$. The relation between $\b^i$ and $\v^i$ is $\v^i = \P(\b^i)$, where the right-hand side is a constant function on $\M$.

\item Alternatively, one can think of the Burgers vector as a bounded linear functional (i.e., an $\R^2$-valued measure) , $\torsion:C_c(\M;\R^2)\to\R$, defined by 
\[
\torsion(\psi) =  \int_{\dM} \TR(\P\otimes\psi).
\]
For $\psi\in C_c(\M;\R^2)$ equal to a constant vector $\u^i\in\R^2$ on the boundary of the $i$-th dislocation for $i=1,\dots,m$,
\beq
\label{eq:Tpsi_const_on_bdry}
\torsion(\psi) = \sum_{i=1}^{m} \ip{\v^i,\u^i}.
\eeq
This interpretation of the Burgers vectors will come out handy when considering the limit of infinitely-many dislocations.
Note that for $\|\psi\|_\infty \le 1$,
\[
|\torsion(\psi)| \le \operatorname{Length}(\dM) \lesssim \sum_{i=1}^m |\v^i|,
\]
where $\dM$ in the middle term only accounts for the inner boundary, and the last inequality follows from the regularity of the inner boundary (the Lipschitz equivalence assumption in \defref{def:regular_dM}). 
On the other hand, for $\psi\in C_c(\M;\R^2)$, with $\|\psi\|_\infty\le 1$, satisfying $\psi = \v_i/|\v_i|$ on the boundary of the $i$-th dislocation,
\[
\torsion(\psi) = \sum_{i=1}^m |\v^i|,
\]
hence
\beq
\|\torsion\|_{\calM(\M;\R^2)} \simeq \sum_{i=1}^{m} |\v^i|.
\label{eq:norm_calT}
\eeq
Finally, for $\psi\in \HOneZero(\M)$, since $\P$ is closed,
\beq
\label{eq:Tpsi_H_1_0}
\torsion(\psi) = - \int_\M d \TR (\P\otimes \psi) = \int_{\M} \TR(\P\wedge d\psi).
\eeq

\item If the system is endowed with a dislocation structure $\bbS$ (\defref{def:disloc_structure}), we further impose that $\v^i\in\bbS$ for all $i=1,\dots,m$, or, if in addition a length-scale $\e$ is introduced, $\v^i \in \e \bbS$. 

\item 
Condition (d) implies that the annuli $A_i$ in the definition of a regular inner boundary are separated from each other and from the outer-boundary.
\end{enumerate}

\subsection{Constructing bodies with multiple dislocations}
\label{sec:disloc_construction}

This section is analogous to Sections~\ref{sec:coordinate_construction}--\ref{sec:dev_from_euc}:
we construct bodies with multiple edge-dislocations, and estimate their geometric deviation from a Euclidean, defect-free body.

One way of constructing bodies with multiple dislocations is by smoothly gluing bodies, each having a single edge-dislocation. 
For example, one can construct bodies with dislocations as in \secref{sec:coordinate_construction}, cut out subsets having rectangular boundaries, and smoothly glue one rectangle to the other. 
Such an approach was used in previous work \cite{KM15,KM16,EKM20}. 
Its upside is that the geometry of the inner boundaries is known explicitly; its drawback, however, is the difficulty to obtain sharp enough energy estimates on the geometric deviation of such bodies from a Euclidean domain.
Thus, we use here a different approach, borrowing ideas from the construction of strains in \cite[Theorem~4.6]{MSZ14} to construct an implant map $\P$.

Let $\W\subset \R^2$ be a bounded Lipschitz domain; a subset of $\W$ will serve both as a body manifold, and as the Euclidean domain to which the body is compared.
Let $p_1,\ldots,p_m\in \W$ and $\v^1,\ldots, \v^m \in \R^2$ be given. 
We think of $p_i$ as the locus  of the $i$-th dislocations and of $\v^i$ as its Burgers vector .
Denote 
\[
b = \max_{i} |\v^i|,
\]
and 
\[
a = \min\BRK{\frac{1}{3} \min_{i\ne j} |p_i - p_j|, \, \frac12 \min_{i} \dist(p_i,\pl \W)}.
\]
That is, the discs $B_a(p_i)$ are disjoint, separated from each other and from the boundary by a distance of at least $a$. We further assume that $10b<a$.

By Comment 2 following \defref{def:body_w_disloc}, burgers vectors can be identified with $\R^2$-valued measures on the body manifold. We introduce two $\R^2$-valued Radon measures on $\W$,
\[
\mu = \sum_{i=1}^m \v^i  \otimes \delta_{p_i}  
\Textand
d\tmu = \sum_{i=1}^m  \v^i \otimes \frac{\chi_{B_a(p_i)}}{\pi a^2} \,\VolumeE  ,
\]
where the second measure is a ``smeared" version of the first.

In order to define a body with edge-dislocations, we specify a manifold $M\subset\W$, along with a closed frame field $\P\in\W^1(\M;\R^2)$.
For $i=1,\ldots,m$, we introduce shifted coordinates $(x_i,y_i) = (x,y) - p_i$, and set $r_i = (x_i^2 + y_i^2)^{1/2}$.
Define first the following 1-forms on $\W\setminus \{p_1,\ldots,p_m\}$,
\[
\alpha_i = \frac{1}{2\pi}\brk{-\frac{y_i}{r_i^2}dx + \frac{x_i}{r_i^2}dy} \chi_{B_a(p_i)} 
\Textand
\alpha = \sum_{i=1}^m   \v^i \otimes \alpha_i,
\]
i.e., $\alpha$ is a discontinuous $\R^2$-valued 1-form, whose support is a disjoint union of $a$-neighborhoods of the points $p_i$. 
The forms $\alpha_i$ are closed in each of the punctured balls $B_a(p_i)$ and trivially closed in their complements. For a positively-oriented loop $C_i$ in $B_a(p_i)$ homotopic to $\pl B_a(p_i)$,
\[
\oint_{C_i} \alpha_i = 1,
\qquad\text{hence}\qquad
\oint_{C_i} \alpha = \v_i.
\]
Thus, the $\R^2$-valued form $\alpha$ satisfies the circulation condition required by $\P$ when restricted to the union of the sets $\{r_i<a\}$. 
The problem is that $\alpha$ is discontinuous on the circles $r_i = a$. To correct this, we define 1-forms on $\W$,
\[
\beta_i = \frac{1}{2\pi}\brk{-\frac{y_i}{a^2}dx + \frac{x_i}{a^2}dy} \chi_{B_a(p_i)}
\Textand
\beta = \sum_{i=1}^m   \v^i \otimes \beta_i,
\]
which coincide with $\alpha_i$ and $\alpha$, respectively, on the circles $r_i = a$. 
The ``corrected" $\R^2$-valued 1-form $\alpha-\beta$ is continuous in $\W\setminus \{p_1,\ldots,p_m\}$, however it is not closed and does not satisfy the required circulation around the points $p_i$. To retrieve the closedness and the circulation while retaining continuity, we introduce an additional correction, defining 
1-forms $\gamma_i$ on $\W$ solving the elliptic first-order differential system
\[
\begin{cases}
d\gamma_i = \frac{1}{\pi a^2}\chi_{B_a(p_i)} \VolumeE
\textand d^*\gamma_i=0 & \text{in $\W$} \\
\gamma_i(\frakn) = 0 & \text{on $\pl\W$},
\end{cases}
\]
where $d^*$ denotes the codifferential and $\frakn$ is the unit normal to the boundary. 
Then, we set
\[
\gamma = \sum_{i=1}^m  \v^i \otimes \gamma_i.
\]
Note that $d\gamma_i = d\beta_i$ in the sets $\{r_i<a\}$ and $\{r_i>a\}$, however, $\gamma_i$, unlike $\beta_i$ is continuous. 
Finally, let 
\[
\P = \IdRtwo + \alpha - \beta + \gamma.
\]
By construction, $\P$ is continuous in $\W\setminus \{p_1,\ldots,p_m\}$, and 
\[
d\P = 0,
\]
i.e., $\P$ is closed in $\W\setminus \{p_1,\ldots,p_m\}$. 
Moreover, since $\IdRtwo$ and $\beta - \gamma$ are exact in the discs $B_a(p_i)$, it follows that for every positively-oriented loop $C_i$ homotopic to $\pl B_a(p_i)$, 
\[
\oint_{C_i} \P  =  \v_i.
\]
In terms of distributional derivatives, $d\P  =\mu$ in $\W$. 

For $\P$ to qualify as an implant map of a body with $m$ dislocations, it must be a frame field, which is only guaranteed far enough from the points $p_i$, i.e., on a submanifold $\M\subset \W\setminus \{p_1,\ldots,p_m\}$. To determine $\M$, we need uniform estimates on $\alpha-\beta+\gamma$:

\begin{lemma}
\label{lem:uniform_estimates}
The following inequalities hold,
\[
\begin{gathered}
\|\alpha\|_{L^\infty\brk{\W\setminus \brk{\bigcup_{i=1}^m B_{3|\v^i|/2}(p_i)}}} < \frac19 \\
\|\beta\|_{L^\infty(\W)} \le \frac{b}{2\pi a} < \frac{1}{60}  \\
\|\gamma\|_{L^\infty(\W)} \lesssim \frac{b}{a^2},
\end{gathered}
\]
where all the norms are with respect to the Euclidean metric on $\W$.
\end{lemma}

While the ratio $b/a^2$ may look strange in terms of dimensions, the constants in the third inequality involve geometric properties of $\W$ that make the estimate of $\|\gamma\|_{L^\infty(\W)}$ dimension-free.

\begin{proof}
The estimates for $\alpha$ and $\beta$ are immediate, using the fact that the supports of $\{\alpha_i\}_{i=1}^m$ and $\{\beta_i\}_{i=1}^m$ are pairwise-disjoint.

Fix $q>2$. Then,
\[
\begin{split}
\|\gamma\|_{L^\infty(\W)} &\lesssim \|\gamma\|_{W^{1,q}(\W)} \lesssim \|d\gamma\|_{L^q(\W)} 
\lesssim \brk{\sum_{i=1}^m \brk{\frac{|\v^i|}{a^2}}^q a^2}^{1/q} 
\lesssim (m  a^2)^{1/q} \frac{b}{a^2} 
\lesssim \frac{b}{a^2},
\end{split}
\]
where the first estimate follows from the Sobolev embedding $L^\infty \hookrightarrow W^{1,q}$, the second estimate follows from elliptic regularity \cite[Theorem 3.2.5]{Sch95}, the third passage follows from an explicit substitution of $d\gamma$, and in the last passage we used the geometric volume bound $m a^2 \lesssim 1$. 
\end{proof}

\begin{proposition}
\label{prop:construct_many_disloc}
There exists a constant $c = c(\W)>0$ such that if $b/a, b/a^2<c$, then $\|-\beta + \gamma\|_{L^\infty(\W)} < 1/9$,
$\P$ is non-degenerate in $\W_\v = \W \setminus \brk{\bigcup_{i=1}^m B_{3|\v^i|/2}(p_i)}$, and $(\W_\v, \P)$ is a body with $m$ dislocations according to \defref{def:body_w_disloc}.
\end{proposition}

\begin{proof}
Denote $\Delta \P = -\beta + \gamma$.
Let $c$ be such that $b/a, b/a^2<c$ implies that $\|\Delta \P\|_{L^\infty(\W)} < 1/9$. 
Then, 
\[
\|\alpha - \beta + \gamma\|_{L^\infty(\W_\v)} 
\le \|\alpha\|_{L^\infty(\W_\v)} + \|\Delta \P\|_{L^\infty(\W)} < \tfrac{1}{9} + \tfrac{1}{9} < 1, 
\]
which implies that $\P = \IdRtwo + (\alpha - \beta + \gamma)$ is invertible, i.e., it is an implant map.
It remains to verify that each hole satisfies the requirements of a regular inner boundary, when $\W_\v$ is endowed with the implant map $\P$.

Consider the $i$-th dislocation: without loss of generality we may set $p_i=0$ and write $\v_i = \v$. For simplicity, assume that $\v = v \,\pl_1$ for $v>0$ (the general case can be obtained by rotation). By the definition of $\alpha$, comparing with \eqref{eq:hQv},
\[
(B_a \setminus B_{3|\v|/2}, \id + \alpha) = (\hM_\v^a \setminus \hM_\v^{3|\v|/2},\hP_\v).
\]
Henceforth, we will write $\hP_\v$ instead of $\id+\alpha$.

The annulus $B_a \setminus B_{3|\v|/2}$ can be endowed with three different metrics: the Euclidean metric $\euc$, the metric induced by $\hP_\v$ and the metric induced by $\P = \hP_\v + \Delta \P$. 
The notations below distinguish between the various metrics. 

The uniform estimates on $\alpha$ and on $-\beta + \gamma$ imply that that in $\W_\v$,
\beq
|\hP_\v - \id|_{\euc,\euc} \le \tfrac19
\qquad
|\Delta \P|_{\euc,\euc} \le \tfrac19
\Textand
|\P - \id|_{\euc,\euc} \le \tfrac29,
\label{eq:estimateQ0}
\eeq
hence
\beq
|\hP_\v|_{\euc,\euc} \le \tfrac{10}{9}
\Textand
|\P|_{\euc,\euc}  \le \tfrac{11}{9}.
\label{eq:estimateQ1}
\eeq
Moreover, using the Neumann series representations, of $\hP_\v^{-1}$ and $\P^{-1}$, e.g.,
$\P^{-1} = \sum_{k=0}^\infty (I - \P)^k$,
\beq
|\hP_\v^{-1}|_{\euc,\euc} \le \tfrac{9}{8}
\Textand
|\P^{-1}|_{\euc,\euc}  \le \tfrac{9}{7}.
\label{eq:estimateQ2}
\eeq

For $u\in T\W_\v$ and $\omega\in T^*\W_\v$,
\[
|u|_{\P} = |\P u|_\euc
\Textand
|\omega|_{\P} = |\omega \P^{-1}|_\euc. 
\]
Hence 
\[
|dr|_\P = |dr \P^{-1}|_\euc \le |dr|_\euc |\P^{-1}|_{\euc,\euc} \le 
\tfrac{9}{7}  < \tfrac32,
\]
and
\[
|\pl_r|_\P = |\P \pl_r|_\euc \le |\P|_{\euc,\euc} |\pl_r|_\euc \le \tfrac{11}{9} < \tfrac54.
\]
By the same argument as in Lemma~\ref{lem:frakr=r}, we obtain that for a point $p=(r,\vp)\in B_a \setminus B_{3|\v|/2}$, 
\beq\label{eq:frakr_Q}
\dist_\P(p,B_{3|\v|/2}) \in \brk{ \tfrac{2}{3} \brk{r - \tfrac{3|\v|}{2}}, \tfrac{5}{4} \brk{r - \tfrac{3|\v|}{2}}}.
\eeq

Consider the set
\[
A' =  \{p\in B_a \setminus B_{3|\v|/2} ~:~ \dist_\P(p,B_{3|\v|/2}) < |\v|\}.
\]
We need to show that some set $A\supset A'$, endowed with $\P$, can be embedded isometrically in $(B_{4|\v|} \setminus B_{|\v|}, \hP_\v) = \hM_\v^{4|\v|}$.
By \eqref{eq:frakr_Q}, for every $p\in A'$,
\[
|\v| > \dist_\P(p,B_{3|\v|/2}) > \tfrac{2}{3} \brk{r - \tfrac{3|\v|}{2}},
\]
from which follows that 
\[
A' \subset B_{3|\v|} \setminus B_{3|\v|/2} \equiv A.
\]
We now show that $(A,\P)$ can be isometrically embedded in $\hM_\v^{4|\v|}$.

The inclusion map (in coordinates) is not an isometry since the metrics in the domain and the target are different.
However, we will use the fact that they differ by $\Delta \P$, which is exact and sufficiently small, to construct such an embedding.
To this end, we use an isometric immersion similar to the one used in proof of the uniqueness theorem (Theorem~\ref{thm:disloc_unique}):
Let $p_0 \in \pl \hM_\v$ be the point on the boundary for which the Burgers vector $\hat \b_{p_0} = \hP_\v^{-1}|_{p_0}(\v)$ is perpendicular to $\pl \hM_\v$ and pointing inwards.
Let $\hat \gamma:[0,t_0)\to \hM_\v^{4|\v|}$ be the unit speed geodesic emanating from $p_0$ in the direction $\hat \b_{p_0}$, where $t_0$ is such that $\hat\gamma(t_0)$ hits the outer-boundary of $\hM_\v^{4|\v|}$.
It follows from Lemma~\ref{lem:frakr=r} that 
\[
\brk{1 - \tfrac{1}{2\pi}} 3|\v| + \tfrac{1}{2\pi}|\v| \le t_0 = \r(\hat{\gamma}(t_0)) - |\v| \le 3|\v|,
\]
implying that
$t_0 \in [2.5|\v|,3|\v|]$.
Define the map $\hat{f}: \hM_\v^{4|\v|} \setminus \hat \gamma \to \R^2$,
\[
\hat{f}(q) = \v + \int_{p_0}^q \hP_\v.
\]
If we extend $\hat{f}$ from $\hM_\v^{4|\v|}\setminus \hat \gamma$ to $\hat \gamma$ by moving clockwise, then $\hat{f}(\hat \gamma(t)) = \v + t\v/|\v|$, whereas, if we extend it by moving counter-clockwise, we obtain
$\hat{f}(\hat \gamma(t)) = 2\v + t\v/|\v|$ (because the circulation of $\hP_\v$ is $\v$).

We construct the analogous map for the set $A$ endowed with the implant map $\P$.
Let $p_1 = (3|\v|/2,\vp_1) \in A$ be the point on the inner boundary of $A$ for which the Burgers vector $\b_{p_1} = \P^{-1}|_{p_1}(\v)$ is perpendicular to the inner boundary (with respect to $\P$) and pointing inwards.
Let $\gamma:[0,t_1) \to A$ be the unit speed geodesic emanating from $p_1$ in the direction $\b_{p_1}$, where $t_1$ is such that $\gamma(t_1)$ hits the outer-boundary of $A$.
By \eqref{eq:frakr_Q}, 
\[
t_1 \in \brk{ \tfrac{2}{3} \brk{3|\v| - \tfrac{3|\v|}{2}}, \tfrac{5}{4} \brk{3|\v| - \tfrac{3|\v|}{2}}},
\]
i.e., $t_1 < 15|\v|/8$, and in particular $t_1 + 3|\v|/2 < t_0 + |\v|$.

Define the map $f: A \setminus \gamma \to \R^2$ by
\[
f(q) = \frac{3}{2}\v + \int_{p_1}^q \P.
\]
As for $\hat{f}$, extending $f$ from $A\setminus\gamma$ to $\gamma$ clockwise yields $f(\gamma(t)) = \frac{3}{2}\v +t\v/|\v|$, whereas counter-clockwise $f(\gamma(t)) = \frac{5}{2}\v +t\v/|\v|$ ($\P$ has the same circulation as $\hP$).

Since both $\hat{f}$ and $f$ are isometric embeddings, it suffices to show that the image of $f$ is contained in the image of $\hat{f}$. 
In that case, the map $\iota := \hat{f}^{-1} \circ f : A \setminus \gamma \to \hM_\v^{3|\v|} \setminus \hat \gamma$ is an isometric embedding, that can be extended to $\gamma$ smoothly by considering the extensions to $\gamma$ and $\hat{\gamma}$ as discussed above (since  $\hP_\v$ and $\P$ differ by an exact one-form, there is no problem with the gluing of these two extensions).

Note first that $\iota(\gamma) \subset \hat{\gamma}$ (choose, say, the counter-clockwise extensions),
as
\[
f(\gamma) = \{\tfrac{3}{2}\v +t\v/|\v| ~:~ t\in [0,t_1)\} \subset  \{\v +t\v/|\v| ~:~ t\in [0,t_0)\} = \hat f(\hat \gamma),
\]
where we used the bound $t_1 + 3|\v|/2 < t_0 + |\v|$.

Next, consider the boundaries of $\hM^{4|\v|}_\v$:
The inner boundary of $\hM^{4|\v|}_\v$ is parametrized, in polar coordinates, by $\{(|\v|,\vp) ~:~ \vp\in [0,2\pi)\}$, and is mapped via $\hat f$ to the set 
\[
\sigma_{\text{in}} = \BRK{|\v|(\cos\vp,\sin\vp) + \frac{\vp}{2\pi} \v ~:~ \vp\in [0,2\pi)}
\]
(in Euclidean coordinates on $\R^2$).
Similarly, the outer-boundary of $\hM^{4|\v|}_\v$ is mapped to
\[
\sigma_{\text{out}} = \BRK{(t_0 + |\v|)(\cos\vp,\sin\vp) + \frac{\vp}{2\pi} \v ~:~ \vp\in [0,2\pi)}.
\]
Our aim is therefore to show that $f(A\setminus \gamma)$ lies between these two curves.
The inner boundary of $A$ is parametrized, in polar coordinates, by $\{(3|\v|/2, \vp_1 + \vp) ~:~ \vp \in [0,2\pi)\}$.
For a point $q = (3|\v|/2, \vp_1 + \vp)$,
\[
f(q) = \frac{3}{2} \v + \int_{p_1}^q \hP_\v + \int_{p_1}^q \Delta \P = \frac{3|\v|}{2} (\cos\vp, \sin\vp) + \frac{\vp}{2\pi} \v + \int_{p_1}^q \Delta \P,
\]
hence
\[
\Abs{f(q) - \frac{3|\v|}{2} (\cos\vp, \sin\vp)  - \frac{\vp}{2\pi} \v } \le 3\pi |\v| \|\Delta \P\|_{L^\infty(\W_\v)}.
\]
The inner boundary lies between $\sigma_{\text{in}}$ and $\sigma_{\text{out}}$ if $\|\Delta \P\|_{L^\infty(\W_\v)}$ is small enough (independently of $\v$).
By the same argument, for $q = (3|\v|, \vp_1 + \vp)$ on the outer-boundary of $A$,
\[
\Abs{f(q) - (t_1 + \tfrac{3}{2}|\v|)(\cos\vp, \sin \vp) - \frac{\vp}{2\pi} \v } \le 6\pi |\v| \|\Delta \P\|_{L^\infty(\W_\v)}.
\]
Since $t_1 + 3|\v|/2 < t_0 + |\v|$, the outer-boundary of $f(A \setminus \gamma)$ is between $\sigma_{\text{in}}$ and $\sigma_{\text{out}}$ for small enough $\|\Delta \P\|_{L^\infty(\W_\v)}$.
This completes the proof that $\iota = \hat f^{-1} \circ f : A \to \hM_\v^{4|\v|}$ is an isometric immersion.

To complete the proof that $(\W_\v,\P)$ has a regular inner boundary, we need to show that $A$ 
is Lipschitz equivalent to $B_{2|\v|}\setminus B_{|\v|}$ with bilipschitz constant $10$, where $A$ is endowed with the metric induced by $\P$.
This follows by the same arguments as the proof that $\hM_\v^R$ has a regular inner boundary, using the fact that the metric induced by $\P$ is equivalent to the Euclidean metric on $A$ with a factor of $9/7$ (which follows from Proposition~\ref{prop:many_disloc_bound} below).

Finally, the condition that $10b<a$ implies that the Euclidean distance between the inner boundaries is at least $a - 3b > 7b$. Since distances with respect to the metrics $\euc$ and $\P^\#\euc$ are equivalent with constant $9/7$ (again, by Proposition~\ref{prop:many_disloc_bound} below), it follows that $\rho$ as defined in  \defref{def:body_w_disloc}(d) satisfies the requirements.
\end{proof}

The following proposition estimates the deviation of $(\W_\v, \P)$ from the Euclidean domain $(\W_\v,\IdRtwo)$.

\begin{proposition}
\label{prop:many_disloc_bound}
Assume that $b$ and $a$ satisfy the assumptions $b/a, b/a^2<c$ of \propref{prop:construct_many_disloc}.
Then,
\[
|\IdRtwo|_{\P^\#\euc,\euc},\,\, |\IdRtwo|_{\euc,\P^\#\euc} \le \tfrac97.
\]
Furthermore, for a point $p \in \W_\v \cap B_a(p_i)$,
\[
|\IdRtwo - \P|_{\P^\#\euc,\euc} (p)\lesssim \frac{|\v^i|}{r_i(p)} + \frac{b}{a^2} \simeq \frac{|\v^i|}{\r(p)} + \frac{b}{a^2},
\]
and for a point $p\in \W_\v \setminus \bigcup_i B_a(p_i)$,
\[
|\IdRtwo - \P|_{\P^\#\euc,\euc} (p)\lesssim  \frac{b}{a^2}.
\]
Finally,
\beq
\int_{\W_\v} |\IdRtwo - \P|_{\P^\#\euc,\euc}^2 \,\dVol_{\P^\#\euc} \lesssim \sum_{i=1}^m |\v^i|^2 \log\brk{\frac{a}{|\v^i|}} + 
\|\tmu\|_{\HminusOne(\W)}^2.
\label{eq:5.4estimate}
\eeq
\end{proposition}

\begin{proof}
By the definition of the pullback metric and the operator norm, 
\[
\begin{aligned}
 |\IdRtwo|_{\P^\#\euc,\euc} &= |\P^{-1}|_{\euc,\euc} < \tfrac97 \\
  |\IdRtwo|_{\euc,\P^\#\euc}  &= |\P|_{\euc,\euc} < \tfrac{11}{9},
\end{aligned}
\]
where we used \eqref{eq:estimateQ1} and \eqref{eq:estimateQ2}.
We have thus proved the uniform bilipschitz bounds on $\IdRtwo$ (with respect to the metrics $\P^\#\euc$ and $\euc$). 
Moreover, by the equivalence of the metrics $\euc$ and $\P^\#\euc$ (with a constant depending only on $\W$, as long as $a$ and $b$ satisfy the constraints), 
\[
|\IdRtwo- \P|_{\P^\#\euc,\euc} \lesssim |\IdRtwo- \P|_{\euc,\euc}  \le |\alpha - \beta|_{\euc,\euc} + |\gamma|_{\euc,\euc} \lesssim \sum_{i=1}^m \frac{|\v^i|}{r_i} \chi_{B_a(p_i)} + \frac{b}{a^2}
\]
from which the pointwise bounds follow, using the fact that $r_i \simeq \r$ in $B_a(p_i)$, which follows from the same analysis as in \eqref{eq:frakr_Q}.

As for the integral bound, by the equivalence of norms and volume forms, 
\[
\int_{\W_\v} |\IdRtwo - \P|_{\P^\#\euc,\euc}^2  \,\dVol_{\P^\#\euc} \lesssim \|\alpha - \beta\|_{L^2(\W_\v)}^2 + \|\gamma\|_{L^2(\W)}^2.
\]
It is immediate that 
\[
\|\alpha - \beta\|_{L^2(\W_\v)}^2 \lesssim \sum_{i=1}^m |\v^i|^2 \log\brk{\frac{a}{|\v^i|}}.
\]
As for the bound on $\gamma$, we use again elliptic regularity \cite{Sch95}:
\[
\|\gamma\|_{L^2(\W_\v)} \lesssim \|d\gamma \|_{\HminusOne(\W)} = \|\tmu\|_{\HminusOne(\W)}.
\]
This completes the proof.
\end{proof}

The pointwise bound on $|\IdRtwo - \P|_{\P^\#\euc,\euc}$ has two contributions: a ``near field" which is affected by the nearest dislocation, and a ``far field"  which accounts for all the dislocations. 
In the forthcoming analysis, these two contributions will be identified with a self-energy and an interaction energy, respectively. 
Either term may be dominant, depending on the relation between the number of dislocations and their magnitude.

\begin{comment}
The implant map $\P$ constructed above is not smooth, however it is continuous and $d\P=0$ distributively. Using a mollification, we may obtain a smooth approximation preserving the circulation and satisfying all the bounds. 
\end{comment}

\subsection{Convergence of bodies with dislocations}\label{sec:conv_disloc}

We next define a notion of convergence of bodies with dislocations as the magnitude of each dislocation tends to zero, while a (possibly rescaled) total Burgers vector tends to a limit (cf. \cite{KM15,KM16,KM16a,EKM20}). 
We start with a definition in which the total Burgers vector is not rescaled. This is a refinement of the definitions used in \cite{KM15,KM16,KM16a,EKM20}, where several examples can be found.
We later focus on the case in which the total Burgers vector tends to zero, hence has to be
rescaled to obtain a non-trivial limit.

\begin{definition}\label{def:man_conv1}
Let $(\MEps,\PEps)$ be a sequence of bodies having $m_\e$ dislocations with dislocation structure $\bbS$ at length-scale $\e$. 
Denote by $\gEps = \PEps^\#\euc$ the Riemannian metric induced by $\PEps$.
We denote the corresponding Burgers vectors by $\e\v_\e^i$, $i=1,\dots,m_\e$, $\v_\e^i\in\bbS$.
Let $(\M,\P)$ be a simply-connected complete elastic body (with $\P$ not necessarily closed), and denote by $\g = \P^\#\euc$ its Riemannian metric.  
We say that $(\MEps,\PEps)$ converges to $(\M,\P)$ if $\MEps$ can be embedded in $\M$, such that the inclusion map $(\MEps,\PEps)\to (\MEps, \P)$ is uniformly-bilipschitz and the following is satisfied: 

\begin{enumerate}[itemsep=0pt,label=(\alph*)]
\item \textbf{Asymptotic surjectivity:} The outer-boundary of $\MEps$ coincides with the boundary of $\M$, and
\[
\Vol_\g(\M \setminus \MEps) \to 0.
\]

\item \textbf{The embeddings are asymptotically rigid pointwise:} 
\beq
|\IdRtwo - \PEps|_{\gEps,\g} \lesssim \frac{\e  |\v_\e^i|}{\r_\e} + c_\e,
\label{eq:local_dist_bnd}
\eeq
almost everywhere, where $|\v_\e^i|$, is the magnitude of the nearest dislocation, $\r_\e$ is the shortest distance to a dislocation in $\MEps$, and $c_\e>0$ is an infinitesimal sequence, $c_\e\to0$ as $\e\to0$.

\item
\textbf{Implant map convergence:}
\[
\|\PEps - \P\|_{L^2(\MEps)} \to 0.
\]
\end{enumerate}
\end{definition}

The requirement that the outer-boundary of $\MEps$ coincides with $\pl\M$ can be relaxed to a weaker condition: that the domains enclosed by the outer-boundaries of $\MEps$ are uniformly Lipschitz equivalent to $\M$. 
 
It can be shown that the limit is unique: if $(\MEps,\PEps)$ also converges to $(M',\P')$, then $(M,\P)$ and  $(M',\P')$ are isometric. This result is a consequence of a generalization of Reshetnyak's rigidity theorem to Riemannian manifolds; see \cite[Thm. 5.3]{KMS19} for a similar statement.

The convergence of the implant maps $\PEps$ implies a convergence of the corresponding distributions,
\[
\torsion_\e : \psi_\e \mapsto \int_{\dM_\e} \TR(\PEps \otimes \psi_\e),
\]
which we can consider as bounded linear functionals on $C_c(\M;\R^2)$ via the restriction $\psi \mapsto \torsion_\e(\psi|_{\MEps} )$, in the following sense:
For every $\psi\in \HOneZero(\M;\R^2)\cap C_c(\M;\R^2)$,
\[
\LimEps \torsion_\e(\psi|_{\MEps}) = \int_\M \TR(d\P \otimes \psi) \equiv \torsion(\psi).
\]
Indeed, it follows from \eqref{eq:Tpsi_H_1_0} that
\[
\begin{split}
\torsion_\e(\psi|_{\MEps})  
&=  \int_{\MEps}  \TR(\PEps\wedge d\psi).
\end{split}
\]
Letting $\e\to0$, using the $L^2$-convergence of $\PEps$, the smoothness of $\P$, the asymptotic surjectivity, and the fact that $\psi\in H_0^1(\M;\R^2)$, 
\[
\LimEps \torsion_\e(\psi) = \int_{\M}  \TR(\P\otimes d\psi) = \int_{\M}  \TR(d\P\otimes \psi).
\]
Note that while the functionals $\torsion_\e$
can be identified with $\R^2$-valued Radon measures on $\M$,
the convergence of $\torsion_\e$ to $\torsion$ on $\HOneZero(\M)\cap C_c(\M)$ cannot be extended to a convergence of measures, without additional assumptions, as $\torsion_\e$ are not necessarily uniformly bounded measures.

If $\P$ is closed, then $\torsion=0$, which we may interpret as $(\M,\P)$ being dislocation-free.
The case where $\torsion\ne0$ was treated in \cite{KM15,KM16} in the context of the emergence of torsion as a limit of dislocation density (see also \cite{EKM20}). 

In this work, we consider converging bodies in a regime where the limiting implant map $\P$ is closed. 
Since $\M$ is simply-connected, this implies that $\P$ is exact, that is $\P=d\phi$ for some $\phi :\M\to \R^2$.
The condition $d\phi=\P$ implies that $\phi$ is an isometric immersion of $\M$ into $\R^2$.
We will further assume that $\phi$ is an embedding,  and thus 
we can assume that $(\M,\P) = (\W,\IdRtwo)$, where $\W$ is a simply-connected domain in $\R^2$.
In this regime, where the limiting body is defect-free, a more refined definition is required to capture the convergence of the density of dislocations:

\begin{definition}
\label{def:man_conv2}
Let $(\MEps,\PEps)$ converge to $(\W,\IdRtwo)$  
according to \defref{def:man_conv1}.
Let  $\mu$ be an $\R^2$-valued Radon measure on $\W$ having finite total mass. We say that $(\MEps,\PEps)$ converge to $(\W,\IdRtwo,\mu)$
with respect to a sequence $n_\e$ satisfying $n_\e\e\to0$, if in addition:
\begin{enumerate}[itemsep=0pt,label=(\alph*)]
\item \textbf{Global distortion bound:}
\beq
\int_{\MEps} |\IdRtwo - \PEps|^2 \,\VolumeEps \lesssim h_\e^2,
\label{eq:global_dist_bnd}
\eeq
where
\[
h_\e^2 = \max\{n_\e^2\e^2,n_\e \e^2\LogEps\}.
\]
\item \textbf{Burgers vector convergence:} 
the measures $\frac{1}{n_\e \e}  \torsion_\e$ weakly-star converge to $\mu$ in $\calM(\W;\R^2)$,
in the sense that  for every $\psi\in C_c(\W;\R^2)$,
\beq
\label{eq:burgers_conv}
\frac{1}{n_\e\e} \torsion_\e(\psi|_{\MEps})  
\to \int_{\W}  \TR(\psi\otimes d\mu).
\eeq
\end{enumerate}
We denote this mode of convergence by
\[
(\MEps,\PEps) 
\xrightarrow[]{n_\e}
(\W,\IdRtwo,\mu).
\]
\end{definition}

Roughly speaking, $\e$ represents the typical magnitude of a dislocation and $n_\e$, which controls the energy and Burgers vector scalings, is related to the number of dislocations $m_\e$.  
Following \cite{GLP10}, we identify three regimes of parameters: the case $n_\e \ll \LogEps$ is called the \Emph{subcritical regime}; the case $n_\e = \LogEps$ is called the \Emph{critical regime}; the case $n_\e \gg \LogEps$ is called the \Emph{supercritical regime}. In the subcritical regime, the distortion bound is of order $n_\e \e^2\LogEps$, and is induced by the ``near field" contributions, whereas in the supercritical regime,  the distortion bound is of order $n_\e \e^2$, and is induced by the ``far field" contribution.

For a sequence $(\MEps,\PEps)\xrightarrow[]{n_\e}(\W,\IdRtwo,\mu)$, we consider sequences of functions and 1-forms defined on $\MEps$ to converge to functions and one-forms on $\W$, when their extensions by zero converge (with respect to the Euclidean metric $\euc$ on $\W$).

\begin{remark}
We will later show, in Theorem~\ref{thm:measure_comp}, that in the critical and subcritical regimes, a sequence satisfying \eqref{eq:global_dist_bnd} has a subsequence satisfying \eqref{eq:burgers_conv} for some $\mu \in \calM(\W;\R^2)$, assuming that the dislocations are well-separated. 
\end{remark}

\begin{example}
\label{ex:conv1}
The sequence of manifolds with a single dislocation $(\hM_{\e\v}^{R},\hP_{\e\v})$ converges to $(B_R,\IdRtwo,\v\,\delta_0)$ with respect to the parameters $n_\e =1$. This follows from Proposition~\ref{prop:conv_manifolds}.
\end{example}

\begin{example}
\label{ex:conv2}
Let $\W = (0,1)^2$ and let $\mu\in \calM(\W;\R^2)$ be absolutely-continuous with respect to the Lebesgue measure, with $d\mu/dx \in L^\infty(\W;\R^2)$.
We construct bodies with dislocations such that $(\MEps,\PEps) \to (\W, \IdRtwo, \mu)$. 
In this example, no dislocation structure is assumed.

Choose any sequence $n_\e\to\infty$ satisfying $n_\e\e\to0$.
Partition $\W$ into $n_\e^{1/2}\times n_\e^{1/2}$ squares, $D_\e^1,\ldots, D_\e^{n_\e}$.
Denote by $p_\e^i$ the center of the $i$-th square, and let $\v_\e^i = n_\e\mu(D_\e^i)$. Set $\mu_\e = \sum_{i=1}^{n_\e} \e \v_\e^i \otimes \delta_{p_\e^i}$,
then 
\[
\frac{1}{n_\e \e}\mu_\e = \sum_{i=1}^{n_\e} \mu(D_\e^i)\delta_{p_\e^i} \weakstar \mu.
\]

Construct $(\MEps,\PEps)$ as in \secref{sec:disloc_construction}, according to the measure $\mu_\e$.
That is,
\[
\MEps = \W \setminus \brk{\bigcup_{i=1}^{n_\e} B_{3\e |\v_\e^i|/2}(p_\e^i)},
\]
and $\P$ is defined by altering $\IdRtwo$ with the $\R^2$-valued 1-forms $\alpha,\beta,\gamma$.
In the notation of \secref{sec:disloc_construction}, 
\[
b  = \max_i \e|\v_\e^i| = n_\e \e \, \max_i |\mu(D_\e^i)| \le \|d\mu/dx\|_{\infty} \e,
\]
$m=n_\e$ and $a \simeq n_\e^{-1/2}$.
By \propref{prop:construct_many_disloc}, $(\MEps,\PEps)$ is a body with $n_\e$ dislocations if $b/a \simeq n_\e^{1/2}\e$ and $b/a^2 \simeq n_\e \e$ are small enough, which is eventually the case, as $n_\e\e\to0$.

Moreover, it follows from \propref{prop:many_disloc_bound} that
\[
|\IdRtwo - \PEps| \lesssim \frac{\e|\v_\e^i|}{\r_\e} + n_\e \e,
\]
that the embeddings $(\MEps,\PEps)\to (\MEps,\IdRtwo)$ are uniformly bilipschitz, and
\[
\int_{\MEps} |\IdRtwo - \PEps|^2 \,\VolumeEps \lesssim n_\e \e^2 \log\brk{\frac{n_\e^{-1/2}}{\e}} + n_\e^2 \e^2 \lesssim h_\e^2.
\]
(Note here the distinct contributions on the near- and far-fields.)
Finally, since $\frac{1}{n_\e \e}\mu_\e \weakstar \mu$, it follows that $\frac{1}{n_\e \e}\torsion_\e \to \mu$ in the sense of \defref{def:man_conv2}(b), namely
\beq
\LimEps \frac{1}{n_\e \e}\torsion_\e(\psi|_{\MEps}) = \int_\W \TR(\psi\otimes d\mu)
\label{eq:burgers_conv_2}
\eeq
for every $\psi\in C_c(\W;\R^2)$.
Indeed, given $\psi\in C_c(\W;\R^2)$ set
\[
\psi_\e = \sum_{i=1}^{n_\e} \psi(p_\e^i) \chi_{D_\e^i}.
\]
On the one hand,
\[
\begin{split}
\frac{1}{n_\e\e} \torsion_\e(\psi_\e|_{\MEps}) &= 
 \frac{1}{n_\e\e} \int_{\dM_\e} \TR(\PEps\otimes (\psi_\e)) \\
&= \frac{1}{n_\e\e} \sum_{i=1}^{n_\e} \TR\brk{\psi(p_\e^i) \otimes \int_{\dM_\e} \PEps (\chi_{D_\e^i})}\\
&= \frac{1}{n_\e\e} \sum_{i=1}^{n_\e} \TR(\e\v_\e^i \otimes \psi(p_\e^i)) \\
&= \frac{1}{n_\e\e} \int_\W \TR(\psi\otimes d\mu_\e),
\end{split}
\] 
which converges to the right-hand side of \eqref{eq:burgers_conv_2},
On the other hand,
\[
\begin{split}
\Abs{\frac{1}{n_\e \e}\torsion_\e(\psi|_{\MEps}) - \frac{1}{n_\e \e}\torsion_\e(\psi_\e|_{\MEps})} 
&= \frac{1}{n_\e\e}\Abs{ \int_{\dM_\e} \TR(\PEps\otimes ((\psi - \psi_\e)|_{\MEps}))} \\
&\lesssim \frac{1}{n_\e \e} \sum_{i=1}^{n_\e} \e |\v_\e^i| \|\psi - \psi_\e\|_{L^\infty(\W)} \\
&= \brk{\sum_{i=1}^{n_\e} |\mu(D_\e^i)|} \|\psi - \psi_\e\|_{L^\infty(\W)} \\
&= |\mu|(\W)\,  \|\psi - \psi_\e\|_{L^\infty(\W)},
\end{split}
\] 
where the inequality uses the fact that the length of the boundary around the $i$-th dislocation is of order $\e|\v_\e^i|$.
Since by the continuity of $\psi$,  the right-hand side tends to zero as $\e\to0$, we obtain \eqref{eq:burgers_conv_2}.

Therefore, $(\MEps,\PEps) \to (\W, \IdRtwo, \mu)$, according to \defref{def:man_conv2}. 
A variant of this construction will be at the heart of the recovery sequence in the $\Gamma$-convergence result below.
\end{example}

The number $m_\e$ of dislocations and the magnitude $|\v_\e^i|$ of individual dislocations is not assumed a priori in the definition of converging sequences of bodies with dislocations. The following lemma asserts that the convergence implies bounds on both:

\begin{lemma}
\label{lem:max_disloc}
Let 
\[
(\MEps,\PEps) 
\xrightarrow[]{n_\e}
(\W,\IdRtwo,\mu).
\]
Then,
\begin{enumerate}[itemsep=0pt,label=(\alph*)]

\item 
\Emph{Burgers vector bound}:
\[
\sum_{i=1}^{m_\e} |\v_\e^i| \lesssim n_\e.
\]
In particular,
\[
\max_i |\v_\e^i| \lesssim n_\e.
\]

\item \Emph{Holes volume bounds:}
\beq
|\W\setminus\MEps| \lesssim n_\e^2 \e^2.
\label{eq:bnd_holes_area}
\eeq

\item \Emph{Number of dislocations bound}:
\beq
m_\e \lesssim n_\e.
\label{eq:me_bound}
\eeq
\end{enumerate}
The constants in all inequalities may only depend on $\Omega$ and $\mu$.
\end{lemma}

\begin{proof}
The Burgers vector convergence \eqref{eq:burgers_conv} implies that $\frac{1}{n_\e \e} \torsion_\e$ is uniformly bounded in $\calM(\M;\R^2)$. By \eqref{eq:norm_calT},
\[
\frac{1}{n_\e \e} \sum_{i=1}^{m_\e} \e |\v_\e^i | \simeq \frac{1}{n_\e \e} \|\torsion_\e\|_{\calM(\MEps)} = \frac{1}{n_\e \e} \|\torsion_\e\|_{\calM(\M)} \lesssim 1,
\]
which completes the proof of the first assertion.

For the second assertion note that the length of the inner boundary of $\MEps$ that corresponds to the $i$-th dislocation, when measured with respect to $\P$, is of the same order as when measured with respect to $\PEps$ , which is $\e|\v_\e^i|$.
Thus  $|\W\setminus \MEps| \simeq \sum_{i=1}^{m_\e} (\e |\v_\e^i|)^2$ by the isoperimetric inequality, and the right-hand side is bounded by $\e^2 n_\e^2$ by the first assertion.

The last assertion follows from the first, since for every $i$, $\v_\e^i\in \bbS$, and therefore $|\v_\e^i| \gtrsim  1$. 
Thus $m_\e  \lesssim \sum_{i=1}^{m_\e} |\v_\e^i| \lesssim n_\e$. 
\end{proof}

The next set of lemmas and propositions concern refined estimates in conjunction with the convergence of the Burgers vector:

\begin{lemma}
\label{lem:5.8}
Let 
\[
(\MEps,\PEps) 
\xrightarrow[]{n_\e}
(\W,\IdRtwo,\mu).
\]
Then, for every $\psi\in \HOneZero(\W;\R^2)$,
\[
\LimEps \frac{1}{n_\e \e} \int_{\MEps} \TR(\IdRtwo\wedge d\psi) = 0.
\]
which in coordinates reads,
\[
\LimEps \frac{2}{n_\e \e} \int_{\MEps} \brk{\pl_1 \psi^2 - \pl_2 \psi^1}\, dx = 0.
\]
\end{lemma}

\begin{proof}
Pulling back,
\[
\begin{split}
\Abs{\frac{1}{n_\e\e}\int_{\MEps} \TR(\IdRtwo\wedge d\psi)} 
&= \Abs{\frac{1}{n_\e\e}\int_{\W\setminus\MEps} \TR(\IdRtwo\wedge d\psi)} \\
&\lesssim \frac{1}{n_\e\e} |\W\setminus\MEps|\, \|\psi\|_{\HOne(\W)} \\
&\lesssim  n_\e\e \|\psi\|_{\HOne(\W)},
\end{split}
\]
where in the first line we used the fact that the integral of $\TR(\IdRtwo\wedge d\psi) = -d\TR(\IdRtwo\wedge \psi)$ over $\W$  vanishes, and the last inequality follows from \eqref{eq:bnd_holes_area}.
The right-hand side tends to zero since $n_\e\e\to 0$. 
\end{proof}

Combining \lemref{lem:5.8} and Eq.~\eqref{eq:Tpsi_H_1_0} we obtain:

\begin{corollary}
Condition \eqref{eq:burgers_conv} for the convergence of the Burgers vector implies that
\beq
\label{eq:modified_burgers_conv}
\LimEps \frac{h_\e}{n_\e\e} \int_{\MEps} \TR\brk{\tfrac{\PEps - \IdRtwo}{h_\e} \wedge d\psi}
= \int_\W \TR(\psi\otimes d\mu)
\eeq
for every $\psi\in \HOneZero(\W;\R^2)\cap C_c(\W;\R^2)$.
\end{corollary}

\begin{proof}
\lemref{lem:5.8} implies that for every $\psi\in \HOneZero(\W;\R^2)$,
\[
\begin{split}
\LimEps \frac{1}{n_\e\e} \torsion_\e(\psi|_{\MEps}) &= 
\LimEps \frac{1}{n_\e\e} \int_{\MEps} \TR\brk{\PEps \wedge d\psi} \\
&=\LimEps \frac{h_\e}{n_\e\e} \int_{\MEps} \TR\brk{\tfrac{\PEps - \IdRtwo}{h_\e} \wedge d\psi},
\end{split}
\]
where the first equality follows from $\eqref{eq:Tpsi_H_1_0}$. The result then follows from \eqref{eq:burgers_conv}, using the fact that $\psi\in C_c(\W;\R^2)$.
\end{proof}

The following proposition shows that if $(\MEps,\PEps)$ converges to $(\W,\IdRtwo,\mu)$ with respect to critical or supercritical parameters $n_\e$, then the measure $\mu$ has $\HminusOne$ regularity:

\begin{proposition}
\label{prop:5.5}
Let 
\[
(\MEps,\PEps) 
\xrightarrow[]{n_\e}
(\W,\IdRtwo,\mu).
\]
If $n_\e$ is critical or supercritical, i.e., $n_\e\gtrsim\LogEps$, then $\mu\in \HminusOne(\W;\R^2)$.
\end{proposition}

\begin{proof}
In the critical and supercritical regimes, $h_\e = n_\e \e$. By \eqref{eq:modified_burgers_conv}, for every $\psi\in \HOneZero(\W;\R^2)\cap C_c(\W;\R^2)$,
\beq
\int_\W \TR(\psi\otimes d\mu) = 
\LimEps \int_{\MEps} \TR\brk{\tfrac{\PEps - \IdRtwo}{h_\e} \wedge d\psi}. 
\label{eq:int_psi_dmu}
\eeq
By the Cauchy-Schwarz inequality and by the global distortion bound \eqref{eq:global_dist_bnd},
\[
\Abs{\int_\W \TR(\psi\otimes d\mu)} 
\lesssim \LimsupEps \brk{\int_{\MEps} \Abs{\tfrac{\PEps - \IdRtwo}{h_\e}}^2 \,\VolumeEps}^{1/2} \, \|\psi\|_{\HOne(\W;\R^2)} 
\lesssim \|\psi\|_{\HOne(\W;R^2)},
\]
i.e., $\mu\in \HminusOne(\W;\R^2)$.
\end{proof}

Finally, we relate the convergence \eqref{eq:burgers_conv} of Burgers vectors to a convergence of measures. We start with the following technical lemmas:

\begin{lemma}
\label{lem:bilip_holes}
Let 
\[
(\MEps,\PEps) 
\xrightarrow[]{n_\e}
(\W,\IdRtwo,\mu).
\]
For every $\e$, denote by $D_\e^i$, $i=1,\dots,m_\e$, the subdomains of $\W$ that are encircled by the cores of the $m_\e$ dislocations in $\MEps$. 
Denote
\[
A_\e^i \supset \{p\in\MEps ~:~ \r_i(p) < 2\e |\v_\e^i|\}
\]
be the regular annular domains as in Definition~\ref{def:regular_dM} (with respect to the distance defined by $\PEps$).
Then, $A_\e^i$ is Lipschitz equivalent (as a domain in $\R^2$) to $B_{2\e|\v_\e^i|}\setminus B_{\e|\v_\e^i|}$, with constants independent of $\e$ and $\v_\e^i$.
In particular, there exists a constant $C$ independent of $\e$ and $\v_\e^i$, such that
\[
\overline{D_\e^i}\subset B_{C\e|\v_\e^i|}(p_\e^i), \qquad i=1,\dots,m_\e,
\]
for some $p_\e^i \in \W$.
\end{lemma}

\begin{proof}
By our assumption on the regularity of the inner boundaries and the inclusion map $(\MEps,\PEps) \to (\MEps, \P)$, 
we have that $B_{2\e|\v_\e^i|}\setminus B_{\e|\v_\e^i|}$ is Lipschitz equivalent to $(A_\e^i,\PEps)$ with an equivalence constant $10$, and $(A_\e^i, \P)$ is Lipschitz equivalent to $(A_\e^i,\PEps)$ with constant independent of $\e$ and $i$.
This proves the first part.
In particular $(A_\e^i, \P)$ has a diameter of order $\e|\v_\e^i|$ and can thus can be contained in a ball of radius $C\e|\v_\e^i|$ for some $C>0$.
Since $D_\e^i$ is the topological disc enclosed by the inner boundary of the annulus $A_\e^i$, the second part follows.
\end{proof}

\begin{lemma}
\label{lem:chi_e}
Let 
\[
(\MEps,\PEps)
\xrightarrow[]{n_\e}
(\W,\IdRtwo,\mu).
\]
and let $D_\e^i$ be as in Lemma~\ref{lem:bilip_holes}.
For every $\psi\in C^1_c(\W;\R^2)$, there exists  a sequence $\psi_\e\in C^\infty_c(\W;\R^2)$ satisfying
\begin{enumerate}[itemsep=0pt,label=(\alph*)]
\item $\psi_\e \to \psi$ uniformly.
\item $d\psi_\e \to d\psi$ in the following sense,
\beq
\|d\psi - d\psi_\e\|_{L^2(\MEps)} \lesssim n_\e\e,
\label{eq:bound_chi_chie}
\eeq
where the constant in the inequality may depend on $\psi$.
\item  $\psi_\e$ is constant on the domains $D_\e^i$,
\[
\psi_\e|_{D_\e^i} \equiv c_\e^i = \frac{1}{|D_\e^i|} \int_{D_\e^i} \psi\, \VolumeE.
\]
\end{enumerate}
\end{lemma}

\begin{proof}
By the previous lemma, there exists a constant $C$ independent of $\e$, such that
\[
\overline{D_\e^i}\subset B_{C\e|\v_\e^i|}(p_\e^i), \qquad i=1,\dots,m_\e,
\]
for some $p_\e^i \in \W$.
Since $\psi \in C^1(\W;\R^2)$, it follows that for every $i$,
\[
|\psi - c_\e^i| \lesssim \e 
\qquad\text{in $B_{2C\e|\v_\e^i|}(p_\e^i)$}, 
\]
where the constant in the inequality only depends on $\|d\psi\|_\infty$.
Hence, there exists a $\tilde{\psi}_\e\in W^{1,\infty}_0(\W;\R^2)$ such that 
\[
\tilde{\psi}_\e|_{B_{C\e|\v_\e^i|}(p_\e^i)} \equiv c_\e^i 
\Textand \tilde{\psi}_\e|_{\W\setminus \bigcup_i B_{2C\e|\v_\e^i|}(p_\e^i)} = \psi, 
\]
such that $\|d\tilde{\psi}_\e\|_\infty \le C'$, for some $C'$ depending only on $\psi$ ($d\tilde{\psi}_\e$ can be constructed, for example, by extending $\tilde{\psi}_\e$ radially on each annulus $B_{2C\e|\v_\e^i|}(p_\e^i)\setminus B_{C\e|\v_\e^i|}(p_\e^i)$).
It is immediate that $\tilde{\psi}_\e \to \psi$ uniformly.

Now,
\[
\|d\psi - d\psi_\e\|_{L^2(\MEps)}^2  \lesssim \|d\psi\|_\infty \sum_{i=1}^{m_\e} |B_{2C\e|\v_\e^i|}| \lesssim \e^2 \sum_{i=1}^{m_\e} |\v_\e^i|^2 \lesssim n_\e^2 \e^2,
\]
where the last inequality follows from Lemma~\ref{lem:max_disloc}(a).
By mollification, we obtain a smooth $\psi_\e$ satisfying all the requirements.
\end{proof}

\begin{lemma}
\label{lem:5.9}
Let 
\[
(\MEps,\PEps) 
\xrightarrow[]{n_\e}
(\W,\IdRtwo,\mu).
\]
For every $\psi\in C^1_c(\W;\R^2)$, there exists a sequence $\psi_\e\in C^\infty_c(\W;\R^2)$ as in \lemref{lem:chi_e}, such that
\beq
\LimEps \frac{1}{n_\e\e} \torsion_\e(\psi_\e|_{\MEps})
= \int_\W \TR(\psi\otimes d\mu).
\label{eq:approx_torsion_chie}
\eeq
\end{lemma}

\begin{proof}
Construct $\psi_\e$ as in \lemref{lem:chi_e}.
By the convergence \eqref{eq:burgers_conv} of Burgers vectors,
we need to prove that
\[
\LimEps \frac{1}{n_\e\e} \torsion_\e((\psi_\e - \psi)|_{\MEps}) = 0,
\]
i.e., that
\[
\LimEps \frac{1}{n_\e\e} \int_{\MEps} \TR\brk{\PEps \wedge d(\psi-\psi_\e)} = 0.
\]
By the same argument as in \lemref{lem:5.8},
it suffices to prove that
\[
\LimEps \frac{h_\e}{n_\e\e} \int_{\MEps} \TR\brk{\tfrac{\PEps - \IdRtwo}{h_\e} \wedge d(\psi-\psi_\e)} = 0,
\]
which is immediate as
\[
\begin{split}
& \Abs{\frac{h_\e}{n_\e\e} \int_{\MEps} 
\TR\brk{\tfrac{\PEps - \IdRtwo}{h_\e} \wedge d(\psi-\psi_\e)}}
\lesssim 
\frac{h_\e}{n_\e\e}  
\Norm{\frac{\PEps - \IdRtwo}{h_\e}}_{L^2(\MEps)} \|d\psi_\e - d\psi\|_{L^2(\MEps)} \lesssim h_\e,
\end{split}
\]
and the last passage follows from \eqref{eq:global_dist_bnd} and \eqref{eq:bound_chi_chie}.
\end{proof}

\begin{proposition}
\label{prop:mu_e_mu}
Let 
\[
(\MEps,\PEps) 
\xrightarrow[]{n_\e}
(\W,\IdRtwo,\mu).
\]
and let $D_\e^i$, $i=1,\dots,m_\e$ be defined as in Lemma~\ref{lem:bilip_holes}. Then,
\[
\tmu_\e \stackrel{*}{\weakly} \mu,
\]
where $\tmu_\e\in\calM(\W;\R^2)$ are given by
\[
d\tmu_\e = \frac{1}{n_\e} \sum_{i=1}^{m_\e}   \v_\e^i \otimes \frac{\ind_{D_\e^i}}{|D_\e^i|} \VolumeE .
\]
\end{proposition}

\begin{proof}
For $\psi\in C^\infty_c(\W;\R^2)$, let $\psi_\e$ be as in \lemref{lem:5.9}. Then
\beq
\begin{split}
\int_\W \TR(\psi\otimes d\mu) &= \LimEps \frac{1}{n_\e\e} \int_{\MEps} \TR(\PEps \wedge d\psi_\e) \\
&= \LimEps \frac{1}{n_\e\e} \sum_{i=1}^{m_\e}   \oint_{\pl D_\e^i} \TR(\PEps \otimes \psi_\e )  \\
&= \LimEps\frac{1}{n_\e} \sum_{i=1}^{m_\e}  \left\langle\frac{1}{|D_\e^i|} \int_{D_\e^i} \psi\, \VolumeE , \v_\e^i\right\rangle \\
&=\LimEps \int_\W \TR(\psi\otimes d\tmu_\e),
\end{split}
\label{eq:tildemu_mu}
\eeq
where in the transition to the third line we use \eqref{eq:Tpsi_const_on_bdry}.
By the Burgers vector bound \lemref{lem:max_disloc}(a),
\[
\int_\W d|\tmu_\e| \le \frac{1}{n_\e} \sum_{i=1}^{m_\e} |\v_\e^i| \lesssim 1.
\]
Thus, $\tmu_\e$ is uniformly bounded in $\calM(\W;\R^2)$ and therefore \eqref{eq:tildemu_mu} extends to all $\psi\in C_c(\W;\R^2)$.
\end{proof}

\subsection{Geometric rigidity}\label{sec:rigidity}

In this section we prove a uniform geometric rigidity statement for converging bodies with dislocations.
\begin{theorem}
\label{thm:rigidity}
Let 
\[
(\MEps,\PEps) 
\xrightarrow[]{n_\e}
(\W,\IdRtwo,\mu).
\]
For every $f_\e \in \HOne(\MEps;\R^2)$, there exists  a matrix $U_\e\in \SO(2)$, such that
\[
\|df_\e - U_\e \PEps\|_{L^2(\MEps)}^2 \lesssim \int_{\MEps} \dist^2(df_\e,\SO(\gEps,\euc))\,\VolumeEps + h_\e^2,
\]
where the constant depends on $\W$ and on the bilipschitz constant of the embedding of $(\MEps,\PEps)$ into $(\W,\IdRtwo)$.
\end{theorem}

\begin{proof}
Using the triangle inequality, for every $V\in\SO(2)$, 
\[
\begin{split}
|df_\e - V\PEps|_{\gEps,\euc}^2 &\lesssim |df_\e - V|_{\gEps,\euc}^2 + |V(\IdRtwo - \PEps)|_{\gEps,\euc}^2 \\
&\lesssim |\IdRtwo|_{\gEps,\euc} |df_\e - V|_{\euc,\euc}^2 +  |\IdRtwo - \PEps|_{\gEps,\euc}^2.
\end{split} 
\]
We integrate over $\MEps$ with respect to the volume form $\VolumeEps$.  The second term is $O(h_\e^2)$ using the global distortion bound \eqref{eq:global_dist_bnd}.
For the first term, we use the uniform boundedness of $|\IdRtwo|_{\gEps,\euc}$ and the equivalence of the volume forms $\VolumeEps$ and $\VolumeE$ to obtain
\[
\int_{\MEps} |df_\e - V\PEps|_{\gEps,\euc}^2\,\VolumeEps \lesssim
\int_{\MEps}  |df_\e - V|_{\euc,\euc}^2\,\VolumeE + h_\e^2.
\]

In Proposition~\ref{prop:FJM_ZeMe} below we show that there exists for every $f_\e: \MEps \to \R^2$  a $U\in \SO(2)$ such that
\beq
\label{eq:uniform_FJM}
\int_{\MEps} |df_\e - U|_{\euc,\euc}^2\,\VolumeE 
\lesssim
\int_{\MEps} \dist_{\euc,\euc}^2(df_\e,\SO(2))\,\VolumeE.
\eeq
It remains to reverse the estimate obtained at the beginning of the proof. 
First, using the equivalence of the volume forms,
\[
\int_{\MEps} \dist_{\euc,\euc}^2(df_\e,\SO(2))\,\VolumeE 
\lesssim 
\int_{\MEps} \dist_{\euc,\euc}^2(df_\e,\SO(2))\,\VolumeEps.
\]
For every $V\in\SO(2)$, using once again the uniformly bilipschitz property of the embedding of $(\MEps,\PEps)$ into $(\W,\IdRtwo)$,
\[
\begin{split}
|df_\e - V|^2_{\euc,\euc} &\lesssim |df_\e - V|^2_{\gEps,\euc} 
\lesssim |df_\e - V \PEps|^2_{\gEps,\euc} + |\PEps- \IdRtwo|^2_{\gEps,\euc}.
\end{split}
\]
Thus,
\[
\dist_{\euc,\euc}^2(df_\e,\SO(2)) \lesssim 
\dist_{\gEps,\euc}^2(df_\e,\SO(\gEps,\euc)) + |\PEps- \IdRtwo|^2_{\gEps,\euc}.
\]
Integrating over $\MEps$, and using once more the global distortion bound \eqref{eq:global_dist_bnd}, we finally obtain 
\[
\begin{split}
\int_{\MEps} |df_\e - U|_{\euc,\euc}^2\,\VolumeE 
&\lesssim \int_{\MEps} \dist^2_{\gEps,\euc}(df_\e,\SO(\gEps,\euc))\,\VolumeEps + h_\e^2.
\end{split}
\]
\end{proof}

The following proposition will be used for proving Proposition~\ref{prop:FJM_ZeMe}:

\begin{proposition}
\label{prop:uniform_FJM}
Let $\W\subset\R^2$ be a bounded, open, simply-connected domain with Lipschitz boundary.  
Let $x_1,\dots,x_m\in \W$ and $r_1,\dots,r_m>0$ such that the discs $B_{2r_i}(x_i)$ are disjoint and their closures in $\W$. 
Denote
\[ 
\W_h = \W \setminus \bigcup_{i=1}^{m} B_{r_i}(x_i). 
\]
Then there exists a constant $C>0$ depending only of $\W$, such that there exists for every $f\in \HOne(\W_h;\R^2)$ a matrix $U\in \SO(2)$, such that
\[
\int_{\W_h} |df - U|^2\,\VolumeE \le C \int_{\W_h} \dist^2(df,\SO(2))\,\VolumeE.
\] 
\end{proposition}

\begin{proof}
Let $f\in \HOne(\W_h;\R^2)$ be given.
Consider the annuli
\[
A_i = B_{2r_i}(x_i) \setminus B_{r_i}(x_i),
\]
which are by assumption disjoint. Since they are all similar, there exists a constant $C_{ann}>0$ and matrices $U_i\in\SO(2)$, such that for every $i$,
\[
\int_{A_i} |df - U_i|^2\,\VolumeE \le C_{ann} \int_{A_i} \dist^2(df,\SO(2))\,\VolumeE.
\]
Furthermore, by the Poincar\'e inequality, there exists a constant $C_P>0$ and vectors $b_i\in\R^2$, such that
\[
\begin{split}
\int_{A_i} |f - U_ix - b_i|^2\,\VolumeE &\le C_P r_i^2 \int_{A_i} |df - U_i|^2\,\VolumeE \\
&\le C_{ann} C_P r_i^2  \int_{A_i} \dist^2(df,\SO(2))\,\VolumeE.
\end{split}
\]
Define $\tf\in \HOne(\W;\R^2)$ as follows
\[
d\tf(x) = \Cases{
df & x\in \W\setminus  \bigcup_{i=1}^n B_{2r_i}(x_i) \\
df - d(\vp(|x - x_i|/r_i)(f - U_ix - b_i)) & x\in A_i \\
U_i & x\in B_{r_i}(x_i),
}
\]
where $\vp\in C^\infty([1,2])$ satisfies $0\le \vp \le 1$, $\vp'\le2$, $\vp=0$ in a neighborhood of $2$ and $\vp=1$ in a neighborhood of $1$. Note that the right-hand side has $L^2$-regularity and is weakly closed, hence it is the differential of an $H^1$-function.
Now,
\[
\begin{split}
\int_{\W_h} |df - d\tf|^2\,\VolumeE &= \sum_{i=1}^n \int_{A_i} |d(\vp(|x - x_i|/r_i)(f - U_ix - b_i)|^2\,\VolumeE \\
&\le \sum_{i=1}^n \brk{\frac{4}{r_i^2} \int_{A_i} |f - U_ix - b_i|^2\,\VolumeE +  \int_{A_i} |df - U_i|^2\,\VolumeE} \\
&\le \sum_{i=1}^n \brk{\frac{4}{r_i^2} C_{ann} C_P r_i^2+ C_{ann}}  \int_{A_i} \dist^2(df,\SO(2))\,\VolumeE \\
&\le (4C_{ann} C_P + C_{ann})  \int_{\W_h} \dist^2(df,\SO(2))\,\VolumeE.
\end{split}
\]
Furthermore, since $U_i\in\SO(2)$,
\[
\int_\W \dist^2(d\tf,\SO(2))\,\VolumeE = \int_{\W_\h} \dist^2(d\tf,\SO(2))\,\VolumeE.
\]
Finally, there exists a constant $C_\W$ and a matrix $U\in\SO(2)$, such that
\[
\int_\W |d\tf - U|^2\,\VolumeE \le C_\W \int_\W \dist^2(d\tf,\SO(2))\,\VolumeE,
\]
hence
\[
\int_{\W_h} |d\tf - U|^2\,\VolumeE \le C_\W \int_{\W_h} \dist^2(d\tf,\SO(2))\,\VolumeE.
\]
Putting it all together,
\[
\begin{split}
\int_{\W_h} |df - U|^2\,\VolumeE &\le 2\int_{\W_h} |d\tf - U|^2\,\VolumeE + 2\int_{\W_h} |d\tf - df|^2\,\VolumeE \\
&\le 2C_\W \int_{\W_h} \dist^2(d\tf,\SO(2))\,\VolumeE + 2\int_{\W_h} |d\tf - df|^2\,\VolumeE \\
&\le 4C_\W \int_{\W_h} \dist^2(df,\SO(2))\,\VolumeE + (2+4C_\W)\int_{\W_h} |d\tf - df|^2\,\VolumeE \\
&\le \brk{4C_\W + (2+4C_\e)(4C_{ann} C_P + C_{ann})}  \int_{\W_h} \dist^2(df,\SO(2))\,\VolumeE.
\end{split}
\]
\end{proof}

\begin{proposition}\label{prop:FJM_ZeMe}
Let 
\[
(\MEps,\PEps) 
\xrightarrow[]{n_\e}
(\W,\IdRtwo,\mu).
\]
Then the domains $\MEps$ satisfy the Friesecke--James--M\"uller rigidity theorem with a constant that is $\e$-independent,
that is, \eqref{eq:uniform_FJM} holds.
\end{proposition}

\begin{proof}
The proof is essentially the same as in \propref{prop:uniform_FJM}, where the annuli $A_i$ in \propref{prop:uniform_FJM} are replaced by $A_\e^i$ from Lemma~\ref{lem:bilip_holes}, which are Lipschitz equivalent to annuli of aspect ratio 2, with constants independent of $\e$ and $i$.
Therefore their FJM constants can be uniformly bounded.
For the construction of $\tf$ in \propref{prop:uniform_FJM}, one can compose $\vp$ with the bilipschitz map from $A_\e^i$ to the appropriate annulus in order to obtain the gluing along the boundaries of $A_\e^i$.
\end{proof}

\begin{remarks}
\begin{itemize}
\item It is interesting to compare \thmref{thm:rigidity_single_disloc} to \thmref{thm:rigidity}: 
Note that the proof of \thmref{thm:rigidity} does not rely directly on the closedness of $\PEps$ (i.e., on the fact that the implant maps represent dislocations), but only on the fact that $\PEps$ is close enough to a Euclidean implant map and that the geometric rigidity constant of the domains $M_\e$ can be controlled uniformly.
This is unlike the proof of  \thmref{thm:rigidity_single_disloc} which crucially depends on the closedness of $\PEps$. 
\item Comparing the statements of \thmref{thm:rigidity_single_disloc} to \thmref{thm:rigidity}, we note that in the former the bound has no $h_\e^2$ correction to the elastic energy.
An equivalent formulation in the setting of \thmref{thm:rigidity} would be
\[
\|df_\e - U\PEps\|_{L^2(\MEps)}^2 \lesssim \int_{\MEps} \dist^2(df_\e,\SO(\gEps,\euc))\,\VolumeEps.
\] 
We do not know if such a statement is true.
In the line of proof of \thmref{thm:rigidity_single_disloc}, the constant $C$ depends on the number of dislocations, 
implying the $h_\e^2$ correction can be omitted (using this approach) only if $m_\e$ is bounded.
\item Using the compactness and $\Gamma$-convergence results, we can prove an asymptotic version of such a theorem for $(\MEps,\PEps)  \xrightarrow[]{n_\e} (\W,\IdRtwo,\mu)$ and $\e$ small enough (under some conditions on $\mu$), with a constant depending explicitly on $\mu$, but independent of the number of dislocations.
	See Theorem~\ref{thm:rigidity2}.
\item In addition, we note that \thmref{thm:rigidity} bears some similarities in its structure, if not in the details or in the method of the proof, to estimates on incompatible strains \cite[Theorem 3.3]{MSZ14}, and to rigidity estimate for non-Euclidean energy \cite[Theorem 2.3]{LP11}.
A key difference between \thmref{thm:rigidity} and \cite[Theorem 3.3]{MSZ14} is that in \thmref{thm:rigidity} the elastic part of the right-hand side contains information about the dislocations via the implant map $\PEps$ and the metric $\gEps$, whereas in \cite[Theorem 3.3]{MSZ14} the elastic term is Euclidean, whose minimum is zero (hence the need for the extra term penalizing for $\curl \beta$ is essential).
\end{itemize}
\end{remarks}

\section{Compactness}\label{sec:compactness}

Having defined a notion of convergence of bodies with dislocation, we proceed to define a notion of convergence for configurations, or more precisely, a convergence of rescaled strains.
We then prove a compactness property for the rescaled strains for configurations of bounded energy.
Finally, in the critical and subcritical regimes, and under an appropriate separation assumption, we prove a compactness property for the measures $\torsion_\e$, in the sense that the energy bound \eqref{eq:global_dist_bnd} implies the convergence \eqref{eq:burgers_conv} of a subsequence.

\begin{definition}
\label{def:conv_fe}
Let 
\[
(\MEps,\PEps) 
\xrightarrow[]{n_\e}
(\W,\IdRtwo,\mu).
\]
We say that $f_\e \in \HOne(\MEps;\R^2)$ converges to $(J,U)$, where $J\in L^2(\W;\R^2\otimes\R^2)$ and $U\in\SO(2)$, if 
\[
\frac{U_\e^T df_\e - \PEps}{h_\e} \weakly J \qquad \text{in }\,\, L^2(\W)
\]
for some sequence $U_\e\in \SO(2)$ converging to $U$.
\end{definition}
In this definition and below, weak convergence in $L^2(\W)$ of sequences defined on a subset of $\W$ (namely $\MEps$) is defined as the convergence of their extension by zero to the whole domain $\W$.

The following proposition asserts that this notion of convergence defines a unique limit modulo an anti-symmetric matrix:

\begin{proposition}
Let $f_\e \to (J, U)$ and $f_\e \to (J', U')$ in the sense of \defref{def:conv_fe}. Then $U'=U$ and $J'$ differs from $J$ by a constant anti-symmetric matrix.
\end{proposition}

\begin{proof}
Suppose that $U_\e\to U$ and $U'_\e\to U'$ along with
\[
\frac{df_\e - U_\e \PEps}{h_\e} \weakly UJ
\Textand
\frac{df_\e - U'_\e \PEps}{h_\e} \weakly U'J' \qquad\text{in $L^2(\W)$}.
\]
Then,
\[
df_\e - U_\e \PEps \to 0 
\Textand
df_\e - U'_\e \PEps \to 0
\qquad\text{in $L^2(\W)$},
\]
from which follows that $U_\e - U'_\e\to0$, i.e., $U = U'$. Furthermore,
\[
\frac{(U'_\e - U_\e)\PEps}{h_\e} \weakly U(J - J')
\qquad \text{in  $L^2(\W)$}.
\]
i.e.,
\[
\frac{U'_\e - U_\e}{h_\e}  + (U'_\e - U_\e) \frac{\PEps - \IdRtwo}{h_\e}  \weakly U(J - J')
\qquad \text{in  $L^2(\W)$}.
\]
The second term of the left-hand side tends to zero strongly in $L^2$ by virtue of $U'_\e-U_\e\to0$ and the global distortion bound \eqref{eq:global_dist_bnd}, hence
\[
\frac{U'_\e - U_\e}{h_\e} 
\to U(J - J')
\qquad \text{in  $L^2(\W)$},
\]
the convergence being a strong convergence since the terms of the left-hand side are constant matrices. Thus, $J$  differs from $J'$ by a constant matrix. Finally, we can write $U_\e = U e^{A_\e}$ and $U'_\e = U e^{A'_\e}$, where $A_\e$ and $A'_\e$ are anti-symmetric matrices converging to zero, leading to 
\[
\frac{A_\e - A_\e'}{h_\e}  \to J - J',
\]
the limit on the left-hand side being an anti-symmetric matrix.
\end{proof}

\begin{theorem}[Strain compactness]
\label{thm:compactness}
Let 
\[
(\MEps,\PEps) 
\xrightarrow[]{n_\e}
(\W,\IdRtwo,\mu),
\]
and denote the elastic energy functional associated with the body manifolds $(\MEps,\PEps)$ by $E_\e(\cdot,\PEps):H^1(\MEps;\R^2) \to \R$.
Let $f_\e \in \HOne(\MEps; \R^2)$ be a sequence of mappings satisfying
\[
E_\e(f_\e,\PEps) \lesssim h_\e^2.
\]
Then, there exists a subsequence (not relabeled) of $f_\e$ such that $f_\e \to (J,U)$ in the sense of \defref{def:conv_fe}.
Furthermore, $J$ satisfies
\beq
\int_\W \TR(J\wedge d\psi) = \Cases{0 & \text{subcritical} \\
-\int_\W \TR(\psi\otimes d\mu) & \text{critical \& supercritical}}
\label{eq:limitJ}
\eeq
for all $\psi\in \HOneZero(\W;\R^2)\cap C_c(\W;\R^2)$. Finally, in the subcritical regime, $J = d\vp$ for some $\vp\in \HOne(\W;\R^2)$. (In the critical and supercritical regimes, \eqref{eq:limitJ} is the weak form of $dJ = -d\mu$, or equivalently, $\curl J = -\mu$.)
\end{theorem}

\begin{proof}
From the lower bound \eqref{eq:bound_calW} on the energy dentity and geometric rigidity (\thmref{thm:rigidity}), there exist matrices $U_\e \in \SO(2)$ such that
\[
\|df_\e - U_\e \PEps\|_{L^2(\MEps)} \lesssim \h_\e.
\]
By moving to a subsequence, we may assume by the compactness of $\SO(2)$ that $U_\e \to U$ for some $U \in \SO(2)$.
Consider the family of closed $\R^2$-valued 1-forms on $\W$,
\[
J_\e = \frac{U_\e^Tdf_\e - \PEps}{h_\e} \quad \text{in $\MEps$,}
\]
extended to zero on $\W\setminus \MEps$.
Since the embeddings of $(\MEps,\PEps)$ into $(\W,\IdRtwo)$ are uniformly bilipschitz,
the uniform boundedness of $J_\e$ in $L^2(\MEps,\PEps)$ implies its uniform boundedness in $L^2(\W)$. Thus, it  has a weakly convergent subsequence, $J_\e \weakly J$ in $L^2(\W)$. 

It remains to obtain relation \eqref{eq:limitJ} between $(J,U)$ and $\mu$.
By approximation, it suffices to prove \eqref{eq:limitJ} for $\psi\in C^\infty_c(\W;\R^2)$.
As before, we denote by $D^1_\e,\ldots, D^{m_\e}_\e$ the subdomains of $\W$ that are encircled by the cores of the $m_\e$ dislocations in $\MEps$. 
Let $\psi \in C^\infty_c(\W;\R^2)$ and let $\psi_\e \in C^\infty_c(\W)$ be as in \lemref{lem:chi_e}.
In particular, recall that $\psi_\e$ are constant on each $D^i_\e$, and their value there is $c^i_\e$ (the average of $\psi$ on these domains).

Since
\[
\Abs{\int_\W \TR(J_\e \wedge d(\psi-\psi_\e))} \lesssim \|J_\e\|_{L^2(\W)} 
\|d\psi-d\psi_\e\|_{L^2(\MEps)} \to 0,
\]
and since $J_\e\weakly J$, it follows that
\beq
\label{eq:J_e_aux1}
\begin{split}
\lim_{\e\to 0} \int_\W \TR(J_\e \wedge d\psi_\e) &= \lim_{\e\to 0} \int_\W \TR(J_\e \wedge d(\psi_\e - \psi)) + \lim_{\e\to 0} \int_\W \TR(J_\e \wedge d\psi) \\
&= \int_\W \TR(J \wedge d\psi).
\end{split} 
\eeq
Integrating by parts,  using that $J_\e$ are closed and $\psi_\e$ are constant on every $\pl D^i_\e$, we obtain,
\beq
\label{eq:J_e_aux2}
\begin{split}
\int_\W \TR(J_\e \wedge d\psi_\e) 
&=\sum_{i=1}^{m_\e} \oint_{\pl D^i_\e} \TR(\psi_\e \otimes J_\e) \\
&=  \frac{1}{\h_\e}\sum_{i=1}^{m_\e} \oint_{\pl D^i_\e} \TR(c_\e^i \otimes (U_\e^Tdf_\e - \PEps)) \\
&= -\frac{1}{\h_\e}\sum_{i=1}^{m_\e} \oint_{\pl D^i_\e} \TR(c_\e^i \otimes  \PEps ) \\
&= -\frac{1}{\h_\e}\sum_{i=1}^{m_\e} \oint_{\pl D^i_\e} \TR(\psi_\e \otimes  \PEps ) \\
&= - \frac{1}{\h_\e} \int_{\pl \MEps} \TR(\psi_\e\otimes  \PEps) \\
&= -\frac{n_\e \e}{\h_\e} \frac{1}{n_\e \e} \torsion_\e(\psi_\e).
\end{split}
\eeq
Taking $\e\to0$, since $\psi_\e \to\psi$ uniformly,
\[
\LimEps\frac{n_\e \e}{h_\e}  = \Cases{0 & \text{subcritical} \\ 1 & \text{critical \& supercritical}}
\]
and
\[
\LimEps\frac{1}{n_\e \e} \torsion_\e(\psi ) = \int_\W \TR(\psi\otimes d\mu) = \int_\W \TR(\psi\otimes d\mu),
\]
we obtain \eqref{eq:limitJ} from \eqref{eq:J_e_aux1}--\eqref{eq:J_e_aux2}. Finally, in the subcritical regime, \eqref{eq:limitJ} is the weak formulation of $dJ=0$, from which follows, since $\W$ is simply-connected, that $J$ is the weak differential of an $\HOne$-function.
\end{proof}

\begin{theorem}[Dislocation measures compactness]\label{thm:measure_comp}
Assume that $n_\e\lesssim \log(1/\e)$ (i.e., critical or subcritical regimes), and that the dislocations are well-separated, in the sense that the minimum separation $\rho_\e$ between dislocations in $\MEps$ (see \defref{def:body_w_disloc}(d)) satisfies $\rho_\e \gtrsim \e^s$ for some $s\in (0,1)$.
Assume that $(\M_\e,\PEps)$ converges to $(\W,\IdRtwo)$ in the sense of Definition~\ref{def:man_conv1}, and that, furthermore, the global distortion bound \eqref{eq:global_dist_bnd} holds.
Then, there exists a measure $\mu \in \calM(\W;\R^2)$ and a subsequence $(\M_\e,\PEps)$ converging to $(\W,\IdRtwo,\mu)$ in the sense of Definition~\ref{def:man_conv2}, i.e., \eqref{eq:burgers_conv} holds as well.
\end{theorem}

\begin{proof}
Let $(\M_\e,\PEps)$ be a sequence of bodies with dislocations satisfying the assumptions, with Burgers vectors $\{\e\v_\e^i\}_{i=1}^{m_\e}$, $\v_\e^i\in \bbS$.
We need to show that the measures $\torsion_\e$ satisfy $\|\torsion_\e\|=O(n_\e\e)$.
By \eqref{eq:norm_calT}, it suffices to show that
\[
\sum_{i=1}^{m_\e} |\v_\e^i| = O(n_\e).
\]
Since $\v_\e^i\in \bbS$, these vectors are uniformly bounded away from zero, and thus it suffices to show that
\[
\sum_{i=1}^{m_\e} |\v_\e^i|^2 = O(n_\e).
\]
Consider the balls
\[
B_{\rho_\e}^i = \{p\in\MEps~:~ \r_\e^i(p) < \rho_\e/2\}, 
\qquad i=1,\dots,m_\e.
\]
By definition, the balls $B_{\rho_\e}^i$ are disjoint, i.e., each ball is a body with a single dislocation; by Corollary~\ref{corr:Escaling}, it follows that
\[
\int_{B_{\rho_\e}^i} \dist^2(\IdRtwo,\SO(\g,\euc)) \,\VolumeEps \gtrsim \e^2|\v_\e^i|^2 \log\frac{\rho_\e}{\e|\v_\e^i|}.
\]
By \eqref{eq:global_dist_bnd}, 
\[
\begin{split}
\sum_{i=1}^{m_\e} \e^2|\v_\e^i|^2 \log\frac{\rho_\e}{\e|\v_\e^i|} &\lesssim
\sum_{i=1}^{m_\e} \int_{B_{\rho_\e}^i} \dist^2(\IdRtwo,\SO(\g,\euc)) \,\VolumeEps \\
&\le \int_{\M_\e} \dist^2(\IdRtwo,\SO(\g,\euc)) \,\VolumeEps \lesssim n_\e \e^2 \log(1/\e),
\end{split}
\]
hence 
\[
\sum_{i=1}^{m_\e} |\v_\e^i|^2  \log\frac{\rho_\e}{\e|\v_\e^i|} \lesssim n_\e \log(1/\e)
\]
(we used here the fact that $n_\e$ is not supercritical).
In particular, since by assumption $\frac{\rho_\e}{\e|\v_\e^i|} > 20$ for all $i$ (\defref{def:body_w_disloc}(d)), and since $n_\e\lesssim \log(1/\e)$, we have that $\max_{i} |\v_\e^i| \lesssim \log(1/\e)$. 
Thus, since $\rho_\e\gtrsim \e^s$,
\[
\log\frac{\rho_\e}{\e|\v_\e^i|} \gtrsim \log\frac{\e^s}{\e\log(1/\e)} \ge \frac{1-s}{2}\log (1/\e).
\]
Therefore, 
\[
\log (1/\e) \sum_{i=1}^{m_\e} |\v_\e^i|^2 \lesssim \sum_{i=1}^{m_\e} |\v_\e^i|^2 \log\frac{\rho_\e}{\e|\v_\e^i|} \lesssim n_\e \log(1/\e),
\]
which completes the proof.
\end{proof}

\begin{remark}
Note that in all the regimes, the above proof shows that $\|\torsion_\e\|=O(h_\e^2/\e\log(1/\e))$. In particular, if \eqref{eq:global_dist_bnd} holds with $h_\e^2 = n_\e \e^2\log(1/\e)$, then \eqref{eq:burgers_conv} holds as well, regardless of the regime.
In the supercritical case, this situation is equivalent to the one treated in \cite{FPP19} in the linear admissible strain model.
\end{remark}

\section{$\Gamma$-convergence}
\label{sec:GammaConv}

Let $X_\e$ be the space of all bodies $(\MEps,\PEps)$ containing dislocations, and let $n_\e\to \infty$ be a parameter.
In the sequel, we assume the following assumptions regarding $n_\e$ and the minimum separation $\rho_\e$ between dislocations in $\MEps$ for all $(\MEps,\PEps)\in X_\e$ (see \defref{def:body_w_disloc}(d)):
\begin{enumerate}[itemsep=0pt,label=(\alph*)]
\item
 $\log(1/\rho_\e) \ll \LogEps$, namely, $\rho_\e$ may tend to zero, however slower than any positive power of $\e$. 
\item 
$\log n_\e \ll  \LogEps$.
Thus, even in the supercritical regime, we assume that $n_\e$ does not grow faster than any negative power of $\e$. 
\end{enumerate}
Note that the separation assumption already implies that the number of dislocations does not grow faster than any negative power of $\e$; the assumption on $n_\e$ is slightly more restrictive, and implies also that the magnitude of all Burgers vectors in $\M_\e$ is at most $\e^{1-o(1)}$ (see Lemma~\ref{lem:max_disloc}(a)).

Let $\mathbb{X}_\e$ be the space of bodies with dislocations along with their configurations:
\[
\mathbb{X}_\e = \BRK{(f_\e,\PEps) ~:~ (\MEps,\PEps) \in X_\e, \,\, f_\e \in \HOne(\MEps;\R^2)},
\]
where we omit $\MEps$ explicitly in the tuple for notational brevity (it is implicit as the domains of $(f_\e,\PEps)$).
Define the rescaled energy $\calE_\e:  \mathbb{X}_\e\to\R$ ,
\[
\label{eq:rescaled_E}
\calE_\e(f_\e,\PEps) = \frac{1}{h_\e^2} E_\e(f_\e,\PEps) =  \frac{1}{h_\e^2} \int_{\MEps} \calW(df_\e\circ \PEps^{-1})\,\VolumeEps.
\]

In this section we prove that the sequence $\calE_\e$ $\Gamma$-converges to the functional
\[
\calE_0 : L^2\W^1(\W;\R^2) \times \calM(\W;\R^2) \to \R\cup\{\infty\}
\] 
defined by
\[
\calE_0(J,\mu) =
\Cases{
\calE^{\text{elastic}}_0(J) + \calE^{\text{self}}_0(\mu) & \text{subcritical and $\curl J =0$} \\
\calE^{\text{elastic}}_0(J) + \calE^{\text{self}}_0(\mu) & \text{critical, $\mu\in \HminusOne(\W;\R^2)$, and $\curl J = -\mu$} \\
\calE^{\text{elastic}}_0(J) & \text{supercritical, $\mu\in \HminusOne(\W;\R^2)$, and $\curl J = -\mu$} \\
\infty & \text{otherwise},
}
\]
where
\[
\calE^{\text{elastic}}_0(J) = \int_\W \bbW(J)\,\VolumeE,
\]
and
\[
\calE^{\text{self}}_0(\mu) = \int_\W \SEF\brk{\frac{d\mu}{d|\mu|}}\,d|\mu|.
\]
The $\Gamma$-convergence is with respect to the topology induced by the convergence of $(\MEps,\PEps)\xrightarrow[]{n_\e} (\W,\IdRtwo,\mu)$ (\defref{def:man_conv2}) and the convergence $f_\e\to(J,U)$ (\defref{def:conv_fe}). 

Specifically, we prove that:

\begin{enumerate}[itemsep=0pt,label=(\alph*)]
\item \textbf{Lower bound:}
For every sequence $(\MEps,\PEps)\xrightarrow[]{n_\e} (\W,\IdRtwo,\mu)$, and for every sequence of  $f_\e\to (J,U)$,   
\[
\liminf_{\e\to 0} \calE_\e(f_\e,\PEps) \ge \calE_0(J,\mu).
\]

\item \textbf{Upper bound, subcritical case:} for every Lipschitz domain $\W$ endowed with a measure $\mu\in\calM(\W;\R^2)$, and for every subcritical sequence $n_\e\to\infty$, there exists a recovery sequence of bodies with dislocations
\[
(\MEps,\PEps) 
\xrightarrow[]{n_\e}
(\W,\IdRtwo,\mu),
\]
for which the following property holds:
For every $U\in\SO(2)$ and $J\in L^2(\W;\R^2\otimes\R^2)$ satisfying $\curl J=0$, there exists a recovery sequence of configurations $f_\e\in H^1(\MEps;\R^2)$ converging to $(J,U)$, such that
\[
\limsup_{\e\to 0} \calE_\e(f_\e,\PEps) \le \calE_0(J,\mu).
\]

\item \textbf{Upper bound, critical and supercritical cases:} for every Lipschitz domain $\W$ endowed with a measure $\mu\in\calM(\W;\R^2)\cap \HminusOne(\W;\R^2)$, 
and for every critical or supercritical sequence $n_\e\to\infty$, there exists a recovery sequence of bodies with dislocations
\[
(\MEps,\PEps) 
\xrightarrow[]{n_\e}
(\W,\IdRtwo,\mu),
\]
for which the following property holds:
For every $U\in\SO(2)$ and $J\in L^2(\W;\R^2\otimes\R^2)$ satisfying $\curl J = -\mu$, there exists a recovery sequence of configurations $f_\e\in H^1(\MEps;\R^2)$ converging to $(J,U)$, such that
\[
\limsup_{\e\to 0} \calE_\e(f_\e,\PEps) \le \calE_0(J,\mu).
\]

\end{enumerate}

\begin{comment}
\begin{enumerate}
\item The upper bound statements are stronger than needed for proving $\Gamma$-convergence. Given $\W$, $\mu$, $U$ and $J$, we should construct a recovery sequence of joint bodies with dislocations $(\MEps,\PEps)$ and configurations $f_\e$. 
Instead, we construct $(\MEps,\PEps)$ independently of $U$ and $J$.
\item By \propref{prop:5.5}, the condition that $\mu\in \HminusOne(\W;\R^2)$ is necessary for critical or supercritical recovery sequences $(\MEps,\PEps)$ to exist. Likewise, by \eqref{eq:limitJ}, the restrictions on $\curl J$ are necessary for recovery sequences $(\MEps,\PEps)$ to exist.
\end{enumerate}
\end{comment}

\subsection{Lower bound}

\begin{theorem}[lim-inf inequality]
\label{thm:liminf}
Let 
\[
(\MEps,\PEps)\xrightarrow[]{n_\e} (\W,\IdRtwo,\mu),
\]
with $\log n_\e, \log(1/\rho_\e) \ll \log(1/\e)$.
Let $f_\e \in \HOne(\MEps;\R^2)$ converges to $(J,U)$ in the sense of \defref{def:conv_fe}. 
Then,
\[
\liminf_{\e\to 0} \calE_\e(f_\e,\PEps) \ge \calE_0(J,\mu).
\]
\end{theorem}

Towards the proof, we denote by $r_\e\le \rho_\e$ an infinitesimal sequence satisfying $n_\e r_\e^2 \ll 1$ and $\log(1/r_\e)\ll \log(1/\e)$, and
\beq
\label{eq:dislocation_balls}
B_\e^i = \{p\in\MEps~:~ \r_\e^i(p) < r_\e\}, 
\qquad i=1,\dots,m_\e
\eeq
the metric annulus of outer-radius $r_\e$ around the $i$-th dislocation in $\MEps$ ($r_\e$ is an intermediate scaling that is introduced to ensure the bound $n_\e r_\e^2 \ll 1$ that does not necessarily hold for $\rho_\e$). By the definition of $\rho_\e$, the annuli $B_\e^i$ are disjoint, hence
\[
\begin{split}
\calE_\e(f_\e,\PEps) &= \frac{1}{h_\e^2} \sum_{i=1}^{m_\e} \int_{B_\e^i} \calW(df_\e\circ \PEps^{-1})\,\VolumeEps + 
\frac{1}{h_\e^2} \int_{\MEps\setminus \cup_i B_\e^i} \calW(df_\e\circ \PEps^{-1})\,\VolumeEps \\
&\equiv \calE_\e^{\text{self}}(f_\e,\PEps) + \calE_\e^{\text{elastic}}(f_\e,\PEps).
\end{split}
\]
Since the balls are disjoint, $B_\e^i$ is isometric to a subset of $\hM_{\e \v_\e^i}$, and specifically, using \propref{prop:M_hM_M}, 
\beq
\label{eq:B_e_i_bounds}
\hM_{\e \v_\e^i}^{r_\e/2} \setminus \hM_{\e \v_\e^i}^{3\e |\v_\e^i|} \subset B_\e^i \subset \hM_{\e \v_\e^i}^{2r_\e}.
\eeq
Since the metrics $\hg_{\e \v_\e^i}$ and the Euclidean metrics on $\hM_{\e\v_\e^i}$ are uniformly equivalent, independent of $\e$ and $\v_\e^i$ (this follows from \eqref{eq:bilipschitz}), we have that $\Vol_{\hg_{\e \v_\e^i}}(\hM_{\e \v_\e^i}^{2r_\e}) \lesssim \Vol_{\euc} (B_{2r_\e}) \simeq r_\e^2$.
Thus
\[
\Vol_{\gEps}\brk{B_\e^i} \lesssim r_\e^2.
\]
Since by \lemref{lem:max_disloc}(c) the number of dislocations $m_\e$ satisfies $m_\e \lesssim n_\e$, we obtain that
\beq
\label{eq:B_e_i_vol}
\Vol_{\gEps}\brk{\bigcup_{i=1}^{m_\e}B_\e^i} \lesssim n_\e r_\e^2 \to 0.
\eeq

We prove \thmref{thm:liminf} by showing in Proposition~\ref{prop:liminf2} that
\[
\LiminfEps \calE_\e^{\text{self}}(f_\e,\PEps) \ge
\Cases{
\calE^{\text{self}}_0(\mu) & \text{subcritical \& critical} \\
0 & \text{supercritical},
}
\]
and in Proposition~\ref{prop:liminf1} that
\[
\LiminfEps \calE_\e^{\text{elastic}}(f_\e,\PEps) \ge
\calE^{\text{elastic}}_0(J).
\]

\begin{lemma}
\label{lem:dZ_Q_unif}
Let 
\[
(\MEps,\PEps)\xrightarrow[]{n_\e} (\W,\IdRtwo,\mu).
\]
Then, $|\IdRtwo - \PEps|_{\gEps,\euc} \to 0$ uniformly on $\MEps\setminus \cup_i B_\e^i$. Similarly, the convergence $|\IdRtwo- \PEps^{-1}|_{\euc,\gEps} \to 0$ 
is uniform on that same set.
\end{lemma}

\begin{comment}
Under the assumption $(\MEps,\PEps)\xrightarrow[]{n_\e} (\W,\IdRtwo,\mu)$, the uniform bilipschitzness of the inclusion maps $\MEps \to \W$ implies that the metric $\gEps = \PEps^\#\euc$ is uniformly equivalent to the Euclidean metric $\euc$, and thus the norms $|\cdot |_{\gEps,\euc}$, $ |\cdot|_{\euc,\euc}$, $ |\cdot|_{\euc,\gEps}$ are all uniformly equivalent. 
Hence, throughout this section, we will not always denote the norms, in order to simplify the notation.
\end{comment}

\begin{proof}
By \lemref{lem:max_disloc}(a), all the dislocations in $(\MEps,\PEps)$ are of magnitude of at most $C \e n_\e$, for some $C>0$ independent of $\e$.
From Property (b) in \defref{def:man_conv1},
\[
|\IdRtwo - \PEps| \lesssim \frac{\e n_\e}{r_\e} + c_\e
\qquad\text{ in $\MEps\setminus \cup_i B_\e^i$}.
\]
By our assumptions on $n_\e$ and $r_\e$, this tends to zero uniformly.
The convergence of the inverse maps follows similarly.
\end{proof}

In the proofs of the lower and upper bounds of $\calE_\e^{\text{self}}(f_\e;\PEps)$, we need the following result \cite[Theorems~2.38 and 2.39]{AFP00}:

\begin{proposition}[Reshetnyak]
\label{prop:Resh_lsc}
Let $X$ be a locally-compact separable metric space. Let $\mu_n,\mu\in \calM(X)$ be $\R^k$-valued Radon measures having finite total mass, such that $\mu_n\stackrel{*}{\weakly} \mu$. Then,
\[
\liminf_{\e\to0} \int_X f\brk{x,\frac{d\mu_n}{d|\mu_n|}} \, d|\mu_n| \ge \int_X f\brk{x,\frac{d\mu}{d|\mu|}} \, d|\mu|, 
\] 
for every continuous $f:X\times \R^k\to \R$, which is 1-homogeneous and convex in its second argument, satisfying the growth bound $|f(x,\xi)| \le C |\xi|$ for some $C>0$.  
If in addition $|\mu_n|(X) \to |\mu|(X)$, then there is an equality.
\end{proposition}

\begin{proposition}
\label{prop:liminf2}
Under the assumptions of \thmref{thm:liminf}, 
\[
\LiminfEps \calE_\e^{\text{self}}(f_\e,\PEps) \ge
\Cases{
\calE^{\text{self}}_0(\mu) & \text{subcritical \& critical} \\
0 & \text{supercritical}.
}
\]
\end{proposition}

\begin{proof}
In the supercritical case there is nothing to prove. Assume therefore a subcritical or critical regime, i.e., $h_\e^2 = n_\e \e^2 \log(1/\e)$.
Each $B_\e^i$ is a manifold with a single dislocation whose Burgers' vector is $\e\v_\e^i$.
Thus,
\[
\begin{split}
\LiminfEps \calE_\e^{\text{self}}(f_\e,\PEps)  &= 
\LiminfEps \frac{1}{h_\e^2} \sum_{i=1}^{m_\e} \int_{B_\e^i} \calW(df_\e\circ \PEps^{-1})\,\VolumeEps  \\
&\ge \LiminfEps \frac{1}{h_\e^2} \sum_{i=1}^{m_\e} \inf_{\tf_\e\in \HOne(\hM_{\e\v_\e^i};\R^2)}\int_{\hM_{\e\v_\e^i}^{r_\e/2}\setminus\hM_{\e\v_\e^i}^{3\e |\v_\e^i|}} \calW(d\tf_\e\circ \hP_{\e\v_\e^i}^{-1})\,\dVol_{\hP_{\e\v_\e^i}} \\
&= \LiminfEps \frac{\log(r_\e/6|\v_\e^i|\e)}{n_\e\log(1/\e)}\sum_{i=1}^{m_\e} \hI_{\e,6\e|\v_\e^i|/r_\e}^{r_\e/2}(\v_\e^i) \\
&\ge \LiminfEps \frac{1}{n_\e}\sum_{i=1}^{m_\e} \hI_{\e,6\e|\v_\e^i|/r_\e}^{r_\e/2}(\v_\e^i),
\end{split}
\]
where $\hI_{\e,6\e|\v_\e^i|/r_\e}^{r_\e/2}$ was defined in Section~\ref{sec:asymp_estimates}, the transition to the second line follows from \eqref{eq:B_e_i_bounds}, and the transition to the last line follows from the fact that $|\log(r_\e)|\ll \log(1/\e)$ and $|\log(|\v_\e^i|)|\lesssim \log n_\e \ll \log(1/\e)$.

Thus, it suffices to prove that
\[
\LiminfEps \frac{1}{n_\e} \sum_{i=1}^{m_\e}\hI_{\e,6\e|\v_\e^i|/r_\e}^{r_\e/2}(\v_\e^i) \ge \calE^{\text{self}}_0(\mu).
\]

We now apply Proposition~\ref{prop:lower_bnd_single2} with $R_\e = r_\e/2$: fixing $s\in (1/2,1)$ and $\delta\in (0,1/10)$, we have that for $\e$ small enough (depending on $s$),
\[
\hI_{\e,6\e|\v_\e|/r_\e}^{r_\e/2}(\v_\e^i) \ge s\, \Ilin_0(\v_\e^i)\brk{1 - C \brk{\frac{1}{\log(1/\delta)} + \tilde{\sigma}_{\e,\delta}}}.
\]
Thus,
\[
\begin{split}
\frac{1}{n_\e}\sum_{i=1}^{m_\e} \hI_{\e,6\e|\v_\e^i|/r_\e}^{r_\e/2}(\v_\e^i)
&\ge s\brk{1 - C \brk{\frac{1}{\log(1/\delta)} + \tilde{\sigma}_{\e,\delta}}} \frac{1}{n_\e}\sum_{i=1}^{m_\e} \Ilin_0(\v_\e^i) \\
&\ge s\brk{1 - C \brk{\frac{1}{\log(1/\delta)} + \tilde{\sigma}_{\e,\delta}}} \frac{1}{n_\e}\sum_{i=1}^{m_\e} \SEF(\v_\e^i) \\
&= s\brk{1 - C \brk{\frac{1}{\log(1/\delta)} + \tilde{\sigma}_{\e,\delta}}} \frac{1}{n_\e}\sum_{i=1}^{m_\e} \SEF\brk{\frac{\v_\e^i}{|\v_\e^i|}}|\v_\e^i|\\
&= s\brk{1 - C \brk{\frac{1}{\log(1/\delta)} + \tilde{\sigma}_{\e,\delta}}} \int_\W \SEF\brk{\frac{d\tmu_\e}{d|\tmu_\e|}}\,d|\tmu_\e|,
\end{split}
\]
where the passage to the second line follows from the inequality $\SEF(\v) \le \Ilin_0(\v)$ (which follows from the the definition of $\SEF(\v)$), and the passage to the third line follows from the 1-homogeneity of $\SEF$ (\lemref{lem:4.15}(a)).
In the passage to the last line, we use the measure 
\[
d\tmu_\e = \frac{1}{n_\e} \sum_{i=1}^{m_\e}  \v_\e^i \otimes \frac{\ind_{D_\e^i}}{|D_\e^i|} \VolumeE  
\]
defined in \propref{prop:mu_e_mu}. Since
\[
d|\tmu_\e| = \frac{1}{n_\e} \sum_{i=1}^{m_\e} |\v_\e^i|  \frac{\ind_{D_\e^i}}{|D_\e^i|} \VolumeE,
\]
it follows that
\[
\frac{d\tmu_\e}{d|\tmu_\e|} = \sum_{i=1}^{m_\e} \frac{\v_\e^i}{|\v_\e^i|} \ind_{D_\e^i},
\]
hence
\[
\begin{split}
\int_\W \SEF\brk{\frac{d\tmu_\e}{d|\tmu_\e|}}\,d|\tmu_\e| &=
\frac{1}{n_\e} \sum_{i=1}^{m_\e} \int_\W \SEF\brk{\frac{d\tmu_\e}{d|\tmu_\e|}}   |\v_\e^i| \frac{\ind_{D_\e^i}}{|D_\e^i|} \VolumeE  \\
&= \frac{1}{n_\e} \sum_{i=1}^{m_\e} \int_\W \SEF\brk{\frac{\v_\e^i}{|\v_\e^i|}}   |\v_\e^i| \frac{\ind_{D_\e^i}}{|D_\e^i|} \VolumeE \\
&= \frac{1}{n_\e} \sum_{i=1}^{m_\e} \SEF\brk{\frac{\v_\e^i}{|\v_\e^i|}} |\v_\e^i|.
\end{split}
\]
By \propref{prop:mu_e_mu}, $\tmu_\e \stackrel{*}{\weakly} \mu$.
By \lemref{lem:4.15}, $\SEF$ satisfies the assumptions of Reshetnyak's theorem (\propref{prop:Resh_lsc}), hence by taking first $\e \to 0$ we obtain
\[
\LiminfEps \calE_\e^{\text{self}}(f_\e,\PEps) \ge 
s\brk{1 - C \brk{\frac{1}{\log(1/\delta)} }} \calE^{\text{self}}_0(\mu).
\]
Letting $\delta\to 0$ and then $s\to 1$ completes the proof.
\end{proof}

\begin{proposition}
\label{prop:liminf1}
Under the assumptions of \thmref{thm:liminf},
\beq
\LiminfEps \calE_\e^{\text{elastic}}(f_\e,\PEps) \ge \calE^{\text{elastic}}_0(J).
\label{eq:liminfI1}
\eeq
\end{proposition}

\begin{proof}
We need to prove that
\[
\LiminfEps \frac{1}{h_\e^2} \int_{\MEps\setminus \cup_i B_\e^i} \calW(df_\e\circ \PEps^{-1})\,\VolumeEps 
\ge 
\int_\W \bbW(J)\,\VolumeE.
\]
Denote $J_\e = (U_\e^T df_\e - \PEps)/h_\e$.
It is given that $J_\e \weakly J$ in $L^2$.
Define the characteristic functions $\chi_\e :\MEps \to\R$,
\[
\chi_\e = \ind_{|J_\e|\le h_\e^{-1/2}} \ind_{\MEps\setminus \cup_i B_\e^i}.
\]
Since the sequence $J_\e$ is uniformly bounded in $L^2$, and since the volume bound \eqref{eq:B_e_i_vol} holds, it follows from Markov's inequality that $\chi_\e$ converge to $1$ in measure boundedly. Since the product of an $L^2$-weakly converging sequence and a sequence converging in measure boundedly converges weakly in $L^2$ to the product of the limits, 
\[
\chi_\e J_\e \weakly J \qquad\text{in $L^2(\W;\R^2\otimes\R^2)$},
\]
By the uniform convergence $|\IdRtwo - \PEps^{-1}|\to0$ in $\MEps\setminus \cup_i B_\e^i$ (\lemref{lem:dZ_Q_unif}), it follows that
\beq
\chi_\e\, J_\e \PEps^{-1} \weakly J \qquad\text{in $L^2(\W;\R^2\otimes\R^2)$}.
\label{eq:chi_G_to_dpsi}
\eeq

Substituting the definition of $J_\e$ and using the frame-indifference of $\calW$,
\[
\begin{split}
\calE_\e^{\text{elastic}}(f_\e,\PEps) 
&\ge \frac{1}{h^2_\e} \int_{\MEps} \chi_\e \calW(U_\e (I + h_\e J_\e \PEps^{-1}))\,\VolumeEps \\
&= \frac{1}{h^2_\e} \int_{\MEps} \chi_\e \calW(I + h_\e J_\e \PEps^{-1})\,\VolumeEps.
\end{split}
\]
As in the proof of \propref{prop:cell_liminf}, we write
\[
\begin{split}
\calW(I +h_\e J_\e \PEps^{-1}) 
&= h_\e^2 \bbW(J_\e \PEps^{-1}) + \omega(h_\e^2 |J_\e|^2),
\end{split}
\]
where $\omega(x)/x^2\to0$ as $x\to0$.
Since $\chi_\e\ne0$ implies that $|J_\e|^2\le h_\e^{-1}$, i.e., $h_\e^2 |J_\e|^2\le h_\e$, we obtain that
\[
\begin{split}
\LiminfEps  \calE_\e^{\text{elastic}}(f_\e,\PEps) 
 &\ge  \LiminfEps  \int_{\MEps} \chi_\e \brk{  \bbW(J_\e \PEps^{-1}) +|J_\e|^2\frac{\omega(h_\e^2 |J_\e|^2)}{h_\e^2 |J_\e|^2}} \,\VolumeEps \\
&=  \LiminfEps  \int_{\MEps} \chi_\e  \bbW(J_\e \PEps^{-1}) \,\VolumeEps \\
&=  \LiminfEps  \int_{\MEps} \bbW(\chi_\e J_\e \PEps^{-1}) \,\VolumeEps.
\end{split}
\]
By the lower-semicontinuity of quadratic forms and the weak convergence \eqref{eq:chi_G_to_dpsi},
it follows that
\[
\LiminfEps  \calE_\e^{\text{elastic}}(f_\e,\PEps)  \ge \int_{\W} \bbW(J)\,\VolumeE
= \calE^{\text{elastic}}_0(J).
\]
\end{proof}

\subsection{Upper bound}
\label{sec:recovery_seq}

In this section we prove:

\begin{theorem}[lim-sup inequality]
\label{thm:limsup}
Let $\W\subset \R^2$ be a simply-connected Lipschitz domain and let $\mu\in \calM(\W;\R^2)$. For every subcritical sequence $n_\e\to\infty$, there exists a sequence of bodies with dislocations
\[
(\MEps,\PEps)\xrightarrow[]{n_\e} (\W,\IdRtwo,\mu),
\]
for which the following property holds:
For every $U\in \SO(2)$ and $J \in L^2(\W;\Hom(\R^2))$ satisfying $\curl J = 0$, 
there exist configurations $f_\e \in \HOne(\MEps;\R^2)$ such that $f_\e\to (J,U)$,
and
\[
\limsup_{\e\to 0} \calE_\e(f_\e,\PEps) \le \calE_0(J,\mu).
\]
If $\mu \in \calM(\W;\R^2)\cap \HminusOne(\W;\R^2)$, then the statement holds for $n_\e$ in either regime; in the critical and supercritical regimes, $J$ has to satisfy $\curl J = - \mu$.
\end{theorem}

The proof is partitioned into three principal steps.
Given $\W$, $\mu$ and $n_\e$, we construct manifolds with dislocations similar to the construction presented in \secref{sec:disloc_construction}:
For every $\e$, we choose points $p_\e^1,\ldots, p_\e^{m_\e} \in \W$ and vectors $\v_\e^1,\ldots, \v_\e^{m_\e}\in \R^2$,  with $m_\e \simeq n_\e$, such that the corresponding combinations of $\R^2$-valued $\delta$-measures 
\[
\mu_\e = \sum_{i=1}^{m_\e} \e \v_\e^i \otimes \delta_{p_\e^i}
\]
approximate $\mu$, and proceed as in Section~\ref{sec:disloc_construction}.
The challenge is as follows: as seen in the proof of \propref{prop:liminf2}, the rescaled self-energy is estimated by a term of the form $\sum_i \Ilin_0(\v_\e^i)$, which is bounded by $\SEF$ from below. 
Thus, the vectors $\v_\e^i$ must be chosen optimal for achieving the relaxation $\SEF$,

This is done in Step I, which bears similarities with previous constructions in \cite{GLP10} and in \cite[Theorem~4.6]{MSZ14}.
In Step II, we construct the manifolds $(\MEps,\PEps)$ from the measures $\mu_\e$, and prove that they converge to $(\W,\IdRtwo, \mu)$.
Finally, in Step III, we construct a recovery sequence of configurations, $f_\e$, and estimate their energy.

\paragraph{Step I: Approximation of $\mu$}

\begin{lemma}
\label{lem:approx_mu}
Let $\W\subset \R^2$ be a simply-connected Lipschitz domain, let $\mu\in \calM(\W;\R^2)$ and let $n_\e\to\infty$ (in either regime).
There exists a $K>0$ (independent of $\mu$) and a sequence $\mu_\e =  \sum_{i=1}^{m_\e} \e \v_\e^i \otimes \delta_{p_\e^i}$ of measures with $|\v_\e^i|< K$, supported on $m_\e \simeq n_\e$ points $p_\e^1\dots,p_\e^{m_\e}$, the distance between each two at least $a_\e$, with $\log (1/a_\e) \ll \log (1/\e)$, such that
\beq \label{eq:mu_e_L_conv}
\frac{1}{n_\e \e} \mu_\e \weakstar \mu \qquad \text{in }\, \calM(\W;\R^2),
\eeq
and
\beq
\label{eq:conv_of_self_energy}
\frac{1}{n_\e} \sum_{i=1}^{m_\e} \Ilin_0(\v_\e^i) \to \int_{\W} \SEF\brk{\frac{d\mu}{d|\mu|}} d|\mu|.
\eeq
Furthermore, denote
\[
d\tmu_\e = \sum_{i=1}^{m_\e} (\e\v_\e^i) \otimes \frac{\chi_{B_{a_\e}(p_\e^i)}}{\pi a_\e^2}\VolumeE  
\]
be a ``smeared" version of $\mu_\e$, constant over discs of radius $a_\e$.
Then, $\|\tmu_\e\|_{H^{-1}}\lesssim h_\e$, and if $n_\e$ is subcritical, then $\|\tmu_\e\|_{H^{-1}} \ll h_\e$.
If $\mu \in H^{-1}(\W;\R^2)$, then we further have
\beq \label{eq:tmu_e_L_conv}
\frac{1}{n_\e \e} \tmu_\e \to \mu \qquad \text{in }\, H^{-1}(\W;\R^2).
\eeq
\end{lemma}

\begin{proof}
Let us first assume that $\mu$ is locally-constant, 
\[
d\mu = \sum_{j=1}^{J} \bxi^j \otimes \chi_{\W_j}\,\VolumeE  
\] 
for some  $\bxi^j \in \R^2\setminus \{0\}$, $j=1,\dots,J$, where $\W_j \subset \W$ are pairwise disjoint squares of edge-length $L$.
For the rest of this proof, we refer to such measures as ``locally-constant on squares''.

By items (c) and (d) in \lemref{lem:4.15}, there exists a $K>0$, such that
for every $j=1,\dots,J$, 
\beq
\label{eq:splitting_vectors}
\bxi^j =  \sum_{k=1}^{M_j} \lambda_k^j \v_k^j 
\Textand
\SEF(\bxi^j) = \sum_{k=1}^{M_j} \lambda_k^j \Ilin_0(\v_k^j),
\eeq
for some $\v_k^j \in \bbS\cap B_{K}$, $\lambda_k^j >0$ and $M_j\in\bbN$. 
Denote
\[
\Lambda = \max_{j=1}^{J} \sum_{k=1}^{M_j} \lambda_k^j 
\Textand 
a_{\e} = \frac{L}{\lceil (L (\Lambda n_\e)^{1/2}\rceil}.
\]
That is, $a_\e \simeq (\Lambda n_\e)^{-1/2}$ is an asymptotically-vanishing length scale which divides $L$.
By our assumptions on $n_\e$, it follows that $\log(1/a_\e) \ll \log(1/\e)$, i.e., $a_\e$ qualifies as a lower bound on the inter-defect separation.
From the $1$-homogeneity of $\SEF$, 
for every $j=1,\dots,J$,
\[
\SEF\brk{\frac{\bxi^j}{|\bxi^j|}} = \frac{\sum_{k=1}^{M_j} \lambda_k^j \Ilin_0(\v_k^j)}{|\bxi^j|},
\]
from which follows that for every $j=1,\dots,J$, 
\[
0 < \frac{\min_{|\bxi|=1} \SEF(\bxi)}{\max_{\bbS\cap B_{K}} \Ilin_0(\v)}
\le 
\frac{\sum_{k=1}^{M_j} \lambda_k^j}{|\bxi^j|} 
\le 
\frac{\max_{|\bxi|=1} \SEF(\bxi)}{\min_{\bbS\cap B_{K}} \Ilin_0(\v)}
< \infty,
\]
and the bounds are independent of $\mu$ (they only depend on the dislocation structure).
Thus, 
\[
\Lambda\simeq \max_{j=1}^J |\bxi^j| = \|d\mu/dx\|_\infty, 
\]
and consequently, 
\[
a_\e \simeq (\|d\mu/dx\|_\infty  n_\e)^{-1/2}.
\]
For every $j=1,\dots,J$, denote by $\{p_{\e,i}^j\}_i$ the centers of a partition of $\W_j$ into squares of edge-length $3a_\e$.
From \cite[Lemma~14]{GLP10}, we can choose measures
\[
\mu_\e =  \sum_{j=1}^J \sum_{k=1}^{M_j} (\e \v_k^j) \otimes \mu_{k,\e}^j  
\Textand
\tmu_\e =  \sum_{j=1}^J \sum_{k=1}^{M_j}  (\e \v_k^j)\otimes  \tmu_{k,\e}^j ,
\]
where $\mu_{k,\e}^j$ are sums of delta measures supported on a subset of $\{p_{\e,i}^j\}_{i}$, and $\tmu_{k,\e}^j$ are the measures obtained from $\mu_{k,\e}^j$ by replacing each $\delta_p$ in its support with $(\pi a_\e^2)^{-1} \chi_{B_{a_\e}(p)}$, such that 
\ref{eq:mu_e_L_conv} and \ref{eq:tmu_e_L_conv} hold (measures that are locally-constant on squares are in $H^{-1}$). Furthermore, for each $j$ and $k$,
\begin{align}
\label{eq:mu_e_L_conv2}
&\frac{1}{n_\e} \mu^j_{k,\e} \weakstar \lambda^j_k \chi_{\W_j} \,\VolumeE & \text{in } & \calM(\W).
\end{align}
The construction of the measures in \cite[Lemma~14]{GLP10} shows that $\mu_\e$ can be chosen to be supported on at most $C |\mu|(\W) n_\e$ points, where $C>0$ is independent of $\e$ and $\mu$.

It remains to show that \eqref{eq:conv_of_self_energy} holds  
(still limited to measures that are locally-constant on squares), namely, that
\[
\LimEps\frac{1}{n_\e} \sum_{j=1}^J \sum_{k=1}^{M_j} \sum_i \Ilin_0(\v_k^j) = \int_{\W} \SEF\brk{\frac{d\mu}{d|\mu|}} d|\mu|,
\]
where the sum over $i$ is over all points in the support of $\mu_{k,\e}^j$ (it is the only sum out of the three whose range depends on $\e$).
This sum can be replaced with $\mu_{k,\e}^j (\W)$, hence 
\[
\begin{split}
\LimEps\frac{1}{n_\e} \sum_{j=1}^J \sum_{k=1}^{M_j} \sum_i \Ilin_0(\v_k^j) &=
\sum_{j=1}^J \sum_{k=1}^{M_j}  \Ilin_0(\v_k^j) \LimEps \int_\W \frac{1}{n_\e} d\mu_{k,\e}^j   \\
&= \sum_{j=1}^J \sum_{k=1}^{M_j}  \Ilin_0(\v_k^j) \lambda_k^j |\W_j|   \\
&= \sum_{j=1}^J \SEF(\bxi^j)|\W_j|   \\
&= \sum_{j=1}^J \SEF\brk{\frac{\bxi^j}{|\bxi^j|}}\, |\bxi^j| |\W_j|   \\
&= \int_{\W} \SEF\brk{\frac{d\mu}{d|\mu|}} d|\mu|.
\end{split}
\]
where in the passage to the second line we used the weak convergence \eqref{eq:mu_e_L_conv2}, and in the passage to the third line we used the choice \eqref{eq:splitting_vectors} of the $\v_k^j$.

The convergence \eqref{eq:tmu_e_L_conv} implies that, 
$\|\tmu_\e\|_{H^{-1}} \simeq n_\e\e$, which in the subcritical case implies that $\|\tmu_\e\|_{H^{-1}} \ll h_\e$.
This completes the proof for measures that are locally-constant on squares.

For a general measure $\mu \in \calM(\W;\R^2)$, we construct a sequence of measures $\mu^k$, locally-constant on squares of edge length $L^k \to 0$, such that 
\beq
\mu^k\weakstar \mu
\Textand
|\mu^k|(\W) \to |\mu|(\W).
\label{eq:mu_conv}
\eeq
If $\mu \in H^{-1}(\W;\R^2)$, then in addition construct $\mu^k$ such that 
\[
\mu^k \to \mu
\qquad\text{in $H^{-1}$}.
\]
(One can construct such a sequence by mollifying $\mu$, approximating the resulting functions by piecewise-constant functions, and taking a diagonal sequence.)
Eq.~\eqref{eq:mu_conv} implies, using the second part of Reshetnyak's continuity theorem (\propref{prop:Resh_lsc}),  that
\[
\lim_{k\to\infty} \int_{\W} \SEF\brk{\frac{d\mu^k}{d|\mu^k|}} d|\mu^k| = \int_{\W} \SEF\brk{\frac{d\mu}{d|\mu|}} d|\mu|.
\]

The proof is completed by a diagonal argument. 
All the measures in question have bounded total variation, and the weak star topology of closed bounded sets in $\calM(\W;\R^2)$ is metrizable. Let $\mu_\e^k$ be above the construction adapted to $\mu^k$. We can choose $k_\e$ such that the diagonal sequence satisfies
\[
\frac{1}{n_\e \e} \mu_\e^{k_\e} \weakstar \mu.
\]
If $\mu\in H^{-1}(\W;\R^2)$, then the sequence $k_\e$ can be chosen such that in addition
\[
\frac{1}{n_\e \e} \tmu_\e^{k_\e} \to \mu
\qquad\text{in $H^{-1}$}.
\]
Finally, if $\mu\in H^{-1}(\W;\R^2)$, then $\|\tmu_\e^{k_\e}\|_{H^{-1}} \lesssim n_\e\e \le h_\e$.
If  $n_\e$ is subcritical, then regardless of whether $\mu$ is in $H^{-1}$ or not,  $k_\e$ can be chosen such that $\frac{1}{n_\e\e}\|\tmu_\e^{k_\e}\|_{H^{-1}}$ blows up as slow as we like, and in particular, 
\[
\frac{1}{n_\e \e} \|\tmu_\e^{k_\e}\|_{H^{-1}} \ll \sqrt{\frac{\log(1/\e)}{n_\e}},
\]
since the right-hand side tends to infinity in the subcritical regime. 
Thus, $\|\tmu_\e^{k_\e}\|_{H^{-1}} \ll \sqrt{n_\e \e^2 \,\log(1/\e)} = h_\e$.

Finally, we can choose $\|d\mu^k/dx\|_\infty$ to blow up slowly enough, such that 
$a_\e \simeq (\|d\mu^{k_\e}_\e/dx\|_\infty  n_\e)^{-1/2}$ satisfies the separation bound $\log(1/a_\e) \ll \log(1/\e)$.
Note also that by construction, $m_\e \lesssim |\mu^{k_\e}|(\W) n_\e \lesssim |\mu|(\W) n_\e$, where we use \eqref{eq:mu_conv} again.
\end{proof}

\paragraph{Step II: Construction of $(\MEps, \PEps)$}

\begin{lemma}
\label{lem:upper_bound_II}
Let $\mu\in \calM(\W;\R^2)$ and let $n_\e\to\infty$ be a subcritical sequence; if $\mu \in  \calM(\W;\R^2)\cap H^{-1}(\W;\R^2)$ then $n_\e\to\infty$ can be in either regime.
Let $\mu_\e$ be an approximating sequence for $\mu$ as in \lemref{lem:approx_mu}.
Let $(\MEps, \PEps)$ be the manifolds with dislocations associated with $\mu_\e$ according to the construction of Section~\ref{sec:disloc_construction}.
Then 
\[
(\MEps,\PEps)\xrightarrow[]{n_\e} (\W,\IdRtwo,\mu).
\]
Furthermore, the following bound holds,
\beq\label{eq:Z_eQ_e_recovery}
|\IdRtwo - \PEps| \lesssim \frac{\e}{\r} + \frac{\e}{\rho_\e^2},
\eeq
where the minimum separation parameter $\rho_\e$ satisfies $\log(1/\rho_\e)\ll \log(1/\e)$.
\end{lemma}

\begin{proof}
For given $\e$, we first identify the parameters $b$ and $a$ in \propref{prop:many_disloc_bound} with
\[
b = \max_{i=1}^{m_\e} \e |\v_\e^i| \lesssim \e
\Textand
a = a_\e,
\]
and since $\log(1/a_\e)\ll \LogEps$,
the requirement that $b/a, b/a^2 < c$ holds for $\e$ small enough.
By construction,
\[
\MEps = \W\setminus \brk{\bigcup_{i=1}^{m_\e} B_{3\e|\v_\e^i|/2}(p_\e^i)}.
\] 
The asymptotic surjectivity of the inclusion maps follows since 
\[
\Abs{\bigcup_{i=1}^{m_\e} B_{3\e|\v_\e^i|/2}(p_\e^i)}\lesssim \e^2 m_\e \lesssim \e^2 n_\e \to 0.
\]
It follows from \propref{prop:many_disloc_bound} that
\[
|\IdRtwo|_{\PEps^\#\euc,\euc},\,\, |\IdRtwo|_{\euc,\PEps^\#\euc} \lesssim 1,
\]
hence the inclusion map is uniformly bilipschitz, as the intrinsic distance on $(\MEps,\IdRtwo)$ are uniformly equivalent to the distances on $\MEps$ as a subset of $\R^2$.
As a result, the minimal separation $\rho_\e$ between defects in $\MEps$ satisfies $\rho_\e\simeq a_\e$, and in particular, $\log(\rho_\e) \ll \LogEps$.
Furthermore,
\[
|\IdRtwo - \PEps|_{\PEps^\#\euc,\euc} \lesssim \frac{\e|\v_\e^i|}{\r_\e} + \frac{\e}{\rho_\e^2},
\]
hence the asymptotic rigidity requirement \eqref{eq:local_dist_bnd} is satisfied.
Moreover, since $|\v_\e^i| \le K$ for all $\e$ and $i$, the pointwise bound \eqref{eq:Z_eQ_e_recovery} holds, 
and since moreover $\|\tmu_\e\|_{H^{-1}} \lesssim h_\e$, we have from \eqref{eq:5.4estimate} that
\[
\begin{split}
\int_{\MEps} |\IdRtwo - \PEps|_{\P^\#\euc,\euc}^2 \,\dVol_{\P^\#\euc} 
&\lesssim \sum_{i=1}^{m_\e} (\e |\v_\e^i|)^2 \log\brk{\frac{\rho_\e}{\e |\v_\e^i|}} + \|\tmu\|_{\HminusOne(\W)}^2 \\
&\lesssim  m_\e \e^2 \log(1/\e) + h_\e^2 \lesssim h_\e^2,
\end{split}
\]
hence the global distortion bound \eqref{eq:global_dist_bnd} is satisfied.
It remains to show the Burgers vector convergence \eqref{eq:burgers_conv}.
This follows from the same argument as in Example~\ref{ex:conv2}, using the fact that $\frac{1}{n_\e\e} \mu_\e \weakstar \mu$ and the bound $m_\e \lesssim n_\e$.
\end{proof}

\paragraph{Step III: Construction of $f_\e$}

\begin{lemma}
\label{lem:tf_e}
Let $(\MEps,\PEps)$ be as in Step~II,
i.e.,
\[
(\MEps,\PEps)\xrightarrow[]{n_\e} (\W,\IdRtwo,\mu),
\]
and
\[
|\IdRtwo - \PEps| \lesssim \frac{\e}{\r} + \frac{\e}{\rho_\e^2},
\]
where $n_\e$ can be in either regime.
Let $r_\e\le \rho_\e^2$ be a sequence satisfying $n_\e r_\e^2 \to 0$ and $\log(1/r_\e) \ll \log(1/\e)$.
Then, there exist functions $\tf_\e:\MEps \to \R^2$ satisfying 
\[
\tf_\e|_{\MEps\setminus \cup_i B_\e^i} = \id,
\]
where $B_\e^i$ is the annulus of radius $r_\e$ around the $i$-th dislocation, as defined in \eqref{eq:dislocation_balls}, and for every $s\in (0,1)$
\beq\label{eq:Quad_beta_e}
\frac{1}{\log(1/\e)} \int_{B_\e^i\setminus B_\e^{i,s}} \bbW(\beta_\e)\,\VolumeEps \le \Ilin_0(\v_\e^i)\brk{s+ O\brk{\e^{1-s} + \frac{\log(1/r_\e)}{\log(1/\e)}}},
\eeq
where 
\[
\beta_\e = \frac{d\tf_\e\circ \PEps^{-1} - I}{\e},
\]
and
\beq\label{eq:B_e_i_s}
B_\e^{i,s} = \{p\in\MEps~:~ \r_\e^i(p) < 10K\e^s \},
\qquad i=1,\dots,m_\e,
\eeq
where $K>0$ is the bound on all $|\v_\e^i|$, as in Lemma~\ref{lem:approx_mu}.
Furthermore, we have the pointwise bound 
\beq\label{eq:tf_e_uni_bound}
|d\tf_\e - \PEps| \lesssim \frac{\e}{\r} + \frac{\e}{\rho_\e^2},
\eeq
and the $L^2$-bound
\beq\label{eq:tf_e_L2_bound}
\int_{\MEps} |d\tf_\e - \PEps|^2 \,\VolumeEps \lesssim h_\e^2.
\eeq
\end{lemma}

\begin{proof}
By Proposition~\ref{prop:M_hM_M}  for $\e$ small enough (independently of $s$),
\[
\hM_{\e\v_\e^i}^{r_\e/2}\setminus \hM_{\e\v_\e^i}^{20K\e^s} \hookrightarrow B_\e^i\setminus B_\e^{i,s} \hookrightarrow \hM_{\e\v_\e^i}^{2r_\e}\setminus \hM_{\e\v_\e^i}^{5K\e^s}. 
\]
To simplify notations, we will treat these isometric embeddings as inclusions.

Next, apply Corollary~\ref{cor:cell_limsup_Quad} to $\hM_{\e\v_\e^i}^{r_\e/2}$ for each $i$; all the assumptions are satisfied, as $\log(1/r_\e)\ll \log(1/\e)$ and the inclusion map satisfies \eqref{eq:Z_eQ_e_recovery} and thus $|\IdRtwo - \PEps| \lesssim \e/\r$ inside $B_\e^i$, since $r_\e \le \rho_\e^2$.  This defines $\tf_\e$ around the cores of the dislocations.
At all other points, define it as the inclusion map.
By Corollary~\ref{cor:cell_limsup_Quad}, the resulting function glues nicely and is in $W^{1,\infty}(\MEps;\R^2)$.
Furthermore, by \eqref{eq:df_e_bound} and \eqref{eq:Z_eQ_e_recovery}, the bound \eqref{eq:tf_e_uni_bound} is  satisfied.
To obtain the bound \eqref{eq:tf_e_L2_bound}, note that $|d\tf_\e - \PEps| \lesssim \e/\r$ in $B_\e^i$ and that
\[
\int_{B_\e^i} \frac{\e^2}{\r^2} \,\VolumeEps \lesssim \e^2\log(1/\e).
\]
Thus,
\[
\begin{split}
\int_{\MEps} |d\tf_\e - \PEps|^2 \,\VolumeEps 
&\le \int_{\MEps} |\IdRtwo - \PEps|^2 \,\VolumeEps + \sum_{i=1}^{m_\e}\int_{B_\e^i} |d\tf_\e - \PEps|^2 \,\VolumeEps \\
&\lesssim h_\e^2 + m_\e \e^2 \log(1/\e) \lesssim h_\e^2.
\end{split}
\]
Finally, to show \eqref{eq:Quad_beta_e}, 
\[
\begin{split}
&\frac{1}{\log(1/\e)} \int_{B_\e^i\setminus B_\e^{i,s}} \bbW(\beta_\e)\,\VolumeEps \\
&\qquad \le \frac{1}{\log(1/\e)} \int_{\hM_{\e\v_\e^i}^{2r_\e}\setminus \hM_{\e\v_\e^i}^{5K\e^s}} \bbW(\beta_\e)\,\VolumeEps \\
&\qquad = \frac{1}{\log(1/\e)} \int_{\hM_{\e\v_\e^i}^{r_\e/2}\setminus \hM_{\e\v_\e^i}^{5K\e^s}} \bbW(\beta_\e)\,\VolumeEps
+\frac{1}{\log(1/\e)} \int_{\hM_{\e\v_\e^i}^{2r_\e}\setminus \hM_{\e\v_\e^i}^{r_\e/2}} \bbW(\beta_\e)\,\VolumeEps \\
&\qquad = \Ilin_0(\v_\e^i)\brk{s + O\brk{\e^{1-s}+ \frac{\log (1/r_\e)}{\log(1/\e)}}}
+\frac{1}{\log(1/\e)} \int_{\hM_{\e\v_\e^i}^{2r_\e}\setminus \hM_{\e\v_\e^i}^{r_\e/2}} \bbW(\beta_\e)\,\VolumeEps \\
\end{split}
\]
where in the transition to the last line we used \eqref{eq:Quad_df_e_optimal} which holds by Corollary~\ref{cor:cell_limsup_Quad}.
The proof is complete by noting that
\[
\begin{split}
&\frac{1}{\log(1/\e)} \int_{\hM_{\e\v_\e^i}^{2r_\e}\setminus \hM_{\e\v_\e^i}^{r_\e/2}} \bbW(\beta_\e)\,\VolumeEps \\
&\qquad = \frac{1}{\e^2 \log(1/\e)} \int_{\hM_{\e\v_\e^i}^{2r_\e}\setminus \hM_{\e\v_\e^i}^{r_\e/2}} \bbW( \PEps^{-1} - I)\,\VolumeEps \\
&\qquad \simeq \frac{1}{\e^2 \log(1/\e)} \int_{\hM_{\e\v_\e^i}^{2r_\e}\setminus \hM_{\e\v_\e^i}^{r_\e/2}} \frac{\e^2|\v_\e^i|^2}{\r^2}\,\VolumeEps \\
&\qquad \simeq \frac{|\v_\e^i|^2}{\log(1/\e)} \simeq  \Ilin_0(\v_\e^i)\frac{1}{\log(1/\e)},
\end{split}
\]
which is negligible compared to $\Ilin_0(\v_\e^i)\log(1/r_\e)/\log(1/\e)$.
Here, in the transition to the third line we use the pointwise bound on $\IdRtwo-\PEps$ and the fact that $\bbW$ is a quadratic form.
\end{proof}

\begin{lemma}
\label{lem:tf_e_conv}
Let $(\MEps,\PEps)$ 
and $\tf_\e$ be as in \lemref{lem:tf_e}.
Then, $\tf_\e \to (I,J_0)$ in the sense of Definition~\ref{def:conv_fe}, where $J_0=0$ in the subcritical case and 
$J_0 \in L^2\W^1(\W;\R^2)$ is the solution of the elliptic first-order differential system
\[
\begin{cases}
dJ_0 = \mu
\textand d^* J_0=0 & \text{in $\W$} \\
J_0(\frakn) = 0 & \text{on $\pl\W$}
\end{cases}
\]
in the critical and supercritical cases.
Furthermore, the $L^2$-convergence of $h_\e^{-1}(d\tf_\e - \PEps) \to J_0$ is strong on $\MEps\setminus \cup_i B_\e^i$.
\end{lemma}

\begin{proof}
We first note that \eqref{eq:tf_e_L2_bound} implies that $\tf_\e \to (I,J_0)$ modulo a subsequence for some $J_0 \in L^2\W^1(\W;\R^2)$.

We start with the the critical and supercritical cases.
We need to analyze $d\tf_\e$ in more detail in $\MEps\setminus \cup_i B_\e^i$.
In this region, $d\tf_\e = \IdRtwo$ and
\[
\IdRtwo - \PEps = \alpha_\e - \beta_\e + \gamma_\e, 
\]
where $\alpha_\e$, $\beta_\e$ and $\gamma_\e$ are defined in Section~\ref{sec:disloc_construction}, and can be considered as one-forms on $\W$.
The sections $\alpha_\e$ and $\beta_\e$ are supported on the balls of radius $a_\e\simeq \rho_\e$ around the points $p_\e^i\in \W$ (defined in \lemref{lem:approx_mu}), and the following the bounds hold in each $B_{a_\e}(p_\e^i)$
\[
|\alpha_\e|(x) \lesssim \frac{\e}{|x-p_\e^i|}, \qquad |\beta_\e|(x) \lesssim \frac{\e}{\rho_\e},
\]
from which follows that
\[
\begin{split}
\int_{\MEps \setminus \cup_i B_\e^i} |\alpha_\e|^2 \,\VolumeEps 
&\lesssim m_\e \int_{r_\e}^{\rho_\e} \e^2\frac{1}{r^2}\,r\,dr \\
&\lesssim n_\e \e^2 \log(\rho_\e/r_\e) \ll n_\e \e^2 \log(1/\e) \lesssim h_\e^2,
\end{split}
\]
and similarly for $\beta_\e$.
Thus, in order to show that $h_\e^{-1}(\tf_\e - \PEps)$ converges strongly in $L^2$ to $J_0$ in $\MEps\setminus \cup_i B_\e^i$ (from which the weak convergence on $\MEps$ also follows), it suffices to show that $h_\e^{-1} \gamma_\e \to J_0$ in $L^2(\W)$.
Recall that
\[
\begin{cases}
d\gamma_\e = \tmu_\e
\textand d^* \gamma_\e=0 & \text{in $\W$} \\
\gamma_\e(\frakn) = 0 & \text{on $\pl\W$},
\end{cases}
\]
where $\tmu_\e$ are as defined in \lemref{lem:approx_mu}.
Thus, from elliptic regularity and \eqref{eq:tmu_e_L_conv},
\[
\left\| \frac{1}{h_\e} \gamma_\e - J_0\right\|_{L^2} \lesssim \left\| \frac{1}{n_\e \e} \tmu_\e - \mu\right\|_{H^{-1}} \to 0,
\]
which completes the proof for the critical and supercritical cases.
For the subcritical case, the proof follows in the same way, only that in the last inequality
\[
\left\| \frac{1}{h_\e} \gamma_\e\right\|_{L^2} \lesssim \frac{1}{h_\e} \left\|  \tmu_\e\right \|_{H^{-1}},
\]
which tends to zero from \lemref{lem:approx_mu}.
\end{proof}

\begin{lemma}[Recovery sequence]\label{lem:rec_seq}
Let $\mu \in \calM(\W;\R^2)$ (in the critical and supercritical cases, $\mu \in \calM \cap H^{-1}$), and let $(\MEps, \PEps)$ be the sequence constructed in \lemref{lem:upper_bound_II}.
Let $U\in \SO(2)$ and $J \in L^2\W^1(\W;\R^2)$ satisfy $\curl J = 0$ (subcritical) or $\curl J = -\mu$ (critical or supercritical).
Let $J_0$ as defined in \lemref{lem:tf_e_conv}, and let $\psi\in W^{1,2}(\W)$ be such that $d\psi = J-  J_0$.
Then, for every sequence $\psi_\e \in W^{1,\infty}(\W;\R^2)$ such that $\psi_\e\to \psi$ in $W^{1,2}(\W)$ and $\|d\psi_\e\|_{L^\infty} \ll \e^{1/2}h_\e^{-1}$, the sequence of functions
\[
f_\e = U(\tf_\e + h_\e \psi_\e )
\]
converges to $(U,J)$, and
\[
\LimsupEps \calE_\e(f_\e,\PEps) \le \calE_0(J,\mu).
\]
\end{lemma}

\begin{proof}
We have
\[
\frac{U^T df_\e - \PEps}{h_\e} = \brk{\frac{d\tf_\e - \PEps}{h_\e}-J_0} + 
(d\psi_\e  + J_0).
\]
From  \lemref{lem:tf_e_conv} and Definition~\ref{def:conv_fe}, the first term on the right-hand side tends to zero weakly in $L^2$. From the $L^2$-convergence $d\psi_\e \to d\psi = J - J_0$, we obtain that
$f_\e \to (U, J)$.
Furthermore, \lemref{lem:tf_e_conv} implies that on $\MEps\setminus \cup_i B_\e^i$ the $L^2$-convergence is strong.

Similarly to the way we proceeded for the lower bound, we split the energy into
\[
\begin{split}
\calE_\e(f_\e,\PEps) &= \frac{1}{h_\e^2} \sum_{i=1}^{m_\e} \int_{B_\e^i} \calW(df_\e\circ \PEps^{-1})\,\VolumeEps + 
\frac{1}{h_\e^2} \int_{\MEps\setminus \cup_i B_\e^i} \calW(df_\e\circ \PEps^{-1})\,\VolumeEps \\
&\equiv \calE_\e^{\text{self}}(f_\e,\PEps) + \calE_\e^{\text{elastic}}(f_\e,\PEps),
\end{split}
\]
evaluating each part separately.
For the ``elastic" (far field) part, 
\[
\begin{split}
\calE_\e^{\text{elastic}}(f_\e,\PEps) &= \frac{1}{h_\e^2} \int_{\MEps\setminus \cup_i B_\e^i} 
\calW(U^T df_\e\circ \PEps^{-1})\,\VolumeEps \\
 &= \frac{1}{h_\e^2} \int_{\MEps\setminus \cup_i B_\e^i} 
 \calW(I+ (d\tf_\e\circ \PEps^{-1}-I)+h_\e d\psi_\e  \circ \PEps^{-1})\,\VolumeEps. \\
\end{split}
\]
Note that both $h_\e d\psi_\e  \circ \PEps^{-1}$ and $d\tf_\e\circ \PEps^{-1}-I$ tend to zero pointwise (the first from the assumption $\|d\psi_\e\|_{L^\infty} \ll \e^{1/2}h_\e^{-1}$, and the second from \eqref{eq:tf_e_uni_bound}, restricted to $\MEps\setminus \cup_i B_\e^i$), and the $L^2$ norms of both are $O(h_\e^2)$ (from the convergence $d\psi_\e\to d\psi$ and from \eqref{eq:tf_e_L2_bound}, respectively). Arguing as in \propref{prop:liminf1}, we linearize and obtain
\[
\begin{split}
\calE_\e^{\text{elastic}}(f_\e,\PEps) 
&= \int_{\MEps\setminus \cup_i B_\e^i} \bbW\brk{h_\e^{-1}\brk{d\tf_\e -\PEps}\PEps^{-1}+ d\psi_\e  \circ \PEps^{-1}}\,\VolumeEps + o(1). \\
&= \int_{\MEps\setminus \cup_i B_\e^i} \bbW\brk{h_\e^{-1}\brk{d\tf_\e -\PEps} + d\psi_\e}\,\VolumeEps + o(1). \\
&= \int_{\MEps\setminus \cup_i B_\e^i} \bbW\brk{\brk{h_\e^{-1}\brk{d\tf_\e -\PEps} - J_0} + d\psi_\e+J_0}\,\VolumeEps + o(1). \\
&= \int_{\MEps\setminus \cup_i B_\e^i} \bbW\brk{d\psi_\e+J_0}\,\VolumeEps + o(1). \\
\end{split}
\]
where in the transition to the second line we used the fact that $\PEps^{-1}$ and $\IdRtwo$ are uniformly close in $\MEps\setminus \cup_i B_\e^i$ (this is immediate from \eqref{eq:Z_eQ_e_recovery}), and in the transition to the last line we used the fact that $h_\e^{-1}(d\tf_\e -\PEps)$ converge strongly to $J_0$ on $\MEps\setminus \cup_i B_\e^i$ (\lemref{lem:tf_e_conv}).
Taking $\e\to 0$ and using the fact that $d\psi_\e \to J-J_0$ in $L^2$, and that the volume of $ \cup_i B_\e^i$ tends to zero, we obtain that
\beq\label{eq:calE_e_elastic_recovery}
\lim_{\e\to 0} \calE_\e^{\text{elastic}}(f_\e,\PEps) = \int_{\W} \bbW(J)\,\VolumeE = \calE_0^{\text{elastic}}(J)
\eeq
as needed.

We next evaluate the energy close to the cores of the dislocations. We split $B_\e^i$ further into $B_\e^{i,s}$ and $B_\e^i\setminus B_\e^{i,s}$, where $B_\e^{i,s}$ is defined in \eqref{eq:B_e_i_s}:
\[
\calE_\e^{\text{self}}(f_\e,\PEps) = \frac{1}{h_\e^2} \sum_{i=1}^{m_\e} \int_{B_\e^{i,s} } \calW(U^T df_\e\circ \PEps^{-1})\,\VolumeEps 
+ \frac{1}{h_\e^2} \sum_{i=1}^{m_\e} \int_{B_\e^i\setminus B_\e^{i,s}} \calW(U^T df_\e\circ \PEps^{-1})\,\VolumeEps.
\]
From the pointwise bounds \eqref{eq:tf_e_uni_bound} and $\|d\psi_\e\|_{L^\infty}\ll \e^{1/2}h_\e^{-1}$ we have in $B_\e^i$,
\beq
\label{eq:df_minus_Q_recovery}
|U^T df_\e - \PEps| = |d\tf_\e - \PEps| + h_\e |d\psi_\e | \lesssim \frac{\e}{\r} +  \e^{1/2},
\eeq
hence, by the upper bound in \eqref{eq:bound_calW}
\[
\calW(U^T df_\e\circ \PEps^{-1}) \lesssim \dist^2(U^T df_\e\circ \PEps^{-1}, \SO(2)) \le |U^T df_\e - \PEps|^2 \lesssim \frac{\e^2}{\r^2} + \e.
\]
Thus we have
\[
\begin{split}
\frac{1}{h_\e^2} \sum_{i=1}^{m_\e} \int_{B_\e^{i,s}} \calW(U^T df_\e\circ \PEps^{-1})\,\VolumeEps 
& \lesssim \frac{1}{h_\e^2} \sum_{i=1}^{m_\e} \int_{\e}^{\e^s} \brk{\frac{\e^2}{r^2} + \e} r\,dr \\
& \lesssim (1-s)\frac{n_\e \e^2 \log(1/\e)}{h_\e^2} + \frac{n_\e \e^{1+2s}}{h_\e^{2}}\\
& \lesssim 1-s.
\end{split}
\]
where we used the estimate $m_\e \lesssim n_\e$ and the fact that $s>1/2$ and thus $n_\e \e^{1+2s} \ll h_\e^2$.
Hence, when we eventually take $s\to 1$, the contribution will be negligible.

For the regions $B_\e^i\setminus B_\e^{i,s}$, 
the bound \eqref{eq:df_minus_Q_recovery} implies that $|U^T df_\e - \PEps| = o(1)$.
As before, we know that $h_\e^{-1}(U^T df_\e - \PEps)$ is $L^2$-bounded.
Thus, as for the far field part, we can linearize and obtain
\[
\begin{split}
&\frac{1}{h_\e^2} \sum_{i=1}^{m_\e} \int_{B_\e^i\setminus B_\e^{i,s}} \calW(U^T df_\e\circ \PEps^{-1})\,\VolumeEps \\
&\qquad = \sum_{i=1}^{m_\e} \int_{B_\e^i\setminus B_\e^{i,s}} \bbW(h_\e^{-1}(d\tf_\e\circ \PEps^{-1} - I) + d\psi_\e  \circ \PEps^{-1})\, \VolumeEps + o(1)
\end{split}
\]
Since the total volume of $\cup_i B_\e^i$ tends to zero as $\e\to 0$, and since $d\psi_\e$ converge strongly in $L^2(\W)$, we can omit the $d\psi_\e$ part:
\[
\begin{split}
\frac{1}{h_\e^2} \sum_{i=1}^{m_\e} \int_{B_\e^i\setminus B_\e^{i,s}} \calW(U^T df_\e\circ \PEps^{-1})\,\VolumeEps 
 = \frac{1}{h_\e^2} \sum_{i=1}^{m_\e} \int_{B_\e^i\setminus B_\e^{i,s}} \bbW(d\tf_\e\circ \PEps^{-1} - I)\, \VolumeEps + o(1).
\end{split}
\]
Using \eqref{eq:Quad_beta_e} and \eqref{eq:conv_of_self_energy}, we obtain
\[
\begin{split}
&\frac{1}{h_\e^2} \sum_{i=1}^{m_\e} \int_{B_\e^i\setminus B_\e^{i,s}} \calW(U^T df_\e\circ \PEps^{-1})\,\VolumeEps \\
&\qquad \le (s + o(1))\frac{\e^2 \log(1/\e)}{h_\e^2} \sum_{i=1}^{m_\e} \Ilin_0(\v_\e^i) + o(1) \\
&\qquad = (s+o(1))\frac{n_\e \e^2 \log(1/\e)}{h_\e^2} \brk{\int_{\W} \SEF\brk{\frac{d\mu}{d|\mu|}} d|\mu| + o(1)}+ o(1) \\
&\qquad = (s+o(1))\frac{n_\e \e^2 \log(1/\e)}{h_\e^2} \brk{\calE^{\text{self}}_0(\mu) + o(1)}+ o(1).
\end{split}
\]
In the supercritical regime $h_\e^2 \gg n_\e \e^2 \log(1/\e)$, hence the right-hand side tends to zero as $\e\to 0$.
In the critical and subcritical cases, it tends to $s\,\calE^{\text{self}}_0(\mu)$.
To conclude (in the critical and subcritical regimes---the supercritical regime is similar), we obtain
\[
\limsup_{\e\to 0} \calE_\e^{\text{self}}(f_\e,\PEps) \le s\, \calE^{\text{self}}_0(\mu) + C(1-s).
\]
Taking $s\to 1$ we obtain
\[
\limsup_{\e\to 0} \calE_\e^{\text{self}}(f_\e,\PEps) \le \calE^{\text{self}}_0(\mu).
\]
Combining this with the far field estimate \eqref{eq:calE_e_elastic_recovery} we obtain
\[
\limsup_{\e\to 0} \calE_\e(f_\e,\PEps) \le \calE^{\text{elastic}}_0(J) + \calE^{\text{self}}_0(\mu) = \calE_0(J,\mu),
\]
which completes the proof.
\end{proof}

\section{Rigidity revisited}
In this section we use the results of Sections~\ref{sec:compactness}--\ref{sec:GammaConv} to improve the result of Theorem~\ref{thm:rigidity}, when the limiting dislocation density $\mu$ is such that the energy $\calE_0(\cdot,\mu)$ has a positive minimum. 

Consider $\calE_0$ is as in \eqref{eq:limenergy} for $\bbW(J) = \frac{1}{4}|J+J^T|^2$, and define, in each regime, $F:\calM(\W;\R^2) \to [0,\infty]$ by
\[
F(\mu) = \inf \calE_0(\cdot,\mu).
\]

\begin{lemma}
$F(\mu)=0$ if and only if $\mu = 0$, or if the regime is supercritical, and $\mu\in H^{-1}(\W;\R^2)$ satisfies $\curl \mu = 0$.
\end{lemma}

\begin{proof}
In the subcritical and critical regimes, the self-energy term
\[
\int_\W \SEF\brk{\frac{d\mu}{d|\mu|}}\,d|\mu|
\]
in $\calE_0$, vanishes if and only if $\mu=0$, due to the positivity of $\SEF$ on the unit sphere.
In the supercritical regime, $\calE_0(J,\mu)$ vanishes if and only if $\mu\in H^{-1}(\W;\R^2)$ and there exists a $J\in L^2\W^1(\W;\R^2)$ satisfying $\curl J = -\mu$ and whose symmetric part vanishes a.e.
By the Saint-Venant condition \cite[Section 6.18]{Cia13}, this implies (and if $\W$ is simply-connected, equivalent to) $\curl \curl J = 0$, hence $\curl \mu = 0$.
\end{proof}

The following is an immediate corollary fo the lim-inf inequality Theorem~\ref{thm:liminf} and the compactness result Theorem~\ref{thm:compactness}:
\begin{corollary}
Let 
\[
(\MEps,\PEps)\xrightarrow[]{n_\e} (\W,\IdRtwo,\mu),
\]
with $\log n_\e, \log(1/\rho_\e) \ll \log(1/\e)$.
If $F(\mu)<\infty$, then there exists $\e_0>0$ such that for every $\e\in (0,\e_0)$,
\[
\inf_{f_\e\in W^{1,2}(\MEps;\R^2)} \int_{\MEps} \dist^2(df_\e,\SO(\gEps,\euc)) \, \VolumeEps \ge \frac{1}{2} h_\e^2 F(\mu).
\]
If $F(\mu)=\infty$ then
\[
\inf_{f_\e\in W^{1,2}(\MEps;\R^2)} \int_{\MEps} \dist^2(df_\e,\SO(\gEps,\euc)) \, \VolumeEps \gg h_\e^2.
\]
\end{corollary}

\begin{proof}
Denote the left-hand sides of both inequalities by $d_\e$.
If $d_\e \gg h_\e^2$, or if $F(\mu) = 0$ then there is nothing to prove.
Assume otherwise, then there exists a sequence $f_\e$ such that 
\[
\int_{\MEps} \dist^2(df_\e,\SO(\gEps,\euc))\, \VolumeEps \le \frac{3}{2}d_\e \lesssim h_\e^2.
\]
By moving to a subsequence, we can assume without a loss of generality that $d_\e/h_\e^2$ converges to a limit $\alpha$.
By Theorem~\ref{thm:compactness},  $f_\e \to (J,U)$ modulo a subsequence in the sense of \defref{def:conv_fe}.
By Theorem~\ref{thm:liminf}, 
\[
\frac{3}{2}\alpha \ge \liminf_{\e\to 0} \frac{1}{h_\e^2} \int_{\MEps} \dist^2(df_\e,\SO(\gEps,\euc))\, \VolumeEps \ge \calE_0(J,\mu) \ge F(\mu),
\]
and in particular $F(\mu)<\infty$. 
There exists $\e_0>0$ such that for every $\e\in (0,\e_0)$, $d_\e/h_\e^2 \ge \frac{3}{4}\alpha\ge \frac{1}{2}F(\mu)$. 
This completes the proof. 
\end{proof}

From this corollary, we immediately obtain a version of rigidity Theorem~\ref{thm:rigidity} without the additional $h_\e^2$ term:
\begin{theorem}\label{thm:rigidity2}
Let 
\[
(\MEps,\PEps) 
\xrightarrow[]{n_\e}
(\W,\IdRtwo,\mu),
\]
with $\log n_\e, \log(1/\rho_\e) \ll \log(1/\e)$, and $F(\mu)>0$.
Then, there exists a $\e_0>0$ such that for every $\e\in (0,\e_0)$ and for every $f_\e \in \HOne(\MEps;\R^2)$, there exists a matrix $U_\e\in \SO(2)$, such that
\[
\|df_\e - U_\e \PEps\|_{L^2(\MEps)}^2 \lesssim \brk{1+ \frac{1}{F(\mu)}} \int_{\MEps} \dist^2(df_\e,\SO(\gEps,\euc))\,\VolumeEps ,
\]
where the constant depends on $\W$ and on the bilipschitz constant of the embedding of $(\MEps,\PEps)$ into $(\W,\IdRtwo)$, but not on $\mu$ or on $n_\e$. 
\end{theorem}

\paragraph{Acknowledgments}
We are grateful to Manuel Friedrich, Adriana Garroni, Or Hershkovits and Dan Mangoubi for various discussions along the preparation of this paper. 
This project was initiated in the Oberwolfach meeting ``Material Theories" in July 2017; we hope that this fruitful series of meetings will resume soon.

\paragraph{Funding} RK was funded by ISF Grant 560/22 and CM was funded by ISF Grant 2304/24 and BSF Grant 2022076.

\footnotesize{
\bibliographystyle{amsalpha}
\bibliography{}
}

\end{document}